\documentclass{article}

\usepackage[T1]{fontenc}
\usepackage{fancyhdr}  
\usepackage[utf8]{inputenc} 
\usepackage{bbm}
\usepackage{amssymb}
\usepackage{amsbsy}
\usepackage{amsmath}
\usepackage{amsthm}
\usepackage{mathrsfs}
\usepackage{url}
\usepackage[overload]{empheq}

\newtheorem{thm}{Theorem}[section]

\newtheorem{lem}[thm]{Lemma}
\newtheorem{cor}[thm]{Corollary}
\newtheorem{conj}[thm]{Conjecture}
\newtheorem{defi}[thm]{Definition}

{\hfill $\square$\par\medbreak}

\title{Twins almost prime under a Elliott-Halberstam's conjecture}
\author{Nathalie Debouzy}

\begin{document}
\maketitle


\begin{abstract}
We improve Bombieri's asymptotic sieve to localise the variables. As a consequence, we prove, under a Elliott-Halberstam conjecture, that there exists an infinity of twins almost prime. Those are prime numbers $p$ such that for all $\varepsilon>0$, $p-2$ is either a prime number or can be written as $p_1 p_2$ where $p_1$ and $p_2$ are prime and $p_1<X^{\varepsilon}$, and we give the explicit asymptotic.
\end{abstract}


\def\sump{\mathop{\sum\mkern-3mu\hbox to 0pt{\raise3pt\hbox{${}^\prime$}\hss}\mkern3mu}}
\def\prode{\mathop{\prod\mkern-3.5mu\hbox to 0pt{\raise3pt\hbox{${}^*$}\hss}\mkern3mu}}
\def\sume{\mathop{\sum\mkern-3mu\hbox to 0pt{\raise3pt\hbox{${}^*$}\hss}\mkern3mu}} 
\def\sumun{\mathop{\sum\mkern-3mu\hbox to 0pt{\raise3pt\hbox{${}^1$}\hss}\mkern3mu}} 

\section*{Introduction}

The twin prime conjecture states that there are infinitely many primes $p$ such that $p + 2$ is also a prime, it has been one of the great open questions in number theory for many years.
Our result as many others in this field requires another conjecture about the distribution of prime numbers in arithmetic progressions. Peter D. T. A. Elliott and Heini Halberstam stated this kind of conjecture in 1968, in \cite{Ell}. The reader may also refer to \cite{Friedlander-Granville*92} and \cite{Friedlander-Granville*89}.

\begin{conj}
Given any $\delta>0$ and any $\theta$, there exists a constant $A_0(\delta, \theta)$ such that
\begin{equation}\label{introEH}
\sum_{\substack{d\le X^{1-\delta} \\ (d, 2) =1}}\max_{y\le X}\left| \psi(y;d,2) 
-\frac{y}{\varphi(d)}\right| \le A_0(\delta, \theta)\frac{X}{(6\log X)^{\theta}}.
\end{equation}
\end{conj}
This conjecture was proven for any $\delta> 1/2$ by Bombieri and Vinogradov and it is known that the conjecture fails at the endpoint $\delta=0$.

Under this conjecture, we prove the following theorem.
\begin{thm}
Under the conjecture (\ref{introEH}), given any $\beta\ge 0$ and $\gamma>\beta$, 
there exists $X_0$ such that for all $X\ge X_0$,
\begin{align}\label{reslt1}
\notag \sum_{n\le X} \Lambda(n)\Lambda(n+2) +
\frac{1}{\gamma-\beta} \sum_{n \le X}\Lambda(n+2)
\sum_{\substack{ d_1 d_2 =n \\ n^{\beta} \le d_1 \le n^{\gamma}}} \frac{\Lambda(d_1)\Lambda(d_2)}{\log n} \\
= 2 \mathfrak{S}_2 X\left( 1+o(1)\right)
\end{align}
where $\mathfrak{S}_2= \prod_{p\ge 3} \frac{p(p-2)}{(p-1)^2}$ is the twin prime numbers constant.
\end{thm}
In particular, if we take $\beta=0$ we obtain the following result.

\begin{cor}
Under the conjecture (\ref{introEH}), for any $\gamma$ where $0<\gamma<1$, there exists $X_0$ such that for all $X\ge X_0$,
\begin{align*}
\sum_{n\le X} \Lambda(n)\Lambda(n+2) +\frac{1}{\gamma} \sum_{n \le X}\Lambda(n+2)\sum_{\substack{ d_1 d_2 =n \\ d_1 \le n^{\gamma}}} \frac{\Lambda(d_1)\Lambda(d_2)}{\log n} \\
= 2 \mathfrak{S}_2 X\left( 1+o(1)\right).
\end{align*}
\end{cor}
The sums implied in those two previous theorems will take twin prime numbers and numbers $n$ such that $n=p_1 p_2$ where $p_1$ and $p_2$ are primes and $p_1$ is relatively small ($\le n^{\gamma}$) and $n+2$ is a prime.
Furthermore, as we suppose that $\sum_{n\le X} \Lambda(n)\Lambda(n+2) \sim \mathfrak{S}_2 X $ (as a consequence of Hardy-Littlewood conjecture), we could deduce that the two terms implied in each sum are the same weight.

These sums include primes powers which contribute to a neglibible part. We prove the same theorem with $\Lambda_1(n) = \log n$ if $n$ is a prime and $0$ otherwise in \ref{sanspuis}.


One of the best previous results on twin prime numbers was obtained by Chen \cite{Chen} in 1973, stating that he number of Chen's numbers (i.e. twin numbers such that $p$ is a prime and $p+2$ is also a prime or a product of two primes) lower of equal to $X$ is greater than $0.335 \Pi(X)$, where $\Pi(X) =2\mathfrak{S}_2 X/\log^2 X$, and Jie Wu \cite{Wu*08} and \cite{Wu*04} gave the best improved constant with $1.104 \Pi(X)$. 

Concerning our result, an expert in the field remarked (in private conversations) that Chen's method would give a similar result with $\gamma=1/10$ under a Elliott-Halberstam conjecture and he might even be able to go slightly lower but the abitrary $\gamma$ that we have in our work is definitely stronger. 
Indeed, the constants implied in our $o(1)$ are explicit and depend on $\delta$ of (\ref{introEH}). 

Finally, we note that a similar result could be obtained for $n$ and $n+k$ prime, and the proof is even more general as it can be used with $\Lambda(n)$ and $f(n)$ (instead of $\Lambda(n+2)$) where $f(n)$ is a function verifying a (quite long) list of properties. 

\vspace{1 cm}
All this work was part of the PhD thesis that I defended under the guidance of O. Ramaré, whom I greatly thank for his supervising work but also for his careful review of this article. 

The process starts with a generalized version of Enrico Bombieri's asymptotic sieve as we approximate
\begin{equation*}
\sum_{n\le X} \Lambda^{(\nu_1+\nu_2)}(n) f(n) + \sum_{n\le X} \Lambda^{(\nu_1)}\star \Lambda^{(\nu_2)}(n) f(n)
\end{equation*}
where $f$ is a function, $\nu_1$ and $\nu_2$ are two integers and $\Lambda^{(k)} =\frac{\Lambda(\log)^{k -1}}{(k-1)!}$, with 
\begin{equation*}
\sum_{n\le X} \Lambda^{(\nu_1+\nu_2)}(n) f_0(n) + \sum_{n\le X} \Lambda^{(\nu_1)}\star \Lambda^{(\nu_2)}(n) f_0(n)
\end{equation*}
where $f_0$ is a simpler function, considered as a "model" for $f$.
The functions $f$ and $f_0$ must satisfy precise properties that will be enumerated in \ref{hypo}. 
The proof will proceed with a level of distribution $X^{1-\delta}$ and we shall let $\delta$ tend to zero at the very end of the proof. In between, we shall use a preliminary sieving up to $z = X^\delta$ and use $D_0=X^{1-\delta}$.


The first part consists in generalizing sizeably Ramar\'e's work \cite{Ram}, which we correct (in minor points) and precise.
We then take the particular case $f(n) = \Lambda(n+2)$ (and $f_0(n) = 1$), under the conjecture (\ref{introEH}), and with the help of Bernstein polynomials towards a indicator-like function, our goal is reached in the second part of this article.



\section{A more flexible Bombieri Asymptotic Sieve}

\subsection{Context}
In many problems we need to calculate sums over integers which are relatively prime to a specific number. The $^*$ on sums and products means that they are taking only numbers which are relatively prime to a number  $\mathfrak{f}$. The sums and product with a $'$ are the case where $\mathfrak{f} =\prod_{p\le z} p$, hence they are restricted to integers that don't have any prime factor lower than $z$.

Let $f$ be a \emph{positive} function with some regularity:
\begin{equation}\label{approx2}
\sume_{n\le y/d}f(dn) = \sigma(d)F(y)+ r_d(f,y)
\end{equation}
where $\sigma$ is a multiplicative function, $F$ any function and $r_d$ is an error term.

We want to approach $f$ with a \emph{positive} but "simpler" function $f_0$ which would have the same kind of regularity.
\begin{equation*}
\sume_{n\le y/d}f_0(dn) = \sigma_0(d)F(y)+ r_d(f_0,y).
\end{equation*}
To manage the error term, we use the sum
\begin{equation}
R(f, D, r) = \sume_{d\le D} \tau_r(d) \max_{y\le X} \vert r_d(f,y)\vert
\end{equation}
where $\tau_r$ is the $r$-fold divisor function, and make hypothesis on the functions $f$ and $f_0$, see $(H_4)$ and $(H_9)$ in the next part. 
In the case of twin prime numbers, these hypothesis are validated with a Elliott-Halberstam-type conjecture, see $(H_{10})$.

To handle the difference from $f$ to $f_0$, we shall use
\begin{equation}
\bar{f} = f - \frac{V_{\sigma}(z)}{V_{\sigma_0}(z)} f_0
\end{equation}
where $V_{\sigma}(z) = \prod_{p\le z}(1 - \sigma(p))$. And we will use the following notations:
\begin{equation*}
r'_{d,z}(w,y)= \sump_{n\le y/d} w(n) \bar{f}(dn)
\end{equation*}
where $w$ is considered as a "weight" function and
\begin{equation*}
R'_z(w, \bar{f}, D, r) = \sump_{d\le D} \tau_r(d)\max_{y\le X}\vert r'_{d,z}(w,y)\vert.
\end{equation*}

At last, $f = \mathcal{O}^*(g)$ means $\vert f \vert \le g$. 


\subsection{Hypothesis and Theorem}\label{hypo}

Let $\nu_1$ et $\nu_2$ be two positive integers such that $0 <\nu_1\le \nu_2$ in the hypothesis below. We point out that in the first part of our work we kept the more general $\min(\nu_1,\nu_2)$ and $\max(\nu_1,\nu_2)$ (which could be easily replaced with $\nu_1$ and $\nu_2$ if the order is set). The number $\delta$ is our main parameter, it is meant to be small, and $z=X^{\delta}$. Our hypothesis are the following.

\begin{multline}\label{H1}\tag{$H_1$}
\left\lbrace
\begin{array}{cccc}
 \displaystyle \vert F(y)\vert \le \hat{F}(X) \qquad (y\le X),\\
 \displaystyle 0\le \sigma(p),\sigma_0(p)\le 1, \\
 \displaystyle \forall n \le X, f_0(n) \le B_0 \hat{F}(X)/X,\\
 \displaystyle \prode_{v\le p\le u} (1- \sigma(p))^{-1} + \prode_{v\le p\le u} (1- \sigma_0(p))^{-1} \le c \log u/\log v \qquad (2 \le v\le u)
\end{array}\right.
\end{multline}
where $c$ is a constant $\ge 2$. Also,
\begin{equation}\label{H2}\tag{$H_2$}
V_{\sigma_0}(z) \le \frac{c}{\log z}
\end{equation}
where we recall that $V_{\sigma}(z) = \prod_{p\le z}(1 - \sigma(p))$.

One of the fundamental hypothesis we make concerns positivity
\begin{equation}\label{H3}\tag{$H_3$}
f \ge 0 \qquad \text{et} \qquad f_0 \ge 0.
\end{equation}
Our hypothesis on remainder terms is:
\begin{equation}\label{H4}\tag{$H_4$}
R(f, D_0, 2\nu_2) + \frac{V_{\sigma}(z)}{V_{\sigma_0}(z)} R(f_0, D_0, 2\nu_2) \le A\hat{F}(X)/\log X
\end{equation}
where $R(f, D, r) = \sum_{d\le D} \tau_r(d)\max_{y\le X} \vert r_d(f,y)\vert$.
And (see \cite{Ram}) this hypothesis can be replaced with:
\begin{align}\label{H9}\tag{$H_9$}
\left\lbrace
\begin{array}{cc}
\displaystyle
\forall d \le D_0, \max_{y\le X}\left(\vert r_d(f,y)\vert, \vert r_d(f_0,y)\vert\right) \le C\hat{F}(X)/d\\
 \displaystyle R(f, D_0, 1) + \frac{V_{\sigma}(z)}{V_{\sigma_0}(z)} R(f_0, D_0, 1) \le A' \hat{F}(X)/\log^2 X
\end{array}\right.
\end{align}
where $A$ in ($H_4$) is then given by $A = \sqrt{C(2\log X)^{4\nu_2^2}A'} $.

Large values of $\sigma_0$ are controlled by:
\begin{equation}\label{H5}\tag{$H_5$}
\prode_{z\le p \le X} \sum_{a\ge 0} \sigma_0(p^a) \le c\frac{\log X}{\log z}.
\end{equation}

And  large values of $\sigma$ are controlled by our choice of $\Delta$ where
\begin{equation*}
\Delta = \sum_{\substack{p^a \le X\\p\ge z}}\vert \sigma(p^a) - \sigma_0(p^a)\vert.
\end{equation*}

When the preliminary sieve is removed, we will need the following hypothesis:
\begin{equation}\label{H6}\tag{$H_6$}
\max_{p\le z}(\sigma(p^h)p^h, \sigma_0(p^h)p^h) \le c,
\end{equation}
i.e. the values taken by $\sigma$ and $\sigma_0$ for prime number powers are relatively stable. We also suppose that:
\begin{equation}\label{H7}\tag{$H_7$}
V_{\sigma_0}(X^{1/4}) \le \frac{c}{\log X}.
\end{equation}
Finally, we need the following bounds:
\begin{equation}\label{H8}\tag{$H_8$}
\hat{F}(X) \ge z \sqrt{X} \delta^{-\nu_2-1}, \qquad \max(\Vert f\Vert_{\infty}, \Vert f_0\Vert_{\infty}) \le B.
\end{equation}

\begin{thm}\label{thm2}
Assuming the hypothesis above,
\begin{multline*}
\sum_{n\le X} \Lambda^{(\nu_1+\nu_2)}(n) f(n) + \sum_{n\le X} \Lambda^{(\nu_1)}\star \Lambda^{(\nu_2)}(n) f(n)\\
= \frac{V_{\sigma}(z)}{V_{\sigma_0}(z)}\left(\sum_{n\le X} \Lambda^{(\nu_1+\nu_2)}(n) f_0(n) + \sum_{n\le X} \Lambda^{(\nu_1)}\star \Lambda^{(\nu_2)}(n) f_0(n)\right)\\
 + (\rho +\theta) \frac{V_{\sigma}(z)}{V_{\sigma_0}(z)} \hat{F}(X) \frac{(\log X)^{\nu_1+\nu_2 - 1}}{(\nu_1+\nu_2 -1)!}
\end{multline*}
where
\begin{equation*}
\vert \rho \vert \le 2 \times (149\nu_2)^{\nu_1+\nu_2} 
 \times \biggl(Ac^2+\left( C_0(c) e \delta^{3\nu_2} +\Delta4^{\frac{1}{\delta}} \right) \left(\frac{c}{\delta}\right)^{2 \nu_2} \biggr)
\end{equation*}
and
\begin{multline*}
\vert\theta\vert \le 2\times(24\nu_2)^{\nu_2}B_0 \left(3\nu_2^2\log \frac{1}{\delta}\right)^{\nu_1+2\nu_2}\delta^{\nu_1} \\
+ 2\times\frac{B}{\log X} c^2(1+e^{c\Delta})\left(1+C_0(c)\right)(4\delta)^{\nu_2}\nu_2^{\nu_1} + A'.
\end{multline*}
\end{thm}
Notations were defined above and others can be found in \ref{autresnotations}.


\subsection{(Even more) generalized Selberg's identity}
Steinig and Diamond proved this identity below in the particular case $n_1=n_2=n$ and $n_1+n_2=2n$ in their lemma 5.2 in \cite{Diam}. For our work we need to desymmetrize it, obtaining the following result, which holds for any integers $n_1$ and $n_2$.
\begin{lem}\label{Diamond}
Let F be an invertible and infinitely differientiable function.
Here we classically denote $F^{(m)}$ the $m^{th}$ derivative of F for any positive integer $m$, and let $n_1$ and $n_2$, two positive integers. We have
\begin{multline*}
 \frac{1}{(n_1+n_2-1)!}\biggl( \frac{F'}{F}\biggr)^{(n_1+n_2-1)}+\frac{1}{(n_1-1)!}\biggl( \frac{F'}{F}\biggr)^{(n_1-1)}\frac{1}{(n_2-1)!}\biggl( \frac{F'}{F}\biggr)^{(n_2-1)}
\\
= \frac{P_{n_1,n_2}(F,F', ...., F^{(n_1+n_2)})}{F^{\max(n_1, n_2)}}
\end{multline*}
where $P_{n_1,n_2}$ is a polynomial with rational coefficients depending on $n_1$ and $n_2$ (independent of $F$).
\end{lem}
\begin{proof}
We define 
\begin{multline*}
\Sigma_{n_1,n_2}(F) = 
\frac{1}{(n_1+n_2-1)!}\biggl( \frac{F'}{F}\biggr)^{(n_1+n_2-1)}
\\
+\frac{1}{(n_1-1)!}\biggl( \frac{F'}{F}\biggr)^{(n_1-1)}\frac{1}{(n_2-1)!}\biggl( \frac{F'}{F}\biggr)^{(n_2-1)}.
\end{multline*}
For each positive integer $j$ there is a polynomial $R_j(Y_0, Y_1, ..., Y_j)$ with rational coefficients and independent of F such that
\begin{equation*}
\frac{1}{(j-1)!}\biggl( \frac{F'}{F}\biggr)^{(j-1)}=R_j(F,F', ... , F^{(j)})F^{-j}.
\end{equation*}
We then have
\begin{equation*}
\Sigma_{n_1,n_2}(F) = S_{n_1,n_2}(F,F', ... , F^{(n_1+n_2)})F^{-(n_1+n_2)}
\end{equation*}
where $S_{n_1,n_2} = R_{n_1+n_2} +R_{n_1}R_{n_2}$ is a polynomial, that we can express in terms of powers of the first variable:

\begin{equation}\label{S}
S_{n_1,n_2}(Y_0,Y_1, ... , Y_{n_1+n_2})=\sum_{j=0}^{n_1+n_2-1} Y_0^j Q_j^{n_1,n_2}(Y_1, ..., Y_{n_1+n_2})
\end{equation}
for suitable polynomials $Q_j^{n_1,n_2}$ which are also independent of F.

We shall show that $Q_j^{n_1,n_2}$ is identically zero for all $j$ when  $0\le j \le \min(n_1, n_2)-1$, i.e. for all real numbers $y_1, ..., y_{n_1+n_2}$, $Q_j^{n_1,n_2}(y_1, ..., y_{n_1+n_2})=0$.\\

We take that there exists a polynomial $g$ such that $g(0) \neq 0$ and let  $f(x) = x g(x)$.

It can easily be shown by induction that for all $j$,
\begin{equation*}
\frac{1}{(j-1)!}\biggl( \frac{f'}{f}\biggr)^{(j-1)}=\frac{(-1)^{j-1}}{x^j}+\frac{1}{(j-1)!}\biggl( \frac{g'}{g}\biggr)^{(j-1)}
\end{equation*}
and so
\begin{align*}
\Sigma_{n_1,n_2}(f) =& \frac{(-1)^{n_1+n_2-1}}{x^{n_1+n_2}}+ \frac{1}{(n_1+n_2-1)!}\biggl( \frac{g'}{g}\biggr)^{(n_1+n_2-1)} 
+\frac{(-1)^{n_1-1}}{x^{n_1}}\frac{(-1)^{n_2-1}}{x^{n_2}}\\
+&\frac{1}{(n_1-1)!}\biggl( \frac{g'}{g}\biggr)^{(n_1-1)}\frac{(-1)^{n_2-1}}{x^{n_2}} + \frac{1}{(n_2-1)!}\biggl( \frac{g'}{g}\biggr)^{(n_2-1)}\frac{(-1)^{n_1-1}}{x^{n_1}}\\
+&\frac{1}{(n_1-1)!}\biggl( \frac{g'}{g}\biggr)^{(n_1-1)}\frac{1}{(n_2-1)!}\biggl( \frac{g'}{g}\biggr)^{(n_2-1)}
\end{align*}

\begin{multline*}
\Sigma_{n_1,n_2}(f)=\frac{1}{(n_1-1)!}\biggl( \frac{g'}{g}\biggr)^{(n_1-1)}\frac{(-1)^{n_2-1}}{x^{n_2}}
\\
 + \frac{1}{(n_2-1)!}\biggl( \frac{g'}{g}\biggr)^{(n_2-1)}\frac{(-1)^{n_1-1}}{x^{n_1}} + \Sigma_{n_1,n_2}(g).
\end{multline*}

Since $g(x)$ is bounded away from $0$ as $x$ approaches the origin, we have for each integer $k \in \lbrace0; ... ; \min(n_1,n_2)-1\rbrace$ in a deleted neighborhood of zero
\begin{equation}\label{ineg}
\vert f(x)^{n_1+n_2-k}\Sigma_{n_1,n_2}(f(x))\vert \le \alpha_k \vert x\vert^{\min(n_1,n_2)-k}
\end{equation}
wher $\alpha_k$ is a constant depending on $g$ and the neighborhood.
If we take $k=0$, the left hand member of the above estimate is $\vert S_{n_1,n_2}(f,f', ... , f^{(n_1+n_2)})\vert$.

So, if we let $x$ approach $0$, by continuity we have:
\begin{equation*}
S_{n_1,n_2}(f(0),f'(0), ... , f^{(n_1+n_2)}(0))=0
\end{equation*}
And as $f(0)=0$, 
\begin{equation*}
\sum_{j=1}^{n_1+n_2-1} f(0)^j Q_j^{n_1,n_2}(f'(0), ...,f^{(n_1+n_2)}(0)) =0.
\end{equation*}

From \eqref{S} we have $Q_0^{n_1,n_2}(f'(0), ...,f^{(n_1+n_2)}(0)) =0$.
As this holds for any $f$, we can choose any real numbers $(y_1, ..., y_{n_1+n_2})=(f'(0), ...,f^{(n_1+n_2)}(0))$ except possibly if $y_1 = 0$. Since $Q_0^{n_1,n_2}$ is continuous, the restriction on $y_1$ may be removed and we have $Q_0^{n_1,n_2}\equiv 0$.

Now let us assume we have shown $Q_0^{n_1,n_2}\equiv 0$,...., $Q_{k-1}^{n_1,n_2}\equiv 0$ for some $k \in \lbrace0, ... ,\min(n_1,n_2)-1\rbrace$.\\
Then
\begin{equation*}
 f(x)^{n_1+n_2-k}\Sigma_{n_1,n_2}(f(x))= \sum_{j=k}^{n_1+n_2-1} f(x)^{j-k} Q_j^{n_1,n_2}(f'(x), ...,f^{(n_1+n_2)}(x)).
\end{equation*}
Using the equation \eqref{ineg}, again in a deleted neighborhood of zero, we see that $Q_k^{n_1,n_2}\equiv 0$.
Thus $Q_j^{n_1,n_2}\equiv 0$ for $0\le j \le \min(n_1, n_2)-1$. And thus
\begin{equation*}
 F^{n_1+n_2}\Sigma_{n_1,n_2}(F)= \sum_{j=\min(n_1,n_2)}^{n_1+n_2-1} F^{j} 
Q_j^{n_1,n_2}(F',\cdots,F^{(n_1+n_2)}).
\end{equation*}

Now we take
\begin{equation}
P_{n_1,n_2}(Y_0,Y_1, ...., Y_{n_1+n_2})=\sum_{j=\min(n_1,n_2)}^{n_1+n_2-1} Y_0^{j-\min(n_1,n_2)} Q_j^{n_1,n_2}(Y_1, ..., Y_{n_1+n_2})
\end{equation}
and so
\begin{equation*}
\Sigma_{n_1,n_2}(F) = \frac{P_{n_1,n_2}(F,F', ...., F^{(n_1+n_2)})}{F^{\max(n_1, n_2)}}.
\end{equation*}
\end{proof}

\subsection{Application}\label{autresnotations}
We use the previous lemma (\ref{Diamond}) with Riemann's function $\zeta$. Let $D(f,s)$ be the Dirichlet series of a function $f$. We remind the reader that $\zeta(s) = D(1,s)$, $D'(f,s)$ = $D(-f \log,s)$ and thus for all $k \ge 1$, $\zeta^{(k)}(s)= D((-\log)^k,s)$.
Furthermore, we know that
\begin{equation}
\frac{\zeta'}{\zeta}(s)= - D(\Lambda,s) \quad  \text{et} \quad \frac{1}{\zeta}(s)= D(\mu,s).
\end{equation}

As for two functions $f$ and $g$, we have $D(f,s) + D(g,s) = D(f+g,s)$ and  $D(f,s) \times D(g,s) = D(f\star g,s)$, the right hand term in our identity is
\begin{multline*}
\frac{P_{n_1,n_2}(\zeta,\zeta', ...., \zeta^{(n_1+n_2)})}{\zeta^{\max(n_1, n_2)}}\\
=P\left(D(1,s),D(\log,s), ... D((\log)^{n_1+n_2-1},s)\right)D(\mu \star...\star \mu,s)
\end{multline*}
where the multiple convolution product of $\mu$ has at most $\max(n_1, n_2)$ factors.

This terms is equal to $D(P^{\star}(\log,\log^2,... \log^{n_1+n_2-1}, \mu),s)$ 
where $P^{\star}$ 
is a polynom in convolution products in which the power on $\mu$ is at most $\max(n_1, n_2)$. We will give an explicit expression of this polynom later, we use \cite{Diam}  for our calculation and their formula for the $n-$\textit{1}-th derivative of  $F'/F$. See their Lemma 5.1.

The left hand term is equal to
\begin{multline}
\frac{D(-\Lambda(-\log)^{n_1+n_2-1},s)}{(n_1+n_2-1)!} + \frac{D(-\Lambda(-\log)^{n_1-1},s)}{(n_1-1)!}\frac{D(-\Lambda(-\log)^{n_2-1},s)}{(n_2-1)!}\\
= D(\frac{-\Lambda(-\log)^{n_1+n_2-1}}{(n_1+n_2-1)!} + \frac{\Lambda(-\log)^{n_1-1}}{(n_1-1)!}\star \frac{\Lambda(-\log)^{n_2-1}}{(n_2-1)!},s).
\end{multline}
We can notice that
\begin{multline*}
\frac{-\Lambda(-\log)^{n_1+n_2-1}}{(n_1+n_2-1)!} + \frac{\Lambda(-\log)^{n_1-1}}{(n_1-1)!}\star \frac{\Lambda(-\log)^{n_2-1}}{(n_2-1)!} \\
= (-1)^{n_1+n_2} \left( \frac{\Lambda(\log)^{n_1+n_2 -1}}{(n_1+n_2-1)!}+\frac{\Lambda(\log)^{n_1 -1}}{(n_1-1)!}\star \frac{\Lambda(\log)^{n_2 -1}}{(n_2-1)!}\right).
\end{multline*}
Identifying, we deduce the fundamental following formula:

Let $\nu_1$ and $\nu_2$ be two integers $\ge 1$,
\begin{equation}
 \frac{\Lambda(\log)^{\nu_1+\nu_2 -1}}{(\nu_1+\nu_2-1)!}+\frac{\Lambda(\log)^{\nu_1 -1}}{(\nu_1-1)!}\star \frac{\Lambda(\log)^{\nu_2 -1}}{(\nu_2-1)!} = (-1)^{\nu_1+\nu_2} \left(\Sigma_1 + \Sigma_2 \right)
\end{equation}
where
\begin{equation*}
\Sigma_1 = \sum_{\pmb{k}}(\nu_1+\nu_2)\frac{(-1)^{k-1}(k-1)!}{k_1!k_2!...k_{\nu_1+\nu_2}!}\mu_k\star \prod_{j=1}^{\nu_1+\nu_2}\biggl( \frac{( \log)^j}{j!}\biggr)^{k_j}
\end{equation*}
where $\mu_k = \mu \star \mu \star ....\star \mu$ ($k$ times) and  $\pmb{k}$ is in $\mathbb{K}$, the set of all $(\nu-1+\nu_2)$-tuples $(k_1, k_2, ...,k_{\nu_1+\nu_2})$ de $\nu_1 + \nu_2$ satisfying
\begin{eqnarray}
\left\lbrace
\begin{array}{cc}
k_1+2 k_2+... +(\nu_1+\nu_2) k_{\nu_1+\nu_2} = \nu_1+\nu_2\\
k_1+... +k_{\nu_1+\nu_2} = k \le \max(\nu_1,\nu_2)\\
\end{array}\right.
\end{eqnarray}
and
\begin{equation*}
\Sigma_2 = \sum_{\pmb{k'},\pmb{k''}}\nu_1 \nu_2\frac{(-1)^{k'+k''}(k'-1)!(k''-1)!}{k'_1!...k'_{\nu_1}! k''_1!...k''_{\nu_2}!}\mu_{k'+k''}\star \prod_{j=1}^{\max(\nu_1,\nu_2)}\biggl( \frac{(\log)^j}{j!}\biggr)^{k'_j+k''_j}
\end{equation*}
where $\pmb{k'}$ is a $\nu_1$-tuple of integers  $(k'_1, k'_2, ...,k'_{\nu_1})$ and $\pmb{k''}$ is a $\nu_2$-tuple of integers $(k''_1, k''_2, ...,k''_{\nu_2})$ such that:
\begin{eqnarray}
\left\lbrace
\begin{array}{ccc}
k'_1+2 k'_2+... +\nu_1 k'_{\nu_1} = \nu_1\\
k''_1+2 k''_2+... +\nu_2 k''_{\nu_2} = \nu_2\\
k'_1+... +k'_{\nu_1} + k''_1+... +k''_{\nu_2} = k'+k'' \le \max(\nu_1, \nu_2).\\
\end{array}\right.
\end{eqnarray}
We define by $\mathbb{H}$ the set of $(\nu_1+\nu_2)$-tuples $\pmb{h}=(\pmb{k'},\pmb{k''})$.

We use the following convention: if $\nu_1 = \max(\nu_1, \nu_2)$ then $k''_j = 0$ for all  $j\ge \nu_2$ and if $\nu_2 = \max(\nu_1, \nu_2)$ then $k'_j = 0$ for all $j\ge \nu_1$.

The key here is that $\mu_k=0$ for all $k > \max(\nu_1, \nu_2)$ because of the previous lemma.

Note that under this last convention we can write:
\begin{eqnarray*}
\prod_{j=1}^{\max(\nu_1,\nu_2)}\biggl( \frac{(\log)^j}{j!}\biggr)^{k'_j+k''_j} =\prod_{j=1}^{\nu_1}\biggl( \frac{(\log)^j}{j!}\biggr)^{k'_j}\prod_{j=1}^{\nu_2}\biggl( \frac{(\log)^j}{j!}\biggr)^{k''_j} .
\end{eqnarray*}

\subsubsection*{Notations}
Let us define the functions $c$ and $w$ over $\mathbb{K}$ by
\begin{multline*}
\left\lbrace
\begin{array}{cc}
 \displaystyle
 c(\pmb{k}) = \frac{(-1)^{k-1}(k-1)!}{k_1!....k_{\nu_1+\nu_2}!}\\
 \displaystyle 
 w(\pmb{k}) = \frac{1}{1!^{k_1}}\frac{1}{2!^{k_2}}...\frac{1}{(\nu_1+\nu_2)!^{k_{\nu_1+\nu_2}}}\\
\end{array}\right.
\end{multline*}
where $k$ denotes the length of $\pmb{k}$ in $\mathbb{K}$, 
and the functions $\bar{c}$ and $\bar{w}$ over $\mathbb{H}$ by
\begin{multline*}
\left\lbrace
\begin{array}{cc}
 \displaystyle \bar{c}(\pmb{h}) = \frac{(-1)^{k'-1}(k'-1)!}{k'_1!....k'_{\nu_1}!}\frac{(-1)^{k''-1}(k''-1)!}{k''_1!....k''_{\nu_2}!}\\
  \displaystyle\bar{w}(\pmb{h}) = \frac{1}{1!^{k'_1}}\frac{1}{2!^{k'_2}}...\frac{1}{\nu_1!^{k'_{\nu_1}}}\frac{1}{1!^{k''_1}}\frac{1}{2!^{k''_2}}...\frac{1}{\nu_2!^{k''_{\nu_2}}}.\\
\end{array}\right.
\end{multline*}
where $h$ denotes the length of  $\pmb{h}=(\pmb{k'},\pmb{k''})$ in $\mathbb{H}$ .
For all  $m\ge 1 $, and where $k_1$, $k_2$,... $k_m$ are integers we define
\begin{equation}
L(k_1, ..., k_m) = \log \star ...\star \log \star \log^2 \star ... \star \log^2 \star ... \star \log^m \star ... \star \log^m
\end{equation}
where each $\log^i$ is in a $k_i$ convolution product of itself.

Hence with $\Lambda^{(k)} =\frac{\Lambda(\log)^{k -1}}{(k-1)!}$,  we have
\begin{multline*}
\Lambda^{(\nu_1+\nu_2)} + \Lambda^{(\nu_1)} \star\Lambda^{(\nu_2)}\\
=(\nu_1 +\nu_2)\sum_{\pmb{k} \in \mathbbm{K}}c(\pmb{k})w(\pmb{k})\mu_{k} \star L(\pmb{k}) + \nu_1 \nu_2 \sum_{\textbf{h} \in \mathbbm{H}}\bar{c}(\textbf{h})\bar{w}(\textbf{h})\mu_h \star L(\textbf{k'})\star L(\textbf{k''})
\end{multline*}
and we will use the following short form
\begin{equation}
\Lambda^{(\nu_1+\nu_2)} + \Lambda^{(\nu_1)} \star\Lambda^{(\nu_2)} = \sum_{\pmb{\ell}} a(\pmb{\ell})\mu_{\ell} \star L(\pmb{\ell}).
\end{equation}

\subsection{Preliminary estimates}

\subsubsection{A few bounds}

\begin{lem}\label{Fmn}
  For fixed integers $k_0$, $\nu_1$ and $\nu_2$,
  \begin{equation*}
  \sum_{\substack{\pmb{h} \in \mathbbm{H} \\ h = k_0}} \vert \bar{c}(\textbf{h})\vert  =\frac{(k'_0 - 1)!(k''_0-1)!}{\nu_1!\nu_2!} \binom{\nu_1+\nu_2}{k_0} \binom{\nu_1+\nu_2-1}{k_0-1} \tag{a}
  \end{equation*}
and  
   \begin{equation*}
\sum_{\substack{\pmb{k} \in \mathbbm{K} \\ k = k_0}} \vert c(\pmb{k})\vert = \frac{1}{k_0 (\nu_1+\nu_2 - k_0)!}\binom{\nu_1+\nu_2-1}{k_0 - 1} \le \frac{1}{k_0}\binom{\nu_1+\nu_2-1}{k_0 - 1} \tag{b}.
  \end{equation*}
\end{lem}

\begin{proof}
The equality in (b) comes from computing the coefficients of  $X^k Y^{\nu_1+\nu_2}$ in $(1 + XY + XY^2 + ... +XY^{\nu_1+\nu_2})^{\nu_1+\nu_2} $ in two different ways (cf. \cite{Diam} between (5.10) and (5.11))
\begin{multline*}
(1 + XY + XY^2 + ... +XY^{\nu_1+\nu_2})^{\nu_1+\nu_2}\\ 
= \sum_{\substack{0 \le k_1, k_2, ...,k_{\nu_1+\nu_2}\le \nu_1+\nu_2 \\ 0 \le k_1+ k_2+ ...+k_{\nu_1+\nu_2}\le \nu_1+\nu_2}} \binom{\nu_1+\nu_2}{k_1, ...,k_{\nu_1+\nu_2}}X^{ k_1+ ...+k_{\nu_1+\nu_2}}Y^{k_1+2 k_2+ ...+(\nu_1+\nu_2)k_{\nu_1+\nu_2}}.
\end{multline*}
While
\begin{equation*}
 (1 + XY + XY^2 + ... +XY^{\nu_1+\nu_2})^{\nu_1+\nu_2}
\end{equation*}
is equal to
\begin{align*}
 &(1 + X(Y + Y^2 + ... +Y^{\nu_1+\nu_2}))^{\nu_1+\nu_2}\\
 &= \biggl( 1 + \frac{XY}{1 - Y} \biggr)^{\nu_1+\nu_2} + \mathcal{O}(XY^{\nu_1+\nu_2+1})\\
&= \sum_{k=0}^{\nu_1+\nu_2} \binom{\nu_1+\nu_2}{k}(XY)^k \sum_{j \ge 0} \binom{k-1+j}{k-1}Y^j 
+ \mathcal{O}(XY^{\nu_1+\nu_2+1}).
\end{align*}
We understand that $\mathcal{O}(XY^{\nu_1+\nu_2+1})$ denotes a power series in $X$ and $Y$ multiplied by $XY^{\nu_1+\nu_2+1}$.

Computing the term $X^k Y^{\nu_1+\nu_2}$ in both identities, which is $j=\nu_1+\nu_2-k$ in the last, we find that for all $k$
\begin{equation}
\sum_{\substack{k_1+ k_2+ ...+k_{\nu_1+\nu_2}=k \\  k_1+ 2 k_2+ ...+(\nu_1+\nu_2) k_{\nu_1+\nu_2}= \nu_1+\nu_2}} \binom{\nu_1+\nu_2}{k_1, ...,k_{\nu_1+\nu_2}} =\binom{\nu_1+\nu_2}{k} \binom{\nu_1+\nu_2-1}{k-1}
\end{equation}
hence
\begin{align}
\notag \sum_{\substack{\pmb{k} \in \mathbbm{K} \\ k = k_0}} \vert c(\pmb{k})\vert &= \sum_{\substack{\pmb{k} \in \mathbbm{K} \\ k = k_0}}\binom{k_0 - 1}{k_1,... k_{\nu_1+\nu_2}} \\
\notag &= \frac{(k_0 - 1)!}{(\nu_1+\nu_2)!}\sum_{\substack{k_1+ k_2+ ...+k_{\nu_1+\nu_2}=k \\  k_1+ 2 k_2+ ...+(\nu_1+\nu_2) k_{\nu_1+\nu_2}= \nu_1+\nu_2}} \binom{\nu_1+\nu_2}{k_1, ...,k_{\nu_1+\nu_2}}\\
\notag &= \frac{(k_0 - 1)!}{(\nu_1+\nu_2)!} \binom{\nu_1+\nu_2}{k_0} \binom{\nu_1+\nu_2-1}{k_0-1}\\
&= \frac{1}{k_0 (\nu_1+\nu_2 - k_0)!}\binom{\nu_1+\nu_2-1}{k_0 - 1}.
\end{align}

The inequality is obvious. We proceed in a similar way to establish the identity (a) using four variables

\begin{multline*}
(1 + X_1Y_1 + X_1Y_1^2 + ... X_1Y_1^{\nu_1})^{\nu_1}(1 + X_2Y_2 + X_2Y_2^2 + ... X_2Y_2^{\nu_2})^{\nu_2}\\ 
= \sum_{\substack{0 \le k'_1, k'_2, ...,k'_{\nu_1}\le \nu_1 \\ 0 \le k'_1+ k'_2+ ...+k'_{\nu_1}\le \nu_1}} \binom{\nu_1}{k'_1, ...,k'_{\nu_1}}X_1^{ k'_1+ ...+k'_{\nu_1}}Y_1^{k'_1+2 k'_2+ ...+\nu_1k'_{\nu_1}}\\
\times \sum_{\substack{0 \le k''_1, k''_2, ...,k''_{\nu_2}\le \nu_2 \\ 0 \le k''_1+ k''_2+ ...+k''_{\nu_2}\le \nu_2}} \binom{\nu_2}{k''_1, ...,k''_{\nu_2}}X_2^{ k''_1+ ...+k''_{\nu_2}}Y_2^{k''_1+2 k''_2+ ...+\nu_2k''_{\nu_2}}.
\end{multline*}
While, as previously,
\begin{align}
\notag & \left(\biggl( 1 + \frac{X_1Y_1}{1 - Y_1} \biggr)^{\nu_1} + \mathcal{O}(X_1Y^{\nu_1+1})\right)\left(\biggl( 1 + \frac{X_2Y_2}{1 - Y_2} \biggr)^{\nu_2} + \mathcal{O}(X_2Y^{\nu_2+1})\right)\\
&= \biggl( 1 + \frac{X_1Y_1}{1 - Y_1} \biggr)^{\nu_1}\biggl( 1 + \frac{X_2Y_2}{1 - Y_2} \biggr)^{\nu_2}+\mathcal{O}(X_1Y^{\nu_1+1}+X_2Y^{\nu_2+1}).\\
\end{align}
Again, we identify the term in $X_1^{k'}Y_1^{\nu_1}X_2^{k''} Y_2^{\nu_2}$ for a fixed $k'+k''=k$ and we finally take $X_1=X_2=X$ and $Y_1=Y_2=Y$:
\begin{align*}
\sum_{\substack{k'_1+ k'_2+ ...+k'_{\nu_1}=k' \\  k'_1+ 2 k'_2+ ...+\nu_1 k_{\nu_1}= \nu_1} }\binom{\nu_1}{k'_1, ...,k'_{\nu_1}} \sum_{\substack{k''_1+ k''_2+ ...+k''_{\nu_2}=k'' \\  k''_1+ 2 k''_2+ ...+\nu_2 k_{\nu_2}= \nu_2} }\binom{\nu_2}{k''_1,...,k''_{\nu_2}}=\binom{\nu_1+\nu_2}{k} \binom{\nu_1+\nu_2-1}{k-1}.
\end{align*}
but
\begin{align*}
\sum_{\substack{\pmb{h} \in \mathbbm{H} \\ h = k_0}} \vert \bar{c}(\pmb{h})\vert &= \sum_{\substack{\pmb{h} \in \mathbbm{H} \\ h = k_0}}\binom{k'_0 - 1}{k'_1,... k'_{\nu_1}}\binom{k''_0 - 1}{k''_1,... k''_{\nu_2}} \\
&= \frac{(k'_0 - 1)!(k''_0-1)!}{\nu_1!\nu_2!}\sum_{\substack{k'_1+ k'_2+ ...+k'_{\nu_1}=k'_0 \\  k'_1+ 2 k'_2+ ...+\nu_1 k'_{\nu_1}= \nu_1 \\ k''_1+ k''_2+ ...+k''_{\nu_2}=k''_0 \\  k''_1+ 2 k''_2+ ...+\nu_2 k''_{\nu_2}= \nu_2 }} \binom{\nu_1}{k'_1, ...,k'_{\nu_1}}\binom{\nu_2}{k''_1, ...,k''_{\nu_2}}\\
&=  \frac{(k'_0 - 1)!(k''_0-1)!}{\nu_1!\nu_2!} \binom{\nu_1+\nu_2}{k_0} \binom{\nu_1+\nu_2-1}{k_0-1}.\\
\end{align*}

\end{proof}
Following this result
\begin{equation*}
 \frac{(k'_0 - 1)!(k''_0-1)!}{\nu_1!\nu_2!} \le  \frac{(k_0 - 1)!}{\nu_1!\nu_2!}=\frac{(\nu_1+\nu_2)!(k_0-1)!(\nu_1+\nu_2-k_0)!}{(\nu_1+\nu_2-k_0)!\nu_1!\nu_2!(\nu_1+\nu_2)!}
\end{equation*}
bounding $\frac{(\nu_1+\nu_2)!}{(\nu_1+\nu_2-k_0)!}$ with $(\nu_1+\nu_2)(\nu_1+\nu_2-1)^{k_0-1}$, we have
\begin{align*}
\sum_{k_0=1}^{\max(\nu_1,\nu_2)}\sum_{\substack{\pmb{h} \in \mathbbm{H} \\ h = k_0}} \vert \bar{c}(\pmb{h})\vert &\le \sum_{k_0=1}^{\max(\nu_1,\nu_2)} \frac{\nu_1+\nu_2}{k_0}\frac{(\nu_1+\nu_2-1)^{k_0-1}}{\nu_1!\nu_2!} \binom{\nu_1+\nu_2-1}{k_0-1}\\
&\le \frac{\nu_1+\nu_2}{\nu_1!\nu_2!}\sum_{k_0=0}^{\nu_1+\nu_2-1}\binom{\nu_1+\nu_2-1}{k_0-1}(\nu_1+\nu_2-1)^{k_0-1} \\
&\le  \frac{\nu_1+\nu_2}{\nu_1!\nu_2!}(\nu_1+\nu_2)^{\nu_1+\nu_2-1} = \frac{(\nu_1+\nu_2)^{\nu_1+\nu_2}}{\nu_1!\nu_2!}. 
\end{align*}
Using the identity $n! \ge \left(\frac{n}{e}\right)^n$ we now have
\begin{equation*}
 \frac{(\nu_1+\nu_2)^{\nu_1+\nu_2}}{\nu_1!\nu_2!} \le \left(\frac{e}{\nu_1}\right)^{\nu_1} \left(\frac{e}{\nu_2}\right)^{\nu_2}(2\max(\nu_1,\nu_2))^{\nu_1+\nu_2} = (2e)^{\nu_1+\nu_2}\left(\frac{\max(\nu_1,\nu_2)}{\min(\nu_1,\nu_2)}\right)^{\min(\nu_1,\nu_2)}.
 \end{equation*} 
And so
\begin{equation}\label{cbar}
\sum_{k_0=1}^{\max(\nu_1,\nu_2)}\sum_{\substack{\pmb{h} \in \mathbbm{H} \\ h = k_0}} \vert \bar{c}(\pmb{h})\vert \le  6^{\nu_1+\nu_2}\left(\frac{\max(\nu_1,\nu_2)}{\min(\nu_1,\nu_2)}\right)^{\min(\nu_1,\nu_2)}.
\end{equation}
This bound will be used below to establish Lemmas \ref{majoc1} and \ref{majoc2}.

\begin{lem}
  For all \textbf{k} in $\mathbbm{K}$, $w(\pmb{k}) \le 2^{-\min(\nu_1, \nu_2)}$ and \textbf{h} in $\mathbbm{H}$, $\bar{w}(\pmb{h}) \le 2^{-\min(\nu_1, \nu_2)}$.
  
\end{lem}

\begin{proof}
As $w(\pmb{k}) = e^{\log(w(\pmb{k}))}$ we find a lower bound for $S= - \log(w(\pmb{k}))=\sum_{j=1}^{\nu_1 + \nu_2}k_j \sum_{\ell=1}^{j}\log \ell $.
\begin{align*}
S &\geq  \log 2 \sum_{j=1}^{\nu_1 + \nu_2}k_j \sum_{\ell=2}^{j}1 \\
&\geq   \log 2 \sum_{j=1}^{\nu_1 + \nu_2}j k_j - \log 2 \sum_{j=1}^{\nu_1 + \nu_2}k_j\\
&\geq  (\nu_1+\nu_2)\log 2 - \max(\nu_1, \nu_2)\log 2 \geq  \min(\nu_1, \nu_2)\log 2.
\end{align*}
Proceeding the same way with  $S= - \log(\bar{w}(\pmb{h}))$, we find that
\begin{align*}
S &\geq  \log 2 \sum_{j=1}^{\nu_1}k'_j \sum_{\ell=2}^{j}1 + \log 2 \sum_{j=1}^{\nu_2}k''_j \sum_{\ell=2}^{j}1\\
&\geq   \log 2 (\sum_{j=1}^{\nu_1}j k'_j+\sum_{j=1}^{\nu_2}j k''_j )- \log 2 (\sum_{j=1}^{\nu_1}k'_j+\sum_{j=1}^{\nu_2}k''_j)\\
&\geq  (\nu_1+\nu_2)\log 2 - \max(\nu_1, \nu_2)\log 2 \geq  \min(\nu_1, \nu_2)\log 2.
\end{align*}
\end{proof}

\begin{lem}\label{majoc2}
\begin{equation}
(\nu_1 +\nu_2)\sum_{\pmb{k} \in \mathbbm{K}}\vert c(\pmb{k})\vert + \nu_1 \nu_2 \sum_{\textbf{h} \in \mathbbm{H}}\vert \bar{c}(\textbf{h})\vert \le 8^{\nu_1+\nu_2}\frac{(\max(\nu_1,\nu_2))^{\min(\nu_1,\nu_2)+1}}{(\min(\nu_1,\nu_2))^{\min(\nu_1,\nu_2)-1}}
\end{equation}

\end{lem}
\begin{proof}
From Lemma \ref{Fmn},
\begin{multline*}
(\nu_1 +\nu_2)\sum_{\pmb{k} \in \mathbbm{K}}\vert c(\pmb{k})\vert 
\le (\nu_1+\nu_2) \sum_{k_0=1}^{\max(\nu_1,\nu_2)} \binom{\nu_1+\nu_2-1}{k_0 - 1}\frac{1}{k_0}
\end{multline*}
but
\begin{equation*}
\binom{\nu_1+\nu_2-1}{k_0 - 1}\frac{1}{k_0} = \frac{1}{\nu_1+\nu_2}\binom{\nu_1+\nu_2}{k_0 }
\end{equation*}
and
\begin{equation*}
 \sum_{k_0=1}^{\max(\nu_1,\nu_2)}\binom{\nu_1+\nu_2}{k_0 } \le  \sum_{k_0=0}^{\nu_1+\nu_2}\binom{\nu_1+\nu_2}{k_0 } = 2^{\nu_1+\nu_2}
\end{equation*}
so, using \eqref{cbar}, we have
\begin{align*}
(\nu_1 +\nu_2)\sum_{\pmb{k} \in \mathbbm{K}}\vert c(\pmb{k})\vert + \nu_1 \nu_2 \sum_{\textbf{h} \in \mathbbm{H}}\vert \bar{c}(\textbf{h})\vert
&\le  2^{\nu_1+ \nu_2} +6^{\nu_1+\nu_2}\nu_1\nu_2\left(\frac{\max(\nu_1,\nu_2)}{\min(\nu_1,\nu_2)}\right)^{\min(\nu_1,\nu_2)}\\
&\le 8^{\nu_1+\nu_2}\frac{(\max(\nu_1,\nu_2))^{\min(\nu_1,\nu_2)+1}}{(\min(\nu_1,\nu_2))^{\min(\nu_1,\nu_2)-1}}.
\end{align*}
\end{proof}

\begin{lem}\label{majoc1}
\begin{equation}
(\nu_1 +\nu_2)\sum_{\pmb{k} \in \mathbbm{K}}\vert c(\pmb{k})w(\pmb{k})\vert + \nu_1 \nu_2 \sum_{\textbf{h} \in \mathbbm{H}}\vert \bar{c}(\textbf{h})\bar{w}(\textbf{h})\vert \le 8^{\nu_1+\nu_2}\frac{(\max(\nu_1,\nu_2))^{\min(\nu_1,\nu_2)+2}}{(2\min(\nu_1,\nu_2))^{\min(\nu_1,\nu_2)}}
\end{equation}
\end{lem}

\begin{proof}
It is easily deduced from the previous lemmas with $\min(\nu_1,\nu_2)\le \max(\nu_1,\nu_2)$.

\end{proof}

We correct Lemma 16 from \cite{Ram} using those two following preliminary lemmas
\begin{lem}\label{lemme14}
For $300 \le a \le b$ we have
\begin{equation*}
\prod_{a<p\le b} \left(1 - \frac{1}{p}\right)^{-1} \le 1.04 \frac{\log b}{\log a}.
\end{equation*}
\end{lem}
See Lemma 14 from \cite{Ram}.

\begin{lem}\label{lemme15}
Let $H$ be a positive function satisfying
\begin{align*}
\sum_{p\le y} H(p) \log p \le \alpha y \qquad \text{pour } y \ge 0\\
\sum_p\sum_{\alpha\ge 2} H(p^{\alpha})p^{-\alpha}\log(p^{\alpha}) \le \beta
\end{align*}
then for $x>1$ we have
\begin{equation*}
\sum_{n\le X} H(n) \le (\alpha+\beta+1)\frac{x}{\log x}\sum_{n\le x}H(n)/n.
\end{equation*}
\end{lem}
See Lemma 15 from \cite{Ram}.

Lemma 16 will be amended in the following version
\begin{lem}\label{lemme16}
For $r\ge 1$ and $z\ge 300$ we have
\begin{equation*}
\sump_{n\le N}\frac{\tau_r(n)}{n}\le \left( 1.04\frac{\log N}{\log z} \right)^r \quad\text{et}\quad \sump_{n\le N} \tau_r(n) \le N 3^r \frac{(\log N)^{r-1}}{(\log z)^r}.
\end{equation*}
\end{lem}
\begin{proof}
\begin{equation*}
\sump_{n\le N}\frac{\tau_r(n)}{n}\le \left( \sump_{n\le N} \frac{1}{n}\right)^r \le \left( \prod_{z\le p \le N} \frac{1}{1 - \frac{1}{p}}\right)^r
\end{equation*}
and we use Lemma \ref{lemme14}. For the second identity, we apply Lemma \ref{lemme15} to the function $\tau$, noticing that $r^2(2.08)^r\le 3^r$.
\end{proof}
Some demonstrations of \cite{Ram} could be more detailed, which is what we do below.


\subsubsection{Other bounds}
The following bounds will be used later to establish our Theorem \ref{thm-crible} with a preliminary sieve. 

We recall that if $g$ is a $C^1$function over [1, X], $ \Vert g \Vert_{\infty} = \max_{1\le t\le X} \vert g(t) \vert$.
We will use $\Vert.\Vert$ for $ \Vert .\Vert_{\infty}$. And if $w=w_1\star ...\star w_k$, we define $\Vert w\Vert = \Vert w_1\Vert...\Vert w_{k}\Vert$.
\begin{lem}\label{lemme19}
Let w be a $C^1$ function over $[1, X]$,
\begin{equation*}
R'_z(w, \bar{f}, D, r)\le 3 R'_z(1, \bar{f}, D, r)\Vert w\Vert.
\end{equation*}
\end{lem}
It is Lemma 19 from \cite{Ram}.

\begin{lem}\label{lemme20}
Let $w_1,..., w_k$ $k$ be $C^1$ functions over $[1, X]$, and $w=w_1\star ...\star w_k$. Then, for $1\le D\le X$, we have
\begin{equation*}
R'_z(w, \bar{f}, X(D/X)^k, r)\le 3(2^k-1)R'_z(1, \bar{f}, D, r+k-1)\Vert w\Vert.
\end{equation*}
\end{lem}
For a proof see Lemma 20 in \cite{Ram}.

\subsubsection{A few lemmas}
And finally, a few lemmas that we will also need. The reader can find complete proofs of those in \cite{Ram} or \cite{Diam}.

\begin{lem}\label{lemme7}
\begin{equation*}
2 \nu \sum_{\pmb{k} \in \mathbbm{K}} \vert c(\pmb{k})w(\pmb{k})\vert + 
\nu^2\sum_{\pmb{h} \in \mathbbm{H}} \vert \bar{c}(\pmb{h})\bar{w}(\pmb{h})\vert \le (\nu +2) 2^{\nu-1}.
\end{equation*}
\end{lem}
Cf. \cite{Ram}, Lemma 7.

\begin{lem}\label{lemme8}
\begin{equation*}
\frac{V_{\sigma}(z)}{V_{\sigma_0}(z)} \ge \frac{1}{2 c^2}.
\end{equation*}
where $c$ is still a constant such that $\prode_{v\le p\le u} (1- \sigma(p))^{-1} + \prode_{v\le p\le u} (1- \sigma_0(p))^{-1} \le c \log u/\log v$ où $2 \le v\le u$.
\end{lem}
Cf. \cite{Ram}, Lemma 8.

\begin{lem}\label{lemme9}
For all $n \ge 1$ there exists $\theta_{+} \in ]0,1[$ such that
\begin{equation*}
n! = (2\pi n)^{1/2} (n/e)^n e^{\theta_{+}/(12n)}.
\end{equation*}
\end{lem}
Cf. \cite{Diam} p203 (2.3 Estimates), or Lemma 9 in \cite{Ram}.

\begin{lem}\label{lemme17}
Let $m \le M$ be an integers with no prime factor $\le z$. For all $\pmb{\ell}$ in $\mathbbm{K}$ or $\mathbbm{H}$, 
\begin{equation*}
L(\pmb{\ell})(m) \le \tau_{\ell}(m)\left( \frac{e \log m}{\nu_1+\nu_2}\right)^{\nu_1+\nu_2} \times \frac{1}{w(\pmb{\ell})}.
\end{equation*}
\end{lem}
Cf. \cite{Ram}, Lemma 17.

The following lemma will be used a few times later.
\begin{lem}\label{lemme24}
If $D_0/D\ge z$ and $\Delta (e + e r \delta)^{1/\delta} \le 1/2$, the remainder term $R'(w, \bar{f	}, X(D/X)^k , r)$ is not more than
\begin{multline*}
{3.2}^k \Vert w \Vert \biggl[ R(f, D_0, r+k)+ \frac{V_{\sigma}(z)}{V_{{\sigma}_0}(z)}R(f_0, D_0, r + k) \\
+ \left(2 C_0(c)e^{-\frac{\log \frac{D_0}{D}}{\log z}} + 2\Delta (e + e r \delta)^{1/\delta} \right) 
V_{\sigma}(z)(c/\delta)^{r+k-1}\hat{F}(X)\biggr].
\end{multline*}
\end{lem}
Cf. \cite{Ram}, Lemma 24.

\begin{lem}\label{lemme25}
\begin{equation*}
\sum_{h \ge 1}\frac{h^{\nu-1}}{p^{h-1}} \le \frac{(\nu-1)!}{(1-1/p)^{\nu}} \le 2^{\nu} (\nu-1)!
\end{equation*}
\end{lem}
Cf. \cite{Ram}, Lemma 25.

And at last
\begin{lem}\label{lemme22} (Fundamental Lemma for a linear sieve)
Let $M \ge 2$ and $\tilde{z}\ge 1$ be two real parameters.

There exist two sequences $(\lambda_m^+)$ and $(\lambda_m^-)$ with the following properties:
\begin{equation*}
\lambda_1^+ = \lambda_1^- = 1 , \qquad \vert \lambda_m^+ \vert , \vert \lambda_m^- \vert \le 1, \qquad \lambda_m^+   = \lambda_m^- = 0
\quad \text{pour } m > M.
\end{equation*}
For all $n$, if $(n, P(\tilde{z})) \neq 1$, we have
\begin{equation*}
\sum_{m \mid n} \lambda_m^- \le 0 \le \sum_{m \mid n} \lambda_m^+
\end{equation*}
and if $(n, P(\tilde{z})) = 1$, we have
\begin{equation*}
\sum_{m \mid n} \lambda_m^- = \sum_{m \mid n} \lambda_m^+ =1.
\end{equation*}
For any multiplicative function $\tilde{\sigma}$ such that 
\begin{equation}
0\le \tilde{\sigma} < 1 \quad \text{et} \quad \prode_{v \le p \le u} \left(1 - \tilde{\sigma}(p)\right)^{-1} \le \tilde{c}\frac{\log u}{\log v} 
\quad (2 \le u \le v),
\end{equation}
we have
\begin{equation*}
\sum_{m \mid P(\tilde{z})} \lambda_m^+ \tilde{\sigma}(m) \le \left( 1 + C_0(\tilde{c})e^{-\frac{\log M}{\log \tilde{z}}}\right) 
\prode_{p \le \tilde{z}}\left( 1 - \tilde{\sigma}(p)\right)
\end{equation*}
and
\begin{equation*}
\sum_{m \mid P(\tilde{z})} \lambda_m^- \tilde{\sigma}(m) \ge \left( 1 - C_0(\tilde{c})e^{-\frac{\log M}{\log \tilde{z}}}\right) 
\prode_{p \le \tilde{z}}\left( 1 - \tilde{\sigma}(p)\right)
\end{equation*}
where $C_0(\tilde{c})$ is a number which depends only on the constant $\tilde{c}$.
\end{lem}
Cf. \cite{Fried} p. 732, Lemma 5.

\subsection{Result with a preliminary sieve}

In the course of this work, $\text{MT}$ denotes $\frac{V_{\sigma}(z)}{V_{\sigma_0}(z)} \hat{F}(X) \frac{(\log X)^{\nu_1+\nu_2 - 1}}{(\nu_1+\nu_2 -1)!}.$
\begin{thm}\label{thm-crible}
Under the previous hypothesis, see \ref{hypo},
\begin{equation*}
\Sigma'_{\nu_1,\nu_2}(f, X) = \frac{V_{\sigma}(z)}{V_{\sigma_0}(z)}\Sigma'_{\nu_1,\nu_2}(f_0, X) + (\rho +\tilde{\rho}) \frac{V_{\sigma}(z)}{V_{\sigma_0}(z)} \hat{F}(X) \frac{(\log X)^{\nu_1+\nu_2 - 1}}{(\nu_1+\nu_2 -1)!}
\end{equation*}
where 
\begin{equation*}
\Sigma'_{\nu_1,\nu_2}(f, X) = \sump_{n\le X} \Lambda^{(\nu_1+\nu_2)}(n) f(n) + \sump_{n\le X} \Lambda^{(\nu_1)}\star \Lambda^{(\nu_2)}(n) f(n),
\end{equation*}
and
\begin{align*}
&\vert \rho \vert \le 2 \times (149 \max(\nu_1,\nu_2))^{\nu_1+\nu_2} \\ 
 &\times \biggl(Ac^2+\left( C_0(c) e^{-\frac{\log \frac{D_0}{D}}{\log z}} +\Delta(e + e \max(\nu_1,\nu_2) \delta)^{\frac{1}{\delta}} \right) \left(\frac{c}{\delta}\right)^{2\max(\nu_1, \nu_2)} \biggr)\\
& \text{et } \vert \tilde{\rho} \vert \le 2(24\max(\nu_1,\nu_2))^{\max(\nu_1,\nu_2)}B_0.\frac{1}{\delta^{2\max(\nu_1,\nu_2)}}\left(\frac{\log \frac{X}{T}}{\log X}\right)^{\nu_1+\nu_2+\max(\nu_1,\nu_2)}.
\end{align*}
\end{thm}
After we choose our parameters D and T and with $\nu_1\le \nu_2$, we will see that
\begin{equation*}
\vert \rho \vert \le 2 \times (149\nu_2)^{\nu_1+\nu_2} 
 \times \biggl(Ac^2+\left( C_0(c) e \delta^{3\nu_2} +\Delta4^{\frac{1}{\delta}} \right) \left(\frac{c}{\delta}\right)^{2 \nu_2} \biggr)
\end{equation*}

and
\begin{equation*}
\vert \tilde{\rho} \vert \le 2(24\nu_2)^{\nu_2}B_0 \left( 3\nu_2^2\log\frac{1}{\delta}\right)^{\nu_1+2\nu_2}\delta^{\nu_1}.
\end{equation*}

The proof sketch of Theorem \ref{thm-crible} is the following, with the same notations
\begin{equation*}
\Sigma'_{\nu_1,\nu_2}(f, X) = \sump_{\pmb{\ell}} a(\pmb{\ell})\bigl( S_1(f, T, \pmb{\ell}) + \mathcal{O}^*(S_2(f, T, \pmb{\ell}) )\bigr)
\end{equation*}
we will use the succession of approximations
\begin{equation}
S_1(f, T, \pmb{\ell})= \sump_{n\le T}   \mu_{\ell}(n)\sump_{m\le X/n}  L(\pmb{\ell})(m) f(mn)
\end{equation}
by $\frac{V_{\sigma}(z)}{V_{\sigma_0}(z)}S_1(f_0, T, \pmb{\ell})$
and
\begin{equation}
S_2(f, T, \pmb{\ell})= \sump_{m\le X/T} L(\pmb{\ell})(m) \sump_{T\le n\le X/m} \tau_{\ell}(n)f(mn)
\end{equation}
by $\frac{V_{\sigma}(z)}{V_{\sigma_0}(z)}S_2(f_0, T, \pmb{\ell})$.

Indeed
\begin{align*}
\sump_{n \le X} \bigl(\Lambda^{(\nu_1+\nu_2)} + \Lambda^{(\nu_1)} \star\Lambda^{(\nu_2)}\bigr)(n) f(n) = \sump_{n\le X}\sump_{\pmb{\ell}} a(\pmb{\ell})\mu_{\ell} \star L(\pmb{\ell})(n) f(n)\\
= \sump_{\pmb{\ell}} a(\pmb{\ell}) \sump_{n\le X} \sump_{d \mid n} \mu_{\ell}(d)  L(\pmb{\ell})(\frac{n}{d}) f(n)\\
= \sump_{\pmb{\ell}} a(\pmb{\ell}) \sump_{d\le X}\mu_{\ell}(d) \sump_{m\le \frac{X}{d}}   L(\pmb{\ell})(m) f(dm).\\
\end{align*}
so
\begin{align*}
\sump_{n\le X}\mu_{\ell}(n) \sump_{m\le \frac{X}{n}}   L(\pmb{\ell})(m) f(mn) =& \sump_{n\le T}\mu_{\ell}(n) \sump_{m\le \frac{X}{n}}   L(\pmb{\ell})(m) f(mn)\\
&\qquad\qquad\quad + \sump_{T\le n\le X}\sump_{m\le \frac{X}{n}}\mu_{\ell}(n)  L(\pmb{\ell})(m) f(mn)\\
=& S_1(f, T, \pmb{\ell}) + \sump_{T\le n\le X}\sump_{m\le \frac{X}{n}}\mu_{\ell}(n)  L(\pmb{\ell})(m) f(mn).
\end{align*}
with
\begin{equation*}
\left| \sump_{T\le n\le X}\sump_{m\le \frac{X}{n}}\mu_{\ell}(n)  L(\pmb{\ell})(m) f(mn) \right| \le \sump_{T\le n\le X}\sump_{m\le \frac{X}{n}}\tau_{\ell}(n)  L(\pmb{\ell})(m) f(mn) = S_2(f, T, \pmb{\ell}).
\end{equation*}

We will find an upper bound for $S_2(f_0, T, \pmb{\ell})$, and so our estimation will be
\begin{align*}
\Sigma'_{\nu_1,\nu_2}(f, X) &= \sum_{\pmb{\ell}} a(\pmb{\ell})\bigl( S_1(f, T, \pmb{\ell}) + \mathcal{O}^*(S_2(f, T, \pmb{\ell}) )\bigr)\\
&= \sum_{\pmb{\ell}} a(\pmb{\ell})\frac{V_{\sigma}(z)}{V_{\sigma_0}(z)}S_1(f_0, T, \pmb{\ell}) \\
&\qquad\qquad+ R_1 + \sum_{\pmb{\ell}} a(\pmb{\ell})\mathcal{O}^*\left(\frac{V_{\sigma}(z)}{V_{\sigma_0}(z)}S_2(f_0, T, \pmb{\ell})\right) + R_2.
\end{align*}
As we approximate $S_1(f, T, \pmb{\ell})$ by $S_1(f_0, T, \pmb{\ell})$, $R_1$ is the remainder term from this approximation. And when we approximate $S_2(f, T, \pmb{\ell})$ by $S_2(f_0, T, \pmb{\ell})$,  $R_2$ is the corresponding remainder term.

Meanwhile
\begin{equation*}
\Sigma'_{\nu_1,\nu_2}(f_0, X) = \sum_{\pmb{\ell}} a(\pmb{\ell})(S_1(f_0, T, \pmb{\ell}) + \mathcal{O}^*(S_2(f_0, T, \pmb{\ell}) )
\end{equation*}
so
\begin{equation*}
\sum_{\pmb{\ell}} a(\pmb{\ell})\frac{V_{\sigma}(z)}{V_{\sigma_0}(z)}S_1(f_0, T, \pmb{\ell})= \frac{V_{\sigma}(z)}{V_{\sigma_0}(z)}\Sigma'_{\nu_1,\nu_2}(f_0, X) + \sum_{\pmb{\ell}} a(\pmb{\ell})\mathcal{O}^*\left(\frac{V_{\sigma}(z)}{V_{\sigma_0}(z)}S_2(f_0, T, \pmb{\ell})\right) .
\end{equation*}
We bound in $\mathcal{O}^*$ in absolute values so
\begin{align*}
\Sigma'_{\nu_1,\nu_2}(f, X) &= \frac{V_{\sigma}(z)}{V_{\sigma_0}(z)}\Sigma'_{\nu_1,\nu_2}(f_0, X) + 2\sum_{\pmb{\ell}} a(\pmb{\ell})\mathcal{O}^* \left(\frac{V_{\sigma}(z)}{V_{\sigma_0}(z)}S_2(f_0, T, \pmb{\ell})\right)\\
&\qquad\qquad\qquad\qquad + R_1 + R_2\\
&=\frac{V_{\sigma}(z)}{V_{\sigma_0}(z)}\Sigma'_{\nu_1,\nu_2}(f_0, X) +(\rho+\tilde{\rho})\text{MT}
\end{align*}

where $\rho$ is such that
\begin{equation}
R_1 + R_2 = \rho \times \text{MT}
\end{equation}
and
\begin{equation}
\tilde{\rho} =\mathcal{O}\left( \frac{2\sum_{\pmb{\ell}} a(\pmb{\ell})\frac{V_{\sigma}(z)}{V_{\sigma_0}(z)}S_2(f_0, T, \pmb{\ell})}{\text{MT}}\right),
\end{equation}
which we will estimate.

\subsubsection{Bounding the first error term $R_1$}
We bound $\vert S_1(f, T, \pmb{\ell}) - \frac{V_{\sigma}(z)}{V_{\sigma_0}(z)}S_1(f_0, T, \pmb{\ell})\vert$ where 
\begin{equation}
S_1(f, T, \pmb{\ell})= \sump_{n\le T}  \mu_{\ell}(n)\sump_{m\le X/n} L(\pmb{\ell})(m) f(mn).
\end{equation}
We use Lemma \ref{lemme24} with $\ell\le \max(\nu_1,\nu_2)$, and hypothesis $(H_2)$ and \eqref{H4}. With $T = X(D/X)^{\max(\nu_1,\nu_2)}$ so $k= \max(\nu_1,\nu_2)$ and noticing that $\vert \vert L(\pmb{\ell})\vert \vert \le (\log X)^{\nu_1 +\nu_2}$,
we find the following upper bound:
\begin{multline*}
3.2^{\max(\nu_1,\nu_2) +1} \hat{F}(X) (\log X)^{\nu_1+\nu_2} \\
\times \biggl( \frac{A}{2 \log X}+\left( C_0(c) e^{-\frac{\log \frac{D_0}{D}}{\log z}} +\Delta(e + e\ell \delta)^{\frac{1}{\delta}} \right) V_{\sigma}(z) \left( \frac{c}{\delta}\right) ^{\ell +\max(\nu_1,\nu_2) -1} \biggr).
\end{multline*}
Indeed, we are still using Ramar\'e 's notations of \cite{Ram}, and we use Lemma \ref{lemme24} to give an upper bound for $R'_z(L(\pmb{\ell}), \bar{f}, T, \ell)$. 
where $r'_{n,z}(L(\pmb{\ell}),y)=\sum_{m\le X/n}L(\pmb{\ell})\bar{f}(mn)$.

Summing over $\pmb{\ell}$ and multiplying by $a(\pmb{\ell})$ (we recall that $\ell \le \max(\nu_1,\nu_2)$), we then use Lemme \ref{majoc1}. We find that $R_1$ is no more than
\begin{align*}
&\qquad 3\times 2^{\max(\nu_1, \nu_2)}\times 8^{\nu_1+\nu_2} \frac{\max(\nu_1,\nu_2)^{\min(\nu_1,\nu_2)+1}}{\min(\nu_1,\nu_2)^{\min(\nu_1,\nu_2)-1}}(\log X)^{\nu_1+\nu_2-1}\frac{V_{\sigma}(z)}{V_{\sigma_0}(z)}\hat{F}(X) \\
&\times \biggl( \frac{A}{2 \frac{V_{\sigma}(z)}{V_{\sigma_0}(z)}}+\left( C_0(c) e^{-\frac{\log \frac{D_0}{D}}{\log z}} +\Delta(e + e \max(\nu_1+\nu_2) \delta)^{\frac{1}{\delta}} \right) V_{\sigma_0}(z)\log X \left( \frac{c}{\delta}\right) ^{2\max(\nu_1,\nu_2) -1} \biggr)\\
&\qquad = \text{MT} (\nu_1+\nu_2-1)! \times 3\times 2^{\max(\nu_1, \nu_2)}\times 8^{\nu_1+\nu_2} \frac{\max(\nu_1,\nu_2)^{\min(\nu_1,\nu_2)+1}}{\min(\nu_1,\nu_2)^{\min(\nu_1,\nu_2)-1}} \times \biggl( ... \biggr)
\end{align*}
(what is inside the brackets remains the same). Now from Lemma \ref{lemme8},
\begin{equation*}
\frac{A}{2 \frac{V_{\sigma}(z)}{V_{\sigma_0}(z)}} \le Ac^2.
\end{equation*}
From $(H_2)$, we have $V_{\sigma_0}(z)< \frac{c}{\delta/\log X}$, so
\begin{equation*}
V_{\sigma_0}(z)\log X \left( \frac{c}{\delta}\right) ^{2\max(\nu_1,\nu_2) -1} \le \frac{c^{2\max(\nu_1, \nu_2)}}{\delta^{2\max(\nu_1, \nu_2)}}.
\end{equation*}
Then
\begin{equation}
(\nu_1+\nu_2-1)! \times 3\times 2^{\max(\nu_1, \nu_2)}\times 8^{\nu_1+\nu_2} \frac{\max(\nu_1,\nu_2)^{\min(\nu_1,\nu_2)+1}}{\min(\nu_1,\nu_2)^{\min(\nu_1,\nu_2)-1}}
\end{equation}
is no more than $(149 \max(\nu_1,\nu_2))^{\nu_1+\nu_2}$.
\begin{proof}
Indeed, we use Lemma \ref{lemme9} which tells us $(\nu_1+\nu_2)!$ is lower than 
\begin{equation*}
\sqrt{2\pi (\nu_1+\nu_2)}\left( \frac{\nu_1+\nu_2}{e}\right)^{\nu_1+\nu_2} e^{\frac{1}{12(\nu_1+\nu_2)}},
\end{equation*}
and take the logarithm.

We use $1$ for a lower bound of $\min(\nu_1,\nu_2)$, and $2$ for $\nu_1+\nu_2$ and $\max(\nu_1,\nu_2)$ is less or equal to $\nu_1+\nu_2$, we also use $\log 2 + \log(\max(\nu_1,\nu_2))$  for an upper bound of $\log(\nu_1+\nu_2)$ and so the logarithm of this whole expression is no more than
 \begin{multline*}
(\nu_1+\nu_2)\left( \log(\frac{8}{e})+2\log(2)+\frac{\log \pi+ \frac{1}{24}}{2}\right) + (\nu_1+\nu_2)\log(\max(\nu_1,\nu_2)) \\
\le (\nu_1+\nu_2)(5+\log(\max(\nu_1,\nu_2)).
 \end{multline*}
 
\end{proof}
 And so
 \begin{multline*}
R_1\le \text{MT} \times (149 \max(\nu_1,\nu_2))^{\nu_1+\nu_2} \\ 
 \times \biggl(Ac^2+\left( C_0(c) e^{-\frac{\log \frac{D_0}{D}}{\log z}} +\Delta(e + e \max(\nu_1,\nu_2) \delta)^{\frac{1}{\delta}} \right) (c/\delta)^{2\max(\nu_1, \nu_2)} \biggr).
 \end{multline*}

\subsubsection{Bounding the second error term $R_2$}
We bound $\vert S_2(f, T, \pmb{\ell}) - \frac{V_{\sigma}(z)}{V_{\sigma_0}(z)}S_2(f_0, T, \pmb{\ell})\vert$ i.e. $\vert S_2(\bar{f}, T, \pmb{\ell})\vert$ where 
\begin{equation}
S_2(\bar{f}, T, \pmb{\ell})= \sump_{m\le \frac{X}{T}} L(\pmb{\ell})(m)\sump_{T\le n\le X/m}\tau_{\ell}(n) \bar{f}(mn).
\end{equation}
In order to do so, we use the fact $L(\pmb{\ell})(m)$ is less than $\vert \vert L(\pmb{\ell})\vert \vert \le (\log \frac{X}{T})^{\nu_1 +\nu_2}$ and we use again Lemma \ref{lemme24} but here with $r=1$, $w=\tau_{\ell}$ and $k=\max(\nu_1,\nu_2)$.

Lemma \ref{lemme24} gives us a upper bound for $R'_z(\tau_{\ell}, \bar{f}, T, 1)$ assuming $T\le X\left(\frac{D}{X}\right)^{\max(\nu_1,\nu_2)}$. We find that
\begin{multline*}
\sump_{m\le \frac{X}{T}}\left| \sump_{T\le n\le X/m} \tau_{\ell}(n) \bar{f}(mn)\right| \\
\le 3 \times 2^{\max(\nu_1,\nu_2)} \Biggl( R(f,D_0,\max(\nu_1,\nu_2)+1) + \frac{V_{\sigma}(z)}{V_{\sigma_0}(z)}R(f_0,D_0,\max(\nu_1,\nu_2)+1)\\
+\Bigl( 2C_0(c)e^{-\frac{\log \frac{D_0}{D}}{\log z}}+ 2 \Delta(e+e\delta)^{\frac{1}{\delta}}\Bigr) V_{\sigma}(z) (\frac{c}{\delta})^{\max(\nu_1,\nu_2)} \hat{F}(X)\Biggr).
\end{multline*}
So, using \eqref{H4} and our upper bound for $\vert \vert L(\pmb{\ell})\vert \vert$ we get
\begin{multline*}
\vert S_2(\bar{f}, T, \pmb{\ell})\vert
 \le 3 \times 2^{\max(\nu_1,\nu_2)+1}\left( \log \frac{X}{T}\right)^{\nu_1+\nu_2}  \\
 \left( \frac{A\hat{F}(X)}{2\log X}
+( C_0(c)e^{-\frac{\log \frac{D_0}{D}}{\log z}}+  \Delta(e+e\delta)^{\frac{1}{\delta}}) V_{\sigma}(z) (\frac{c}{\delta})^{\max(\nu_1,\nu_2)} \hat{F}(X)\right)\\
\le 3 \times 2^{\max(\nu_1,\nu_2)+1}\hat{F}(X)\left( \log \frac{X}{T}\right)^{\nu_1+\nu_2} \frac{V_{\sigma}(z)}{V_{\sigma_0}(z)} \\
 \left( \frac{A}{2\log X} \times \frac{V_{\sigma_0}(z)}{V_{\sigma}(z)}
+( C_0(c)e^{-\frac{\log \frac{D_0}{D}}{\log z}}+  \Delta(e+e\delta)^{\frac{1}{\delta}}) V_{\sigma_0}(z) (\frac{c}{\delta})^{\max(\nu_1,\nu_2)} \right).
\end{multline*}
From Lemma \ref{lemme8}, $\frac{V_{\sigma}(z)}{V_{\sigma_0}(z)}$ is less or equal to $2 c^2$  and $V_{\sigma_0}(z) \le \frac{c}{\log z}$ cf. $(H_2)$.
We multiply by $a(\pmb{\ell})$ and we sum over $\ell$ using Lemma \ref{lemme7} to get an upper bound.
So our remainder term is no more than
\begin{multline*}
3 \times 2^{\max(\nu_1,\nu_2)+1}\times 8^{\nu_1+\nu_2} \frac{\max(\nu_1,\nu_2)^{\min(\nu_1,\nu_2)+1}}{\min(\nu_1,\nu_2)^{\min(\nu_1,\nu_2)-1}}\hat{F}(X)\left( \log \frac{X}{T}\right)^{\nu_1+\nu_2-1} \frac{V_{\sigma}(z)}{V_{\sigma_0}(z)} \\
 \left( \frac{Ac^2 \log \frac{X}{T}}{\log X}
+( C_0(c)e^{-\frac{\log \frac{D_0}{D}}{\log z}}+  \Delta(e+e\delta)^{\frac{1}{\delta}}) V_{\sigma_0}(z) (\frac{c}{\delta})^{\max(\nu_1,\nu_2)}c\frac{\log \frac{X}{T}}{\log z} \right).
\end{multline*}
This identity can be related to $R_1$'s upper bound. Indeed we have $\left( \log \frac{X}{T}\right)^{\nu_1+\nu_2-1}$ instead of $(\log X)^{\nu_1+\nu_2-1}$, and inside the brackets we can use a few boundings:
\begin{equation*}
\frac{Ac^2 \log \frac{X}{T}}{\log X} \le Ac^2
\end{equation*}
\begin{equation*}
(e+e\delta)^{\frac{1}{\delta}} \le (e+e\max(\nu_1,\nu_2)\delta)^{\frac{1}{\delta}}
\end{equation*}
because $1/\delta >1$.
\begin{equation*}
(\frac{c}{\delta})^{\max(\nu_1,\nu_2)}c\frac{\log \frac{X}{T}}{\log z} \le (\frac{c}{\delta})^{\max(\nu_1,\nu_2)}c\frac{\log X}{2\log z}
\le (\frac{c}{\delta})^{\max(\nu_1,\nu_2)+1} \le (\frac{c}{\delta})^{2\max(\nu_1,\nu_2)+1}
\end{equation*}
because $T^2 \ge X$.

And so $R_2$'s upper bound is less than $R_1$'s upper bound with the factor (multiplied by)
\begin{equation}
2\left(\frac{\log \frac{X}{T}}{\log X}\right)^{\nu_1+\nu_2-1} \le 2 \left(\frac{1}{2}\right)^{\nu_1+\nu_2-1} \le 1.
\end{equation}
\subsubsection{Bounding $\sum_{\pmb{\ell}} a(\pmb{\ell})\frac{V_{\sigma}(z)}{V_{\sigma_0}(z)}S_2(f_0, T, \pmb{\ell})$}
From $(H_1)$, $f_0$ is less than or equal to $B_0\frac{\hat{F}(X)}{X}$ so we just need to find an upper bound for $S_2(\mathbbm{1}, T, \pmb{\ell})$.
We use Lemma \ref{lemme16} and the fact that $X/m\ge T \ge \sqrt{X}$.
\begin{align}
S_2(\mathbbm{1}, T, \pmb{\ell}) &= \sump_{m\le X/T}L(\pmb{\ell})(m)\sump_{T < n\le X/m}\tau_\ell(n)\\
\notag &\le X\frac{3^\ell(\log \frac{X}{T})^{\ell-1}}{(\log z)^\ell}\sump_{m\le X/T}\frac{L(\pmb{\ell})(m)}{m}\\
\notag &\le X\frac{3^\ell(\log \sqrt{X})^{\ell-1}}{(\log z)^\ell}\sump_{m\le X/T}\frac{L(\pmb{\ell})(m)}{m}.
\end{align}
Using Lemma \ref{lemme17} and recalling that $\delta=\frac{\log z}{\log X}$ we get
\begin{align}
S_2(\mathbbm{1}, T, \pmb{\ell}) &\le \frac{2X}{\log X}\left(\frac{3}{2\delta}\right)^\ell \sump_{m\le \frac{X}{T}}\frac{\tau_\ell(m)}{m} \left( \frac{e \log \frac{X}{T}}{\nu_1+\nu_2}\right)^{\nu_1+\nu_2} \times \frac{1}{w(\pmb{\ell})}\\
\notag &\le \frac{2X}{\log X}\left(\frac{3}{2\delta}\right)^{\max(\nu_1,\nu_2)} \left( \frac{e \log \frac{X}{T}}{\nu_1+\nu_2}\right)^{\nu_1+\nu_2} \times \frac{1}{w(\pmb{\ell})}\left(1.04\frac{\log \frac{X}{T}}{\log z}\right)^{\max(\nu_1,\nu_2)}.
\end{align}
In the last line we used Lemma \ref{lemme16} and the fact that $\ell \le \max(\nu_1,\nu_2)$. 
As  $\log z=\delta\log X$ and $\nu_1+\nu_2\le 2\max(\nu_1,\nu_2)$, we have
\begin{multline*}
S_2(\mathbbm{1}, T, \pmb{\ell})\\
\le  \frac{2X}{\log X}\left(\frac{12}{\delta^2}\right)^{\max(\nu_1,\nu_2)} . \frac{1}{(\nu_1+\nu_2)^{\nu_1+\nu_2}}. \frac{1}{w(\pmb{\ell})}.\left(\log \frac{X}{T}\right)^{\nu_1+\nu_2}.\left(\frac{\log \frac{X}{T}}{\log X}\right)^{\max(\nu_1,\nu_2)}
\end{multline*}
and we find that
\begin{multline*}
w(\pmb{\ell})S_2(f_0, T, \pmb{\ell})\\
\le  \frac{2X}{\log X}\left(\frac{12}{\delta^2}\right)^{\max(\nu_1,\nu_2)}  \frac{1}{(\nu_1+\nu_2)^{\nu_1+\nu_2}}\left(\log \frac{X}{T}\right)^{\nu_1+\nu_2}\left(\frac{\log \frac{X}{T}}{\log X}\right)^{\max(\nu_1,\nu_2)}B_0\frac{\hat{F}(X)}{X}.
\end{multline*}

At last we sum over $\pmb{\ell}$, multiplying by $c(\pmb{\ell})\times \frac{V_{\sigma_0}(z)}{V_\sigma(z)}$ and we use the bounding in Lemma \ref{majoc2}:
\begin{align*}
&\sum_{\pmb{\ell}} a(\pmb{\ell})\frac{V_{\sigma}(z)}{V_{\sigma_0}(z)}S_2(f_0, T, \pmb{\ell})\\
&\le \frac{2X}{\log X}\left(\frac{12}{\delta^2}\right)^{\max(\nu_1,\nu_2)}  \frac{1}{(\nu_1+\nu_2)^{\nu_1+\nu_2}}\left(\log \frac{X}{T}\right)^{\nu_1+\nu_2}\left(\frac{\log \frac{X}{T}}{\log X}\right)^{\max(\nu_1,\nu_2)}\\
&\qquad \qquad \times B_0\frac{\hat{F}(X)}{X} 8^{\nu_1+\nu_2} \frac{\max(\nu_1,\nu_2)^{\min(\nu_1,\nu_2)+1}}{\min(\nu_1,\nu_2)^{\min(\nu_1,\nu_2)-1}}
\frac{V_{\sigma_0}(z)}{V_\sigma(z)}\\
&\le 2\left(\frac{12}{\delta^2}\right)^{\max(\nu_1,\nu_2)}  \frac{1}{(\nu_1+\nu_2)^{\nu_1+\nu_2}}\left(\frac{\log \frac{X}{T}}{\log X}\right)^{\nu_1+\nu_2+\max(\nu_1,\nu_2)}\\
&\qquad \qquad \times 8^{\nu_1+\nu_2} \frac{\max(\nu_1,\nu_2)^{\min(\nu_1,\nu_2)+1}}{\min(\nu_1,\nu_2)^{\min(\nu_1,\nu_2)-1}}B_0. \text{MT}.(\nu_1+\nu_2-1)!
\end{align*}

Here, $X$ has been simplified, and by reintroducing MT we divided by $(\log X)^{\nu_1+\nu_2 - 1}$. A $\log X$ has been simplified.
Finally,
\begin{equation}
\sum_{\pmb{\ell}} a(\pmb{\ell})\frac{V_{\sigma}(z)}{V_{\sigma_0}(z)}S_2(f_0, T, \pmb{\ell})\le \text{Cst}. B_0.\text{MT}\frac{1}{\delta^{2\max(\nu_1,\nu_2)}}\left(\frac{\log \frac{X}{T}}{\log X}\right)^{\nu_1+\nu_2+\max(\nu_1,\nu_2)}
\end{equation}
where the constant is no more than $(24\max(\nu_1,\nu_2))^{\max(\nu_1,\nu_2)}$.

Indeed, using Lemma \ref{lemme9} to bound $(\nu_1+\nu_2)!$, we find that
\begin{multline*}
\text{Cst} \le 2 \times 12^{\max(\nu_1,\nu_2)}\times \left(\frac{8}{e}\right)^{\nu_1+\nu_2}\frac{\max(\nu_1,\nu_2)^{\min(\nu_1,\nu_2)+1}}{\min(\nu_1,\nu_2)^{\min(\nu_1,\nu_2)-1}}\times \frac{\sqrt{2\pi}}{(\nu_1+\nu_2)^\frac{1}{2}}e^{\frac{1}{12(\nu_1+\nu_2)}}.\\
\end{multline*}
We take the logarithm, $\log(\text{Cst})$ is no more than
\begin{align*}
&\log 2 + \max(\nu_1,\nu_2)\log 12 + (\nu_1+\nu_2)\log\frac{8}{e}+ (\min(\nu_1,\nu_2)+1)\log(\max(\nu_1,\nu_2)) \\
&- (\min(\nu_1,\nu_2)-1)\log(\min(\nu_1,\nu_2)) +\frac{1}{2}\log(2\pi) - \frac{1}{2}\log(\nu_1+\nu_2)+\frac{1}{12(\nu_1+\nu_2)} \\
& \le \log 2 +\frac{1}{2}\log(\pi)+\frac{1}{24}+\max(\nu_1,\nu_2)\log \left(\Bigl(\frac{8}{e}\Bigr)^2 \max(\nu_1,\nu_2)\right) \\
& \le \max(\nu_1,\nu_2)\biggl(1 + \log\left(\biggl(\frac{8}{e}\biggr)^2\max(\nu_1,\nu_2)\right)\biggr),
\end{align*}
as $1\le \min(\nu_1,\nu_2)$,  $2\le \nu_1+\nu_2 \le 2\max(\nu_1,\nu_2)$ and $\frac{ \log 2 +\frac{1}{2}\log(\pi) + \frac{1}{24}}{\max(\nu_1,\nu_2)}\le 1$.


\subsubsection{Choice of the parameters $D$ and $T$}
The delicate matter here is to simultaneously minimize $\vert \rho\vert$ et $\vert \tilde{\rho}\vert$.
First, let us recall that
\begin{eqnarray}
\left\lbrace
\begin{array}{cc}
\log z = \delta \log X\\
\log D_0 = (1 - \delta) \log X\\
\end{array}\right.
\end{eqnarray}
Also, we supposed that
\begin{eqnarray*}
X \ge T \ge \sqrt{X} \ge z \text{ et } T = X\left(\frac{D}{X}\right)^{\max(\nu_1,\nu_2)}.
\end{eqnarray*}
Let us take $\eta$ such that
\begin{equation*}
\log T = (1-\eta\max(\nu_1,\nu_2))\log X
\end{equation*}
so $\log D = (1-\eta)\log X$. Minimizing $\vert \rho\vert$, where 
\begin{align*}
&\vert \rho \vert \le 2 \times (149 \max(\nu_1,\nu_2))^{\nu_1+\nu_2} \\ 
 &\times \biggl(Ac^2+\left( C_0(c) e^{-\frac{\log \frac{D_0}{D}}{\log z}} +\Delta(e + e \max(\nu_1,\nu_2) \delta)^{\frac{1}{\delta}} \right) \left(\frac{c}{\delta}\right)^{2\max(\nu_1, \nu_2)} \biggr),\\
 \end{align*}
$\Delta$ being typically very small if not null (so being $A$), is indeed minimizing
\begin{equation}
\displaystyle e^{-\frac{\log \frac{D_0}{D}}{\log z}}\delta^{-2\max(\nu_1,\nu_2)} = e^{-\frac{\eta -\delta}{\delta}}\delta^{-2\max(\nu_1,\nu_2)}.
\end{equation}
At the same time, we want to minimize $\vert \tilde{\rho}\vert$ where
\begin{align*}
&\vert \tilde{\rho} \vert \le (24\max(\nu_1,\nu_2))^{\max(\nu_1,\nu_2)}B_0.\frac{1}{\delta^{2\max(\nu_1,\nu_2)}}\left(\frac{\log \frac{X}{T}}{\log X}\right)^{\nu_1+\nu_2+\max(\nu_1,\nu_2)}.
\end{align*}
so we want to minimize
\begin{align*}
&\left(\frac{\log \frac{X}{T}}{\log X}\right)^{\nu_1+\nu_2+\max(\nu_1,\nu_2)}\delta^{-2\max(\nu_1,\nu_2)} = (\eta\max(\nu_1,\nu_2))^{\nu_1+\nu_2+\max(\nu_1,\nu_2)} \delta^{-2\max(\nu_1,\nu_2)}.
\end{align*}

We choose
\begin{equation}
\eta = 3\max(\nu_1,\nu_2)\delta\log\frac{1}{\delta},
\end{equation}
so
\begin{equation*}
e^{-\frac{\eta -\delta}{\delta}}\delta^{-2\max(\nu_1,\nu_2)}=\exp(-3\max(\nu_1,\nu_2)\log\frac{1}{\delta}+1)\delta^{-2\max(\nu_1,\nu_2)}=e\delta^{\max(\nu_1,\nu_2)}
\end{equation*}
while
\begin{align*}
&(\eta\max(\nu_1,\nu_2))^{\nu_1+\nu_2+\max(\nu_1,\nu_2)} \delta^{-2\max(\nu_1,\nu_2)}\\
&\qquad \qquad =\left( 3\max(\nu_1,\nu_2)^2\delta\log\frac{1}{\delta}\right)^{\nu_1+\nu_2+\max(\nu_1,\nu_2)} \delta^{-2\max(\nu_1,\nu_2)}\\
&\qquad \qquad = \left( 3\max(\nu_1,\nu_2)^2\log\frac{1}{\delta}\right)^{\nu_1+\nu_2+\max(\nu_1,\nu_2)}\delta^{\min(\nu_1,\nu_2)}.
\end{align*}
And so
\begin{multline*}
\vert \rho \vert \le 2 \times (149\max(\nu_1,\nu_2))^{\nu_1+\nu_2} \\
 \times \biggl(Ac^2+\left( C_0(c) e \delta^{3\max(\nu_1,\nu_2)} +\Delta (e + e \max(\nu_1,\nu_2) \delta)^{\frac{1}{\delta}} \right) \left(\frac{c}{\delta}\right)^{2\max(\nu_1, \nu_2)} \biggr)
\end{multline*}
and
\begin{equation*}
\vert \tilde{\rho} \vert \le 2\times(24\max(\nu_1,\nu_2))^{\max(\nu_1,\nu_2)}B_0 \left(3\max(\nu_1,\nu_2)^2\log \frac{1}{\delta}\right)^{\nu_1+\nu_2+\max(\nu_1,\nu_2)}\delta^{\min(\nu_1,\nu_2)}.
\end{equation*}

\subsection{Removal of the preliminary sieving}
We give an upper bound for all that was not counted in the sieved sums, i.e. we sum over integers having at least a prime factor $\le z$, using the first expression which is
\begin{equation*}
\sum_{\substack{n \le X \\ (n, P(z)) \neq 1}}\Lambda^{(\nu_1+\nu_2)}(n) g(n) + \sum_{\substack{n \le X \\ (n, P(z)) \neq 1}}\Lambda^{(\nu_1)} \star \Lambda^{(\nu_2)}(n) g(n)
\end{equation*}
where we recall that $\Lambda^{(\nu)}(n)=\frac{\Lambda(n)(\log n)^{\nu -1}}{(\nu - 1)!}$, and when $g$ is $f$ or $f_0$.

We have $\delta \log X \ge \nu_1 + \nu_2 +1$ and $z \ge e^{\nu_1+\nu_2 +1}$.

The first sum contains few summands and we use \eqref{H8}  to get an upper bound for $f$ or $f_0$. For the second sum, we proceed in three steps, unsing the fact that $n$ has at most two prime factors. We write $n = p^h m$ with $p\le z$ and we first take $n\le X^{1/2}$, then we take $m$ a prime power $\le X^{1/2}$ and at last we take $n= p^h m> X^{1/2} $ where $m$ is a prime $>X^{1/2}$ (and so $p^h \le X^{1/2}$). In the first two cases, we also use the upper bound for $f$ or $f_0$ in \eqref{H8}. The third case is more delicate to handle.

We shall use the following results of \cite{Ros} (p. 1175) (for $X> 1$):
\begin{equation}\label{majopi}
\sum_{p\le X} \frac{\log p}{p} < \log X \qquad \text{et}\qquad   \pi (X) \le 1.26\frac{X}{\log X}.
\end{equation}
At last, we recall that MT $= \hat{F}(X)\frac{(\log X)^{\nu_1+\nu_2-1}}{(\nu_1+\nu_2-1)!}\frac{V_{\sigma}(z)}{V_{\sigma_0}(z)}$.

Note: from now we consider that $\nu_1 \le \nu_2$ so $\nu_1 = \min(\nu_1, \nu_2)$ and $\nu_2= \max(\nu_1, \nu_2)$ so our identities are easier to write (and read). All our calculations remain accurate for any $\nu_1$ and $\nu_2$ but the final upper bound is more obvious this way.

\subsubsection{First term: an upper bound for $\sum \Lambda^{(\nu_1+\nu_2)}(n) f(n)$}\label{7.1}
The only non-zero summands here are prime powers $n=p^h \le X$ and thus $\Lambda(n)=\log p$.
\begin{align*}
\sum_{\substack{n \le X \\ (n, P(z)) \neq 1}}\Lambda(n)(\log n)^{\nu_1+\nu_2-1} &=& \sum_{p\le z}\sum_{h\le \frac{\log X}{\log p}} (\log p)^{\nu_1+\nu_2} h^{\nu_1+\nu_2-1}\\
 &\le&  \sum_{p\le z}(\log p)^{\nu_1+\nu_2} \left(\frac{\log X}{\log p}\right)^{\nu_1+\nu_2}
&\le& 1.26 (\log X)^{\nu_1+\nu_2}\frac{z}{\log z}.
\end{align*}
We divide by $(\nu_1+\nu_2-1)!$ and use Lemma \ref{lemme8}  $\frac{V_{\sigma}(z)}{V_{\sigma_0}(z)}\ge \frac{1}{2 c^2}$, we use the upper bound $B$ for $f$ in \eqref{H8} and so
\begin{align}
\notag \sum_{\substack{n \le X \\ (n, P(z)) \neq 1}}\Lambda^{(\nu_1+\nu_2)}(n) f(n) &\le \text{MT}. B. c^2 \times 2.52 \frac{z}{\hat{F}(X)}\frac{\log X}{\log z}\\
&\le 2.52 B c^2 \text{MT} \frac{\delta^{\nu_2}}{\sqrt{X}}.
\end{align}
We have used \eqref{H8} to give a lower bound for $\hat{F}(X)$. 
Note that we can deduce the upper bound:
\begin{equation}
\sum_{\substack{n \le X \\ (n, P(z)) \neq 1}}\Lambda^{(\nu_1+\nu_2)}(n) f(n)  \le \frac{B}{\log X} c^2 \text{MT} \delta^{\nu_2}.
\end{equation}
\subsubsection{Upper bound for $\sum\Lambda^{(\nu_1)} \star \Lambda^{(\nu_2)}(n) f(n)$ pour $n\le X^{1/2}$}
In this case $n=p^h m$ such that $p^h$ and $m$ are $\le X^{1/2}$. 
We only calculate the convolution product over $p \le z$. It is equal to 
\begin{align*}
\sum_{\substack{n \le X^{1/2} \\ (n, P(z)) \neq 1}}\Lambda (\log)^{\nu_1-1}\star \Lambda (\log)^{\nu_2-1}(n)  &=& \sum_{p\le z}\sum_{h\le \frac{\log X}{2\log p}} (\log p)^{\nu_1} h^{\nu_1-1} \sum_{m\le \frac{\sqrt{X}}{p^h}}\Lambda(m)(\log m)^{\nu_2-1} \\
 &+& \sum_{p\le z}\sum_{h\le \frac{\log X}{2\log p}} (\log p)^{\nu_2} h^{\nu_2-1} \sum_{m\le \frac{\sqrt{X}}{p^h}}\Lambda(m)(\log m)^{\nu_1-1}.
\end{align*}
We call the first term of this sum $\sumun$ and give an upper bound for it, which we will also use for the second term, swapping $\nu_1$ and $\nu_2$.
\begin{align}
\sumun =& \sum_{p\le z}\sum_{h\le \frac{\log X}{2\log p}} (\log p)^{\nu_1} h^{\nu_1-1} \sum_{m\le \frac{\sqrt{X}}{p^h}}\Lambda(m)(\log m)^{\nu_2-1} \\
\notag \le& \sum_{p\le z}\sum_{h\le \frac{\log X}{2\log p}} (\log p)^{\nu_1} h^{\nu_1-1} \pi\left(\frac{\sqrt{X}}{p^h}\right)\left(\log \frac{\sqrt{X}}{p^h}\right)^{\nu_2}.
\end{align}
Using the upper bound (\ref{majopi}) for $\pi\left(\frac{\sqrt{X}}{p^h}\right)$ we get
\begin{align*}
\sumun  \le 1.26 \left(\frac{1}{2}\log X\right)^{\nu_2-1}\sqrt{X}\sum_{p\le z}\frac{(\log p)^{\nu_1}}{p}\sum_{h\le \frac{\log X}{2\log p}}\frac{h^{\nu_1-1}}{p^{h-1}}.
\end{align*}
Now, from Lemma \ref{lemme25}, $\sum_{h\ge 1}\frac{h^{\nu_1-1}}{p^{h-1}} \le 2^{\nu_1}(\nu_1-1)!$ (we use this upper bound for the far right hand factor) and having $\sum_{p\le z}\frac{(\log p)^{\nu_1}}{p} \le (\log z)^{\nu_1-1}\sum_{p\le z}\frac{\log p}{p} \le (\log z)^{\nu_1}$ we get:
\begin{align*}
\sumun  &\le 1.26 \times 2^{\nu_1-\nu_2+1}(\nu_1-1)!\sqrt{X} (\log X)^{\nu_2-1}(\log z)^{\nu_1}\\ 
&\le 2.52\times 2^{\nu_1-\nu_2}(\nu_1-1)!\sqrt{X} (\log X)^{\nu_1+\nu_2-1}\delta^{\nu_1}.
\end{align*}
We divide by $(\nu_1-1)!(\nu_2-1)!$,
\begin{equation}
\sumun\le 2.52\text{MT}. 2 c^2 \times 2^{\nu_1-\nu_2}\frac{(\nu_1+\nu_2-1)!}{(\nu_2-1)!}\frac{\sqrt{X}}{\hat{F}(X)}\delta^{\nu_1}.
\end{equation}
Here, again, we use \eqref{H8} and Lemma \ref{lemme9}: $(\nu_1+\nu_2)!\le \sqrt{2\pi (\nu_1+\nu_2)}\left( \frac{\nu_1+\nu_2}{e}\right)^{\nu_1+\nu_2} e^{\frac{1}{12(\nu_1+\nu_2)}}$ and $\nu_2!\ge \sqrt{2\pi \nu_2}\left( \frac{\nu_2}{e}\right)^{\nu_2}$ so our upper bound becomes:
\begin{align*}
& 5.04\text{MT} .c^2\times 2^{\nu_1-\nu_2}\delta^{\nu_1} \frac{\sqrt{\nu_2}}{\sqrt{\nu_1+\nu_2}}\left(\frac{\nu_1+\nu_2}{e}\right)^{\nu_1+\nu_2}\times \left(\frac{e}{\nu_2}\right)^{\nu_2}\frac{\delta^{\max(\nu_1,\nu_2)}}{z}\\
& \le 5.04 \text{MT}.c^2\times 2^{\nu_1-\nu_2}\frac{((\nu_1+\nu_2)\delta)^{\nu_1+\nu_2}}{z e^{\nu_1} \nu_2^{\nu_2}}\frac{\sqrt{\nu_2}}{\sqrt{\nu_1+\nu_2}}.
\end{align*}

Meanwhile, the second term
\begin{equation*}
\sum_{p\le z}\sum_{h\le \frac{\log X}{2\log p}} (\log p)^{\nu_2} h^{\nu_2-1} \sum_{m\le \frac{\sqrt{X}}{p^h}}\Lambda(m)(\log m)^{\nu_1-1}
\end{equation*}
is no more than 
\begin{equation*}
5.04 \text{MT}. c^2\times 2^{\nu_2-\nu_1}\frac{((\nu_1+\nu_2))^{\nu_1+\nu_2}\delta^{2\nu_2}}{z e^{\nu_2} \nu_1^{\nu_1}}\frac{\sqrt{\nu_1}}{\sqrt{\nu_1+\nu_2}}.
\end{equation*}
And because $\delta^{2\nu_2} \le \delta^{\nu_1+\nu_2}$ and $f(n)\le B$ we easily get:
\begin{equation*}
\sum_{\substack{n \le X^{1/2} \\ (n, P(z)) \neq 1}}\Lambda^{(\nu_1)} \star \Lambda^{(\nu_2)}(n) f(n) \le 0.01 \text{MT}. B. c^2 ((\nu_1+\nu_2)\delta)^{\nu_1+\nu_2}.
\end{equation*}

But the following upper bound will be more useful in our case $f(n)=\Lambda(n+2)$:
\begin{equation}
\sum_{\substack{n \le X^{1/2} \\ (n, P(z)) \neq 1}}\Lambda^{(\nu_1)} \star \Lambda^{(\nu_2)}(n) f(n)\le 11 \text{MT} \frac{B}{X^{\delta}}c^2\left(\frac{2}{e}\right)^{\nu_2}((\nu_1+\nu_2)\delta)^{\nu_1+\nu_2}.
\end{equation}
Assuming $X$ is large enough to have $11/X^{\delta}\le 1/\log X$, we deduce
\begin{equation}
\sum_{\substack{n \le X^{1/2} \\ (n, P(z)) \neq 1}}\Lambda^{(\nu_1)} \star \Lambda^{(\nu_2)}(n) f(n)\le  \text{MT} \frac{B}{\log X}c^2\left(\frac{2}{e}\right)^{\nu_2}((\nu_1+\nu_2)\delta)^{\nu_1+\nu_2}.
\end{equation}

\subsubsection{Upper bound for $\sum\Lambda^{(\nu_1)} \star \Lambda^{(\nu_2)}(n) f(n)$ for $m\le X^{1/2}$}
In this part we find an upper bound for $\sum_{\substack{n \le X \\ (n, P(z)) \neq 1}}\Lambda^{(\nu_1)} \star \Lambda^{(\nu_2)}(n) f(n)$ when $n=p^h m$ with $m=p'^k \le X^{1/2}$. 

Let us denote this sum by $\Sigma_m$. 

We deal with the convolution product the same way we did in part \ref{7.1} using the same upper bound for $\pi(n)$. Because of the values taken by $\Lambda$ and noticing that $m=p'^k \le X^{1/2}$ is equivalent to $k\le \frac{\log X}{\log p'}$, we have
\begin{multline*}
\sum_{p\le z}\sum_{h\le \frac{\log X}{\log p}} (\log p)^{\nu_1} h^{\nu_1-1} \sum_{p'\le \sqrt{X}}\sum_{k\le \frac{\log X}{2\log p'}}(\log p')^{\nu_2}k^{\nu_2-1}\\
 + \sum_{p\le z}\sum_{h\le \frac{\log X}{\log p}} (\log p)^{\nu_2} h^{\nu_2-1} \sum_{p'\le \sqrt{X}}\sum_{k\le \frac{\log X}{\log p'}}(2\log p')^{\nu_1}k^{\nu_1-1}
\end{multline*}
now
\begin{align}
& \sum_{p\le z}\sum_{h\le \frac{\log X}{\log p}} (\log p)^{\nu_1} h^{\nu_1-1} \sum_{p'\le \sqrt{X}}\sum_{k\le \frac{\log X}{2\log p'}}(\log p')^{\nu_2}k^{\nu_2-1} \\
\notag & \le \sum_{p\le z} \left( \frac{\log X}{\log p}\right)^{\nu_1}(\log p)^{\nu_1}\sum_{p'\le \sqrt{X}} \left( \frac{\log X}{2\log p'}\right)^{\nu_2}(\log p')^{\nu_2}\\
\notag & \le 2^{-\nu_2}(\log X)^{\nu_1+\nu_2} \pi (z) \pi (\sqrt{X}) \le 1.26^2 2^{-\nu_2}(\log X)^{\nu_1+\nu_2} \frac{z}{\log z} \frac{\sqrt{X}}{\log \sqrt{X}}\\
& \le 1.26^2 \times \frac{1}{2^{\nu_2-1}} (\log X)^{\nu_1+\nu_2-2}\frac{z\sqrt{X}}{\delta}.
\end{align}
The other term will have the same upper bound with a factor $\frac{1}{2^{\nu_1-1}}$. Finally we take the upper bound:
\begin{equation*}
1.26^2 \times \frac{1}{2^{\nu_1-2}} (\log X)^{\nu_1+\nu_2-2}\frac{z\sqrt{X}}{\delta}.
\end{equation*}
We divide the sum by $(\nu_1-1)!(\nu_2-1)!$, and use $B$ as the upper bound of $f(n)$ and we use Lemma \ref{lemme8}.

So 
\begin{align}
\notag \Sigma_m &\le 1.26^2 \times \frac{1}{2^{\nu_1-2}}  \text{MT}. B. \frac{(\nu_1+\nu_2-1)!}{(\nu_1-1)!(\nu_2-1)!}2 c^2 \frac{z\sqrt{X}}{\delta\hat{F}(X)\log X}\\
&\le 1.26^2 \times \frac{1}{2^{\nu_1-3}}  c^2 \text{MT}. B. \frac{(\nu_1+\nu_2-1)!}{(\nu_1-1)!(\nu_2-1)!} \frac{\delta^{\nu_2}}{\log X}.
\end{align}
Indeed, we use \eqref{H8} and $\frac{1}{2^{\nu_1-3}} \le 4$, and invoking Lemma \ref{lemme9} we have $(\nu_1+\nu_2)!\le \sqrt{2\pi (\nu_1+\nu_2)}\left( \frac{\nu_1+\nu_2}{e}\right)^{\nu_1+\nu_2} e^{\frac{1}{12(\nu_1+\nu_2)}}$ and $(\nu_i)!\ge \sqrt{2\pi \nu_i}\left( \frac{\nu_i}{e}\right)^{\nu_i}$ for $i\in \lbrace 1; 2 \rbrace$.
\begin{align*}
\Sigma_m &\le 1.26^2 \times 4 c^2 \text{MT}. B.\frac{\delta^{\nu_2}}{\log X}\frac{\sqrt{\nu_1}\sqrt{\nu_2}}{\sqrt{2\pi}\sqrt{\nu_1+\nu_2}} e^{\frac{1}{12(\nu_1+\nu_2)}} \frac{(\nu_1+\nu_2)^{\nu_1+\nu_2}}{\nu_1^{\nu_1}\nu_2^{\nu_2}}\\
&\le 3 c^2 \text{MT}. B.\delta^{\nu_2}\frac{\sqrt{\nu_1}\sqrt{\nu_2}}{\sqrt{\nu_1+\nu_2}} \frac{(2\nu_2)^{\nu_1+\nu_2}}{\nu_1^{\nu_1}\nu_2^{\nu_2}}\\
&\le 3 c^2 \text{MT}. \frac{B}{\log X} \left(\frac{\nu_2}{\nu_1}\right)^{\nu_1}(4 \delta)^{\nu_2}.
\end{align*}

\subsubsection{Upper bound for $\sum\Lambda^{(\nu_1)} \star \Lambda^{(\nu_2)}(n) f(n)$ for $m$ prime $\ge X^{1/2}$}
Note that $m \ge \sqrt{X}$ implies $p^h \le \sqrt{X}$. So now we are left with:

\begin{align}
\notag & \sum_{\substack{p^h \le\sqrt{X} \\ p\le z}} (\log p)^{\nu_1} h^{\nu_1-1} \sum_{ \sqrt{X}\le p'\le \frac{X}{p^h}}(\log p')^{\nu_2}f(p' p^h) \\
&+ \sum_{\substack{p^h \le\sqrt{X} \\ p\le z}} (\log p)^{\nu_2} h^{\nu_2-1} \sum_{ \sqrt{X}\le p'\le \frac{X}{p^h}}(\log p')^{\nu_1}f(p' p^h).
\end{align}
These terms are symmetrical in $\nu_1$ and $\nu_2$ so we find an upper bound for one of them and will deduce the second one by swapping $\nu_1$ and $\nu_2$. Finally we will take twice the greater upper bound.

We have
\begin{align}
\notag & \sum_{\substack{p^h \le\sqrt{X} \\ p\le z}} (\log p)^{\nu_2} h^{\nu_2-1} \sum_{ \sqrt{X}\le p'\le \frac{X}{p^h}}(\log p')^{\nu_1}f(p' p^h) \\
&\le \sum_{\substack{p^h \le\sqrt{X} \\ p\le z}} (\log p)^{\nu_2} h^{\nu_2-1}\max_{p'\le X/p^h}(\log p')^{\nu_1}  \sum_{ p'\le \frac{X}{p^h}}f(p' p^h).
\end{align}

We set a fixed $p$,  and we use the fact that the sum over primes $p'$ is no more than the sum over integers $n$ such that $n$ is relatively prime to $P(\tilde{z})$ but not with $p$, let $P_p(\tilde{z}) $ be the product of all primes $\le \tilde{z}$ except $p$, where $\tilde{z}=X^{1/4}$. 
It means we are using anothe sieve here, and we invoke the Fundamental Lemma (lemme \ref{lemme22}) with $M= \tilde{z}=X^{1/4}$, $\tilde{\sigma}=\sigma$ and the constant $\tilde{c}=c$ is unchanged. 
The $\lambda^+_d$ generated by the Fundamental Lemma will be used below. First
\begin{equation*}
\sum_{ p'\le \frac{X}{p^h}}f(p' p^h)\le \sum_{\substack{n \le X/p^h \\ (n,P(\tilde{z}))=1\\ (n,p)\neq 1}}f(n p^h)
\end{equation*}
with our sieve we get
\begin{equation}
\sum_{\substack{n \le X/p^h \\ (n,P(\tilde{z}))=1\\ (n,p)\neq 1}}f(n p^h)\le \sum_{n \le X/p^h}\sum_{\substack{d\mid P_p(\tilde{z}) \\ d\mid n}}\lambda^+_d f(n p^h)
\le \sum_{d\mid P_p(\tilde{z})}\lambda^+_d\sum_{\substack{n \le X/p^h \\ d\mid n}}f(n p^h).
\end{equation}

We write $n = dm$ (as $d\mid n$) in the second sum, we get $\sum_{m\le \frac{X}{d p^h}} f(m d p^h)$.

We use approximation (\ref{approx2}):
\begin{equation*}
\sum_{n\le y/d} f(dn)= \sigma(d)F(y)+r_{d}(f,y)
\end{equation*}
 with $y=X$ here, and $d$ is our $d p^h$. 
 
 \begin{align}
\notag & \sum_{p^h\le \sqrt{X}}(\log p)^{\nu_2} h^{\nu_2-1}\max_{p'\le X/p^h}(\log p')^{\nu_1}\sum_{d\mid P_p(\tilde{z})}\lambda^+_d r_{d p^h}(f,X)\\
 &\le \sum_{q\le D_0} \sum _{p^h\mid q}(\log p)^{\nu_2} h^{\nu_2-1}\max_{p'\le X/p^h}(\log p')^{\nu_1}r_q(f,X)\\
\notag  &\le \sum_{q\le D_0} (\log p)^{\nu_2-1} h^{\nu_2-1}\max_{p'\le X/p^h}(\log p')^{\nu_1}\sum _{p^h\mid q} (\log p) r_q(f,X)
 \end{align}
where $D_0 \ge X^{3/4}  (= \sqrt{X} X^{1/4})$ and we have taken $q=dp^h$.

The sum over $p^h$ is actually over $p$ and over $h$. Studying the variations of the function $y \longrightarrow (\log(y))^{\nu_2-1}(\log(\frac{X}{y}))^{\nu_1}$ for $y\le \sqrt{X}$, the reader can check that the maximum is reached over $\mathbb{R}$ for $y = X^{\frac{\nu_2-1}{\nu_1+\nu_2-1}} > \sqrt{X}$ (we assume that $\nu_2\ge \nu_1+1$) but $y\le \sqrt{X}$, and using the fact that $\sum _{p^h\mid q}\log p= \log q \le \log X$, the previous term is lower than or equal to
\begin{equation}
 \frac{1}{2^{\nu_1+\nu_2-1}}(\log X)^{\nu_1+\nu_2-1} \log X \sum_{q\le D_0}r_q(f,X)\le
\frac{1}{2^{\nu_1+\nu_2-1}}(\log X)^{\nu_1+\nu_2}R(f, D_0,1).
\end{equation}

Noticing that $d\mid P_p(\tilde{z})$ implies $d$ relatively prime to $p$, and $\sigma$ being multiplicative, $\sigma(d p^h) =\sigma(d)\sigma(p^h)$. 
 
So after we have applied the Fundametal Lemma, our initial expression is no more than
\begin{align*}
(\log X)^{\nu_1}\sum_{\substack{p^h \le\sqrt{X} \\ p\le z}}(\log p)^{\nu_2} h^{\nu_2-1}\sigma(p^h)\frac{V_{\sigma}(X^{1/4})}{1 - \sigma(p)}\hat{F}(X)(1+C_0(c)) \\
+ \frac{1}{2^{\nu_1+\nu_2-1}}(\log X)^{\nu_1+\nu_2}R(f, D_0,1).
\end{align*}
Indeed, we have here $e^{-(\log M)/\log z} = e^{-1} \le 1$ and the eulerian product in the upper bound is no more than $\frac{V_{\sigma}(X^{1/4})}{1 - \sigma(p)}$ (the factor $1 - \sigma(p)$ might be missing in the case of $p\le X^{1/4}$ as $(n,p)=1$).
In the first term, we simply used $\max_{p'\le X/p^h}(\log p')^{\nu_1}\le (\log X)^{\nu_1}$.

Regarding this first term, hypothesis $(H_1)$ imply, using $u = v= p$, that $(1- \sigma(p))^{-1}\le c$, and using hypothesis $(H_6)$ we have $\sigma(p^h)\le \frac{c}{p^h}$. 

So the first term is no more than
\begin{equation*}
(\log X)^{\nu_1}V_{\sigma}(X^{1/4})\sum_{p\le z}\frac{(\log p)^{\nu_2}}{p} \sum_{h\ge 1}\frac{h^{\nu_2-1}}{p^{h-1}} c^2 \hat{F}(X)(1+C_0(c)).
\end{equation*}

We now use Lemma \ref{lemme25}. On the other hand, hypothesis $(H_2)$ gives us $V_{\sigma_0}(z)\le c/\log z$, and $V_{\sigma_0}(X^{1/4})\le c/\log X$ comes from hypothesis $(H_7)$.
An upper bound for $V_{\sigma}(X^{1/4})$ can be found in (44) and (45) from \cite{Ram}, 
a constant $u$ arises, equal to $c$ in the case of $\sigma_0$ and equal to $c e^{c\Delta}$ in the case of $\sigma$ and so our upper bound is now:
\begin{equation*}
(\log X)^{\nu_1}\frac{V_{\sigma}(z)}{V_{\sigma_0}(z)}\frac{u}{\log X} (\log z)^{\nu_2} 2^{\nu_2}(\nu_2-1)! c^2 \hat{F}(X)(1+C_0(c)).
\end{equation*}

We divide both terms by $(\nu_1-1)!(\nu_2-1)!$ bounding again the factorials according to Lemma \ref{lemme9} and get the following upper bound:

\begin{align*}
&\qquad \qquad \sum_{\substack{p^h p' \le X \\ p^h \le\sqrt{X} \\ p\le z}}\Lambda^{(\nu_1)} \star \Lambda^{(\nu_2)}(n) f(n)\\
 &\le \text{MT}. u ( 2\delta)^{\nu_2} \frac{(\nu_1+\nu_2)^{\nu_1+\nu_2}}{\nu_2^{\nu_2}} e^{-\nu_1} e^{\frac{1}{24}} c^2 (1+C_0(c)) 
+\frac{1}{2^{\nu_1+\nu_2-1}}\frac{(\log X)^{\nu_1+\nu_2}}{(\nu_1-1)!(\nu_2-1)!} R(f, D_0,1)              \\
& \le \text{MT}. u \left( 2\delta (1 + \frac{\nu_1}{\nu_2})\right)^{\nu_2} \left( \frac{\nu_1+\nu_2}{e}\right)^{\nu_1} c^2 e^{\frac{1}{24}}(1+C_0(c))
 +\frac{1}{2^{\nu_1+\nu_2-1}}\frac{(\log X)^{\nu_1+\nu_2}}{(\nu_1-1)!(\nu_2-1)!} R(f, D_0,1)\\   
 & \le \text{MT}. u (4\delta)^{\nu_2} ( \nu_2)^{\nu_1} c^2 e^{\frac{1}{24}}(1+C_0(c))
 +\frac{1}{2^{\nu_1+\nu_2-1}}\frac{(\log X)^{\nu_1+\nu_2}}{(\nu_1-1)!(\nu_2-1)!} R(f, D_0,1).
\end{align*}
Thus the second term is no more than
\begin{equation*}
\text{MT}. u (4\delta)^{\nu_1} ( \nu_1)^{\nu_2} c^2 e^{\frac{1}{24}}(1+C_0(c))
 +\frac{1}{2^{\nu_1+\nu_2-1}}\frac{(\log X)^{\nu_1+\nu_2}}{(\nu_1-1)!(\nu_2-1)!} R(f, D_0,1)
\end{equation*}
and the convolution product is no more than
\begin{equation*}
2 \text{MT}. u (4\delta)^{\nu_1} ( \nu_1)^{\nu_2} c^2 e^{\frac{1}{24}}(1+C_0(c))
 +\frac{1}{2^{\nu_1+\nu_2-2}}\frac{(\log X)^{\nu_1+\nu_2}}{(\nu_1-1)!(\nu_2-1)!} R(f, D_0,1).
\end{equation*}

\subsubsection{Combining all that}
We add up the upper bounds from the fours parts above for $f_0$ and $f$, using $(H_9)$ so we know the contribution of  $R(f, D_0,1) $ are $\le A' \hat{F}(X)/(\log X)^2$ and we find the global upper bound $(\theta_1+\theta_2)\text{MT}$ where
\begin{equation}
\vert \theta_1 \vert \le 4 \frac{B}{\log X} c^3(1 + e^{c\Delta})(1+ C_0(c))(4\delta)^{\nu_1}\nu_1^{\nu_2}
\end{equation}
and
\begin{equation}
\vert \theta_2 \vert \le \frac{1}{2^{\nu_1+\nu_2-2}}\frac{(\nu_1+\nu_2-1)!}{(\nu_1-1)!(\nu_2-1)!}\frac{V_{\sigma_0}}{V_{\sigma}}\frac{A'}{\log X}.
\end{equation}
And using Lemma \ref{lemme9},
\begin{equation*}
\frac{(\nu_1+\nu_2-1)!}{(\nu_1-1)!(\nu_2-1)!}\le \frac{\sqrt{\nu_1}\sqrt{\nu_2}}{\sqrt{\nu_1+\nu_2}}e^{\frac{1}{24}} \frac{(\nu_1+\nu_2)^{\nu_1+\nu_2}}{\nu_1^{\nu_1}\nu_2^{\nu_2}},
\end{equation*}
we use $\frac{V_{\sigma_0}}{V_{\sigma}}\le 2c^2$ and we get:
\begin{equation}
\vert \theta_2 \vert \le 4 \frac{\sqrt{\nu_1}\sqrt{\nu_2}}{\sqrt{\nu_1+\nu_2}}e^{\frac{1}{24}} \frac{(\nu_1+\nu_2)^{\nu_1+\nu_2}}{(2\nu_1)^{\nu_1}(2\nu_2)^{\nu_2}} 2c^2 \frac{A'}{\log X}.
\end{equation}
We write $\nu_1= x(\nu_1+\nu_2)$, and so we need to find an upper boud for $\exp(h(x))$ where
\begin{multline*}
h(x) = \frac{1}{2}\log x + \frac{1}{2}\log(1-x) - x(\nu_1+\nu_2)\log x - (1- x)(\nu_1+\nu_2)\log(1 - x)\\
= \left(\frac{1}{2} - x(\nu_1+\nu_2)\right)\log x +  \left(\frac{1}{2} -(1- x)(\nu_1+\nu_2)\right)\log (1-x)
\end{multline*}
whose derivative is
\begin{equation*}
h'(x) = (\nu_1+\nu_2)(\log(1-x) -\log(x)) + \frac{1}{2}\left(\frac{1}{x} - \frac{1}{1-x}\right).
\end{equation*}
The maximum of $h$ is reached for $x= \frac{1}{2}$ so $\nu_1=\nu_2$. We finally get
\begin{equation}
\vert \theta_2 \vert \le 4 c^2 \sqrt{\nu_1+\nu_2}e^{\frac{1}{24}} \frac{A'}{\log X} \le 5 c^2 \sqrt{\nu_1+\nu_2}\frac{A'}{\log X} \le 0.01A'.
\end{equation}

\subsection{Final result}
Our final result, see \ref{hypo} is then\\

For all $1\le \nu_1\le\nu_2$,
\begin{multline*}
\sum_{n\le X} \Lambda^{(\nu_1+\nu_2)}(n) f(n) + \sum_{n\le X} \Lambda^{(\nu_1)}\star \Lambda^{(\nu_2)}(n) f(n)\\
= \frac{V_{\sigma}(z)}{V_{\sigma_0}(z)}\left(\sum_{n\le X} \Lambda^{(\nu_1+\nu_2)}(n) f_0(n) + \sum_{n\le X} \Lambda^{(\nu_1)}\star \Lambda^{(\nu_2)}(n) f_0(n)\right)\\
 + (\rho +\theta) \frac{V_{\sigma}(z)}{V_{\sigma_0}(z)} \hat{F}(X) \frac{(\log X)^{\nu_1+\nu_2 - 1}}{(\nu_1+\nu_2 -1)!}
\end{multline*}
where
\begin{equation*}
\vert \rho \vert \le \frac{3}{2} \times (149\nu_2)^{\nu_1+\nu_2} 
 \times \biggl(Ac^2+\left( C_0(c) e \delta^{3\nu_2} +\Delta4^{\frac{1}{\delta}} \right) \left(\frac{c}{\delta}\right)^{2 \nu_2} \biggr)
\end{equation*}
and we get an upper bound for $\vert \theta\vert$ by adding the contribution of the two first remainder terms for the preliminary sieve (i.e. $\rho$) and we take together the other contribution of the sieve (i.e. $\tilde{\rho}$) and the one from $\theta_1$ and $\theta_2$ which came when removing the sieve. In the second term, the last line, $2c^2(1+e^{c\Delta})\left(1+C_0(c)\right)$ can be considered as bounded by a constant and $\text{Cste}\times (4\delta)^{\nu_2}\nu_2^{\nu_1} \le (24\nu_2)^{\nu_2} \left(3\nu_2^2\log \frac{1}{\delta}\right)^{\nu_1+2\nu_2}\delta^{\nu_1}$. So
\begin{align}\label{theta}
 \vert\theta\vert &\le 2(24\nu_2)^{\nu_2}B_0 \left(3\nu_2^2\log \frac{1}{\delta}\right)^{\nu_1+2\nu_2}\delta^{\nu_1} 
+ 2\frac{B}{\log X} c^2(1+e^{c\Delta})\left(1+C_0(c)\right)(4\delta)^{\nu_2}\nu_2^{\nu_1} + A'\\
\notag&\le 2 \times 648^{\nu_2} \nu_2^{2\nu_1+5\nu_2} (\frac{B}{\log X}+B_0) \delta^{\nu_1}\left(\log \frac{1}{\delta}\right)^{\nu_1+2\nu_2}+A'.
\end{align}


\section{A step towards the twin prime conjecture}

\subsection{Conjecture}\label{conje}
We work under the following conjecture:
\begin{conj}
For all $\delta$ and for all $\theta$, there exists a constant $A_0(\delta, \theta)$ such that
\begin{equation}
\sum_{\substack{d\le X^{1-\delta} \\ (d, 2) =1}}\max_{y\le X}\left| \psi(y;d,2) -\frac{y}{\varphi(d)}\right| \le A_0(\delta, \theta)\frac{X}{(6\log X)^{\theta}}.
\end{equation}
\end{conj}

It is a conjecture about the distribution of prime numbers in arithmetic progressions close to Elliott-Halberstam conjecture (but over only one congruence class).
This hypothesis is similar to the one used by Zhang in \cite{Zhang}:  
\begin{equation*}
\sum_{\substack{q\le X^{0.5+2\omega}\\ p|q => p\le X^{\omega}}} | \psi(X;q,a)-X/\varphi(q) | \ll X/(\log X)^c
\end{equation*}
with $\omega = 1/1168$ and for any integer $a$ such that
$p|a \Rightarrow p \le X^\omega$.

\subsection{Using theorem \ref{thm2} with $f(n) = \Lambda(n+2)$}
Theorem \ref{thm2} can be applied to twin prime numbers when taking $f(n) = \Lambda(n+2)$ and $f_0(n) = 1$. 

First, we calculate the asymptotic sums implying $f_0$.  
\subsubsection{Reference asymptotics}
We need the following definitions and properties:
\begin{defi}
For all $\kappa \ge 1$,
\begin{equation}
\mathscr{F}_{\kappa}(X)= \int_1^X \log^{\kappa-1} t dt.
\end{equation}
\end{defi}
Then we have
\begin{lem}\label{lemme26}
\begin{equation*}
\mathscr{F}_{\kappa}(X) = X \sum_{1\le k \le \kappa} (-1)^{k-1} \frac{(\kappa-1)!}{(\kappa-k)!}\log^{\kappa-k} X + (-1)^{\kappa}(\kappa -1)!
\end{equation*}
\end{lem}
See \cite{Ram}, Lemma 26.
And thus
\begin{equation}
\mathscr{F}_{\kappa}(X) \asymp X \log^{\kappa-1} X.
\end{equation}
We will also use the following results (from \cite{Ram}, Lemmas 27 and 28):
\begin{lem}\label{lemme27}
There exists a positive constant $c$ such that, for $X \ge 3$ and $X \ge \mathfrak{f} \ge 1$,
\begin{equation*}
\sume_{n\le X}\Lambda^{(\kappa)}(n) = \frac{\mathscr{F}_{\kappa}(X)}{(\kappa-1)!}\left(1 + \mathcal{O}(\mathcal{L}^c)\right)
\end{equation*}
\end{lem}
where $\mathcal{L}=\exp\left(-\frac{\log^{3/5} X}{\log \log^{1/5}X}\right)$.

\begin{lem}\label{lemme28}
For any integers $\nu_1$ and $\nu_2$, and $X \ge 1$, we have
\begin{equation*}
\int_1^X \frac{\mathscr{F}_{\nu_1}(X/t)\mathscr{F}'_{\nu_2}(t)}{(\nu_1-1)!(\nu_2-1)!}dt = \frac{\mathscr{F}_{\nu_1+\nu_2}(X)}{(\nu_1+\nu_2-1)!}.
\end{equation*}
\end{lem}

From which we deduce the following lemma:
\begin{lem}\label{lem29}
There exists a positive constane $c$ such that, for $X \ge 3$ and $X\ge \mathfrak{f} \ge 1$ we have
\begin{equation}
\sume_{n\le X}\Lambda^{(\nu_1+\nu_2)}(n)+\sume_{n\le X}\Lambda^{(\nu_1)}\star\Lambda^{(\nu_2)}(n)=\frac{2\mathscr{F}_{\nu_1+\nu_2}(X)}{(\nu_1+\nu_2-1)!} \left(1 + \mathcal{O}(\mathcal{L}^c)\right).
\end{equation}

\end{lem}
\begin{proof}
Lemma \ref{lemme27} gives us the first term of the left hand side sum. For the second term (the convolution product) we use the Dirichlet hyperbola method: $\sum_{\ell \le X}\sum_{m\le X/\ell}=\sum_{\substack{\ell \le \sqrt{X}\\ m\le X/\ell}}+ \sum_{\substack{m \le \sqrt{X}\\ \ell \le X/m}} - \sum_{\substack{\ell \le \sqrt{X}\\ m\le \sqrt{X}}}$.

Thus, we calculate:
\begin{equation}
\sume_{\ell\le \sqrt{X}}\sume_{m\le X/\ell}\Lambda^{(\nu_1)}(\ell)\Lambda^{(\nu_2)}(m).
\end{equation}
From Lemma \ref{lemme27}, there exists a constant $c_1$ such that the last term is equal to
\begin{equation*}
\sume_{\ell\le \sqrt{X}} \Lambda^{(\nu_1)}(\ell) \frac{\mathscr{F}_{\nu_2}(\frac{X}{\ell})}{(\nu_2-1)!} \left(1 + \mathcal{O}(\mathcal{L}^{c_1})\right).
\end{equation*}
Using Lemma \ref{lemme26} and Lemma \ref{lemme9}, we find the remainder term is the size of
\begin{equation*}
X\frac{\log^{\nu_1+\nu_2} X}{(\nu_1-1)!(\nu_2-1)!}\mathcal{L}^{c_1} \ll \frac{\mathscr{F}_{\nu_1+\nu_2}(X)}{(\nu_1+\nu_2-1)!}\mathcal{L}^{c_2}
\end{equation*}
where $c_2$ is another constant.
We calculate the main term the following way:
\begin{align*}
 \sume_{\ell\le \sqrt{X}} \Lambda^{(\nu_1)}(\ell) \frac{\mathscr{F}_{\nu_2}(\frac{X}{\ell})}{(\nu_2-1)!} &= \sume_{\ell\le \sqrt{X}} \frac{\Lambda^{(\nu_1)}(\ell)}{(\nu_2-1)!}\int_1^{\frac{X}{\ell}}\mathscr{F}'_{\nu_2}(t) dt\\
&=\int_1^X \sume_{\ell \le \min(\frac{X}{t},\sqrt{X})}\frac{\Lambda^{(\nu_1)}(\ell)\mathscr{F}'_{\nu_2}(t)}{(\nu_2-1)!} dt\\
&= \int_1^X\frac{\mathscr{F}_{\nu_1}(\min(\frac{X}{t},\sqrt{X}))}{(\nu_1-1)!}\frac{\mathscr{F}'_{\nu_2}(t)}{(\nu_2-1)!}dt +\text{terme d'erreur}.
\end{align*}
Now
\begin{multline*}
\int_1^X\frac{\mathscr{F}_{\nu_1}(\min(\frac{X}{t},\sqrt{X}))}{(\nu_1-1)!}\frac{\mathscr{F}'_{\nu_2}(t)}{(\nu_2-1)!}dt =\\
 \int_1^{\sqrt{X}}\frac{\mathscr{F}_{\nu_1}(\sqrt{X})}{(\nu_1-1)!}\frac{\mathscr{F}'_{\nu_2}(t)}{(\nu_2-1)!}dt +
 \int_{\sqrt{X}}^X\frac{\mathscr{F}_{\nu_1}(\frac{X}{t})}{(\nu_1-1)!}\frac{\mathscr{F}'_{\nu_2}(t)}{(\nu_2-1)!}dt
\end{multline*}
where
\begin{equation*}
\int_1^{\sqrt{X}}\frac{\mathscr{F}_{\nu_1}(\sqrt{X})}{(\nu_1-1)!}\frac{\mathscr{F}'_{\nu_2}(t)}{(\nu_2-1)!}dt = \frac{\mathscr{F}_{\nu_1}(\sqrt{X})\mathscr{F}_{\nu_2}(\sqrt{X})}{(\nu_1-1)!(\nu_2-1)!}.
\end{equation*}
We use again the Dirichlet hyperbola method, unsing the symmetry between $\ell$ and $m$ we find:
\begin{multline*}
\sume_{n\le X}\Lambda^{(\nu_1)}\star\Lambda^{(\nu_2)}(n) = \int_{\sqrt{X}}^X\frac{\mathscr{F}_{\nu_1}(\frac{X}{t})}{(\nu_1-1)!}\frac{\mathscr{F}'_{\nu_2}(t)}{(\nu_2-1)!}dt 
+ \int_{\sqrt{X}}^X\frac{\mathscr{F}_{\nu_2}(\frac{X}{t})}{(\nu_2-1)!}\frac{\mathscr{F}'_{\nu_1}(t)}{(\nu_1-1)!}dt \\
+  \frac{\mathscr{F}_{\nu_1}(\sqrt{X})\mathscr{F}_{\nu_2}(\sqrt{X})}{(\nu_1-1)!(\nu_2-1)!} +\mathcal{O}\left(\frac{\mathscr{F}_{\nu_1+\nu_2}(X)}{(\nu_1+\nu_2-1)!}\mathcal{L}^{c_2}\right).
\end{multline*}
We have used Lemma \ref{lemme27} for the subtracted term 
\begin{equation*}
\sume_{\substack{\ell \le \sqrt{X}\\ m\le \sqrt{X}}}\Lambda^{(\nu_1)}(\ell)\Lambda^{(\nu_2)}(m) = \frac{\mathscr{F}_{\nu_1}(\sqrt{X})\mathscr{F}_{\nu_2}(\sqrt{X})}{(\nu_1-1)!(\nu_2-1)!}(1 + \mathcal{O}(\mathcal{L}^{c_1})).
\end{equation*}

Using integration by parts and then the substitution $u = X/t$, we find that
\begin{equation}
\int_{\sqrt{X}}^X\frac{\mathscr{F}_{\nu_1}(\frac{X}{t})}{(\nu_1-1)!}\frac{\mathscr{F}'_{\nu_2}(t)}{(\nu_2-1)!}dt  = \int_1^{\sqrt{X}}\frac{\mathscr{F}'_{\nu_1}(u)}{(\nu_1-1)!}\frac{\mathscr{F}_{\nu_2}(\frac{X}{u})}{(\nu_2-1)!}du 
\end{equation}
and so the main term is equal to
\begin{equation}
\int_1^X \frac{\mathscr{F}_{\nu_2}(\frac{X}{t})\mathscr{F}'_{\nu_1}(t)}{(\nu_1-1)!(\nu_2-1)!}dt = \frac{\mathscr{F}_{\nu_1+\nu_2}(X)}{(\nu_1+\nu_2-1)!}
\end{equation}
from Lemma \ref{lemme28}.
\end{proof}

\subsubsection{Using theorem \ref{thm2}}
We apply the thoerem taking the following parameters: 
\begin{multline*}
\mathfrak{f} = 2,\quad \sigma(d)= \frac{\mathbbm{1}_{(d,2)=1}}{\varphi(d)}, \quad \sigma_0(d) = \frac{1}{d}, \quad
 f_0(n) = 1,  \quad  \hat{F}(X) = X\\
\end{multline*}
and $F(y) = y$. Note that for $f(n)= \Lambda(n+2)$ and $f_0(n)=1$, we have $B=\log X$ and $B_0=1$.

The conjecture \ref{conje} used with $\theta = 5(\nu_1+ \nu_2)^2+5$ gives us the following hypothesis: for all $\delta$  there exists a constant $A_0$ depending only on $\delta$, such that
\begin{equation}\label{H10}\tag{$H_{10}$}
\sum_{\substack{d\le X^{1-\delta} \\ (d, 2) =1}}\max_{y\le X}\left| \psi(y;d,2) -\frac{y}{\varphi(d)}\right| \le \frac{A_0}{\log^2 X} \frac{X}{(6\log X)^{5(\nu_1+ \nu_2)^2+3}}.
\end{equation}
And we define $A''= A_0/\log^2 X$.


\begin{cor}\label{cor}
There exists a constant $c>0$ such that for any $X$ large enough, under the conjecture \ref{conje}, for all $\delta> 100/\sqrt{\log X} $, and for all $\nu_1$ and $\nu_2$, such that $\nu_1\le \nu_2 \le (\log X)^{1/10}$ and $\nu_2^2\delta \log \left(\frac{1}{\delta}\right)\le \frac{1}{8}$, we have
\begin{multline*}
\sum_{n\le X} \Lambda(n+2)(\Lambda^{(\nu_1+\nu_2)}+\Lambda^{(\nu_1)}\star \Lambda^{(\nu_2)})(n)\\
= \frac{2 \mathfrak{S}_2 \mathscr{F}_{\nu_1+\nu_2}(X)}{(\nu_1+\nu_2-1)!}(1+\mathcal{O}(\mathcal{L}^c+ \sqrt{A''} +\nu_2^{2\nu_1+5\nu_2}\delta^{\frac{\nu_1}{2}}))
\end{multline*}
where $ \mathfrak{S}_2 = \prod_{3\le p} \frac{p(p-2)}{(p-1)^2}$ is the twin prim numbers constant ($\mathfrak{S}_2=0,660...$).

In a more simple way,
\begin{multline*}
\sum_{n\le X} \Lambda(n+2)(\Lambda^{(\nu_1+\nu_2)}+\Lambda^{(\nu_1)}\star \Lambda^{(\nu_2)})(n)\\
= \frac{2 \mathfrak{S}_2 \mathscr{F}_{\nu_1+\nu_2}(X)}{(\nu_1+\nu_2-1)!}\left(1+\mathcal{O}\biggl(\frac{1}{\log X} +\nu_2^{2\nu_1+5\nu_2}\delta^{\frac{\nu_1}{2}}\biggr)\right)
\end{multline*}
as by definition $\mathcal{L}^c$ and $ \sqrt{A''}$ are both some $\mathcal{O}\left(1/\log X\right)$.
\end{cor}
\begin{proof}
We want to use Theorem \ref{thm2} which is
\begin{multline*}
\sum_{n\le X}\Lambda(n+2)\left( \Lambda^{(\nu_1+\nu_2)} + \Lambda^{(\nu_1)}\star \Lambda^{(\nu_2)}\right)(n)\\
= \frac{V_{\sigma}(z)}{V_{\sigma_0}(z)}\sum_{n\le X}\left( \Lambda^{(\nu_1+\nu_2)} + \Lambda^{(\nu_1)}\star \Lambda^{(\nu_2)}\right)(n)
 + (\rho +\theta) \frac{V_{\sigma}(z)}{V_{\sigma_0}(z)} X \frac{(\log X)^{\nu_1+\nu_2 - 1}}{(\nu_1+\nu_2 -1)!}
\end{multline*}
where
\begin{equation}
\frac{V_{\sigma}(z)}{V_{\sigma_0}(z)}= \prod_{3\le p\le z} \frac{1- \frac{1}{p-1}}{1-\frac{1}{p}} = \prod_{3\le p\le z} \frac{p(p-2)}{(p-1)^2} = \mathfrak{S}_2(1+ \mathcal{O}(\frac{1}{z}))= \mathfrak{S}_2(1 + \mathcal{O}(X^{-\delta})).
\end{equation}
So we check all the needed hypothesis.

In particular, hypothesis $(H_9)$ follows on one side from the Brun-Titchmarsh theorem (see \cite{Mont} for instance) with $C=\mathcal{O}(\log(\log X))$ and on the other side from hypothesis $(H_{10})$.
Indeed, this last hypothesis gives us an upper bound for $R(f, D_0,1)$ where $D_0 = X^{1-\delta}$ and $f(n)= \Lambda(n+2)$. And with the same notations as in Part 1, we also get the upper bound for $R(1, X^{1-\delta},1)$ as $f_0(n) = 1$, and knowing that $r_d(1,y)$ is the fractional part of $y/d$ as $\sigma_0(d) = 1/d$ and $F(y)=y$ and $\sum_{n\le y/d}1 = \frac{1}{d} \times y + r_d(1,y)$.
So $R(1, X^{1-\delta},1)\le X^{1-\delta}\le \frac{X}{(\log X)^{\theta}}$ for all $\theta$. At last we notice that $V_{\sigma}(z)/V_{\sigma_0}(z)=\mathfrak{S}_2(1 + \mathcal{O}(X^{-\delta}))$ from the calculation just above. 

So we need $A' $ such that
\begin{equation}
A'\ge A''(6\log X)^{-5(\nu_1+\nu_2)^2-1} + \mathfrak{S}_2 X^{-\delta}\log^2X .
\end{equation}

Furthermore, from the previous lemma,
\begin{equation}
\sum_{n\le X}\left( \Lambda^{(\nu_1+\nu_2)} + \Lambda^{(\nu_1)}\star \Lambda^{(\nu_2)}\right)(n) = \frac{2\mathscr{F}_{\nu_1+\nu_2}(X)}{(\nu_1+\nu_2-1)!} \left(1 + \mathcal{O}(\mathcal{L}^c))\right)
\end{equation}
and we notice that
\begin{equation*}
X \frac{(\log X)^{\nu_1+\nu_2 - 1}}{(\nu_1+\nu_2 -1)!} = \mathcal{O}\left(  \frac{\mathscr{F}_{\nu_1+\nu_2}(X)}{(\nu_1+\nu_2-1)!}\right).
\end{equation*}
At last, from Theorem \ref{thm2} and (\ref{theta}) we had
\begin{multline*}
\vert \rho\vert + \vert \theta \vert \ll \frac{3}{2} \times (149\nu_2)^{\nu_1+\nu_2} 
 \times \biggl(Ac^2+\left( C_0(c) e \delta^{3\nu_2} +\Delta4^{\frac{1}{\delta}} \right) \left(\frac{c}{\delta}\right)^{2 \nu_2} \biggr)\\
  2 \times 648^{\nu_2} \nu_2^{2\nu_1+5\nu_2} (\frac{B}{\log X}+B_0) \delta^{\nu_1}\left(\log \frac{1}{\delta}\right)^{\nu_1+2\nu_2}+A'
\end{multline*}
with
\begin{align*}
\Delta = \sum_{\substack{p^{a}\le X \\ p \ge z}}\left| \sigma(p^{a}) -\sigma_0(p^{a})\right| = 
\sum_{\substack{p^{a}\le X \\ p \ge z\\ (p, 2) = 1}}\left| \frac{1}{\varphi(p^{a})}- \frac{1}{p^{a}}\right| 
= \sum_{\substack{p^{a}\le X \\ p \ge z\\ (p, 2) = 1}}\frac{1}{p^{a}}\times \frac{1}{p-1}\\
\le \frac{1}{z} \sum_{n \le X} \frac{1}{n} \ll X^{-\delta}\log X.
\end{align*}
And here $ \frac{B}{\log X}+B_0= 2$.
Our aim is now to simplify this upper bound.

First, let us prove that $4^{\frac{1}{\delta}}\Delta$ is no more than $\delta^{3\nu_2}$ for a large enough $X$.

On one hand we have
\begin{equation*}
4^{\frac{1}{\delta}}\Delta \le 4^{\frac{1}{\delta}}\frac{\log X}{X^{\delta}} \qquad \text{et} \qquad \left(\frac{100}{\sqrt{\log X}}\right)^{3\nu_2} < \delta^{3\nu_2},
\end{equation*}
taking the logarithm and having $\delta\ge \frac{100}{\sqrt{\log X}}$, we find that
\begin{equation*}
4^{\frac{1}{\delta}}\frac{\log X}{X^{\delta}} \le \left(\frac{100}{\sqrt{\log X}}\right)^{3\nu_2}
\end{equation*}
is equivalent to 
\begin{equation*}
\log (\log X) \left(1+ \frac{3}{2}\nu_2\right) \le 3\nu_2\log 100 + \left(100 - \frac{1}{25}\right) \sqrt{\log X}.
\end{equation*}

On the other hand, $A'$ is obviously (much) lower than $\frac{1}{2}(149\nu_2)^{\nu_1+\nu_2} Ac^2$ so we bound both contributions from $A$ and $A'$ with $2(149\nu_2)^{\nu_1+\nu_2} Ac^2$.

Taking $A = c_1(149\nu_2)^{-\nu_1-\nu_2}\sqrt{A''}$, and using the fact that $\log(\log X)\le \log X$ in the constant $C$, we get
\begin{equation*}
A'' \ll (149\nu_2)^{2(\nu_1+\nu_2)}c^4(2 \log X)^{4\nu_2^2 +1}A' \ll (6\log X)^{5\nu_2^2 +1} A'
\end{equation*}
assuming $4 \le \nu_2 \le (\log X)^{\frac{1}{10}}$ (and still $\nu_1 \le \nu_2$).

At last, for a small enough $\delta$ we have
\begin{equation*}
648^{\nu_2} \nu_2^{2\nu_1+5\nu_2} \delta^{\nu_1}\left(\log \frac{1}{\delta}\right)^{\nu_1+2\nu_2} \ll \nu_2^{2\nu_1+5\nu_2}\delta^{\frac{\nu_1}{2}}.
\end{equation*}
\end{proof}

\subsection{Consequence: A theorem on twins almost prime}
\subsubsection{Simplifying the asymptotic}
In this section, we denote
\begin{equation*}
\Lambda_{\nu_1,\nu_2}(n) = (\Lambda^{(\nu_1+\nu_2)}+\Lambda^{(\nu_1)}\star \Lambda^{(\nu_2)})(n)
\end{equation*}
and
\begin{equation*}
\tilde{\Lambda}_{\nu_1,\nu_2}(n) = \frac{\Lambda_{\nu_1,\nu_2}(n)}{(\log n)^{\nu_1+\nu_2-1}}.
\end{equation*}
We then notice that
\begin{equation}
\tilde{\Lambda}_{\nu_1,\nu_2}(n) = \frac{\Lambda(n)}{(\nu_1+\nu_2-1)!} + \frac{\sum_{d_1 d_2 =n}\Lambda(d_1)(\log d_1)^{\nu_1-1}\Lambda(d_2)(\log d_2)^{\nu_2-1}}{(\nu_1-1)!(\nu_2-1)!(\log n)^{\nu_1+\nu_2-1}}.
\end{equation}

And so Corollary \ref{cor} can be written
\begin{equation*}
\sum_{n\le X} \Lambda(n+2)\Lambda_{\nu_1,\nu_2}(n)= \frac{2 \mathfrak{S}_2 \mathscr{F}_{\nu_1+\nu_2}(X)}{(\nu_1+\nu_2-1)!}\left(1+\mathcal{O}\biggl(\frac{1}{\log X} +\nu_2^{2\nu_1+5\nu_2}\delta^{\frac{\nu_1}{2}}\biggr)\right).
\end{equation*}
Here, we work under the conditions
\begin{equation*}
\delta > 100/\sqrt{\log X} \qquad \nu_1, \nu_2 \le (\frac{1}{2}\log X)^{1/10} \qquad  \nu_2^2\delta \log \left(\frac{1}{\delta}\right)\le \frac{1}{8}.
\end{equation*}

We look into 
\begin{equation}
\sum_{n\le X} \Lambda(n+2)\tilde{\Lambda}_{\nu_1,\nu_2}(n) = \sum_{n\le X} \Lambda(n+2)\frac{\Lambda_{\nu_1,\nu_2}(n)}{(\log n)^{\nu_1+\nu_2-1}},
\end{equation}
but
\begin{align*}
\sum_{n\le X} \Lambda(n+2)\frac{\Lambda_{\nu_1,\nu_2}(n)}{(\log n)^{\nu_1+\nu_2-1}} &= \sum_{n\le X} \Lambda(n+2)\Lambda_{\nu_1,\nu_2}(n) \left( \int_n^X -\frac{(\nu_1+\nu_2-1)dt}{t(\log t)^{\nu_1+\nu_2}} + \frac{1}{(\log X)^{\nu_1+\nu_2-1}} \right)\\
&=  \frac{1}{(\log X)^{\nu_1+\nu_2-1}}\sum_{n\le X} \Lambda(n+2)\Lambda_{\nu_1,\nu_2}(n) \\
&- \int_2^X\sum_{n\le t}\Lambda(n+2)\frac{(\nu_1+\nu_2-1)\Lambda_{\nu_1,\nu_2}(n)}{t(\log t)^{\nu_1+\nu_2}}dt.
\end{align*}
This last integral will turn out to be an error term. We give an estimate of it using the fact that $\int_2^X = \int_2^{\sqrt{X}} + \int_{\sqrt{X}}^X$ and use Corollary \ref{cor} for the second term as the condition $\nu_1,\nu_2 \le \min_{\sqrt{X} \le t\le X} (\log t)^{1/10} = (\frac{1}{2}\log X)^{1/10}$ is fulfilled.

Hence
\begin{multline*}
\int_{\sqrt{X}}^X \sum_{n\le t}\Lambda(n+2)\frac{(\nu_1+\nu_2-1)\Lambda_{\nu_1,\nu_2}(n)}{t(\log t)^{\nu_1+\nu_2}}dt  \\
= \int_{\sqrt{X}}^X  \frac{2 \mathfrak{S}_2 \mathscr{F}_{\nu_1+\nu_2}(t)}{(\nu_1+\nu_2-1)! t(\log t)^{\nu_1+\nu_2}}\left(1+\mathcal{O}\biggl(\frac{1}{\log X} +\nu_2^{2\nu_1+5\nu_2}\delta^{\frac{\nu_1}{2}}\biggr)\right)dt.
\end{multline*}
From Lemma \ref{lemme26},
\begin{equation}
\mathscr{F}_{\nu_1+\nu_2}(t) \sim t(\log t)^{\nu_1 +\nu_2-1}.
\end{equation}
More precisely, $\mathscr{F}_{\nu_1+\nu_2}(t)$ being an alternating series we even have
\begin{equation*}
\mathscr{F}_{\nu_1+\nu_2}(t) \le t(\log t)^{\nu_1 +\nu_2-1} + (\nu_1+\nu_2-1)!
\end{equation*}
So this integral is a $\mathcal{O}\left(\frac{X}{\log X}\right)$.

However,
\begin{multline*}
\int_2^{\sqrt{X}} \sum_{n\le t}\Lambda(n+2)\frac{(\nu_1+\nu_2-1)\Lambda_{\nu_1,\nu_2}(n)}{t(\log t)^{\nu_1+\nu_2}}dt  \\
\le \int_2^{\sqrt{X}} \log (t+2)\sum_{n\le t}\frac{(\nu_1+\nu_2-1)\Lambda_{\nu_1,\nu_2}(n)}{t(\log t)^{\nu_1+\nu_2}}dt 
\end{multline*}
and from Lemma \ref{lem29},
\begin{equation}
\sum_{n \le t}\Lambda_{\nu_1,\nu_2}(n) =\frac{2\mathscr{F}_{\nu_1+\nu_2}(t)}{(\nu_1+\nu_2-1)!} \left(1 + \mathcal{O}(\mathcal{L}^c))\right)
\end{equation}
with again $\mathscr{F}_{\nu_1+\nu_2}(t) \sim t(\log t)^{\nu_1 +\nu_2-1}$, this second integral is a $\mathcal{O}(\sqrt{X})$.

So now we can write
\begin{equation}
\sum_{n\le X} \Lambda(n+2)\frac{\Lambda_{\nu_1,\nu_2}(n)}{(\log n)^{\nu_1+\nu_2-1}} =  \frac{1}{(\log X)^{\nu_1+\nu_2-1}}\sum_{n\le X} \Lambda(n+2)\Lambda_{\nu_1,\nu_2}(n) + \mathcal{O}\left(\frac{X}{\log X}\right).
\end{equation}
To give an estimate of the main term, we use again Corollary \ref{cor}:
\begin{equation*}
\sum_{n\le X} \Lambda(n+2)\Lambda_{\nu_1,\nu_2}(n)= \frac{2 \mathfrak{S}_2 \mathscr{F}_{\nu_1+\nu_2}(X)}{(\nu_1+\nu_2-1)!}\left(1+\mathcal{O}\biggl(\frac{1}{\log X} +\nu_2^{2\nu_1+5\nu_2}\delta^{\frac{\nu_1}{2}}\biggr)\right)
\end{equation*}
and as $\mathscr{F}_{\nu_1+\nu_2}(X) \sim X(\log X)^{\nu_1 +\nu_2-1}$ and more precisely, from Lemma \ref{lemme26}
\begin{multline*}
\mathscr{F}_{\nu_1+\nu_2}(X) = X(\log X)^{\nu_1 +\nu_2-1} - (\nu_1 +\nu_2-1)\mathscr{F}_{\nu_1+\nu_2-1}(X) 
\\= X(\log X)^{\nu_1 +\nu_2-1} + \mathcal{O}(X (\log X)^{\nu_1+\nu_2-2}),
\end{multline*}
we get:
\begin{align*}
&\qquad \qquad\qquad\sum_{n\le X} \Lambda(n+2)\tilde{\Lambda}_{\nu_1,\nu_2}(n) \\
&=  \frac{2 \mathfrak{S}_2 \left( X(\log X)^{\nu_1 +\nu_2-1} + \mathcal{O}(X (\log X)^{\nu_1+\nu_2-2})\right)}{(\nu_1+\nu_2-1)!(\log X)^{\nu_1+\nu_2-1}}\left(1+\mathcal{O}\biggl(\frac{1}{\log X} +\nu_2^{2\nu_1+5\nu_2}\delta^{\frac{\nu_1}{2}}\biggr)\right)+ \mathcal{O}\left(\frac{X}{\log X}\right)\\
 &= \frac{2 \mathfrak{S}_2 X}{(\nu_1+\nu_2-1)!}\left(1+\mathcal{O}\biggl(\frac{1}{\log X} +\nu_2^{2\nu_1+5\nu_2}\delta^{\frac{\nu_1}{2}}\biggr)\right)
+ \mathcal{O}\left(\frac{X}{\log X}\right)\\
 &= \frac{2 \mathfrak{S}_2 X}{(\nu_1+\nu_2-1)!}\left(1+\mathcal{O}\biggl(\frac{1}{\log X} +\nu_2^{2\nu_1+5\nu_2}\delta^{\frac{\nu_1}{2}}\biggr)\right)\\
\end{align*}
i.e.
\begin{equation}
 \frac{2 \mathfrak{S}_2 X}{(\nu_1+\nu_2-1)!}\left(1+\mathcal{O}\biggl(\frac{1}{\log X} +\nu_2^{2\nu_1+5\nu_2}\delta^{\frac{\nu_1}{2}}\biggr)\right).
\end{equation}
Now 
 $\nu_2^{2\nu_1+5\nu_2}\delta^{\frac{\nu_1}{2}} = \mathcal{O}((\nu_1+\nu_2)^{5(\nu_1+\nu_2)}\delta^{\frac{1}{2}})$.
So our result can be stated this way:
\begin{thm}\label{thm3}
\begin{multline*}
\sum_{n\le X}  \frac{\Lambda(n+2)\Lambda(n)}{(\nu_1+\nu_2-1)!} + \sum_{n \le X} \Lambda(n+2)\frac{(\Lambda^{(\nu_1)}\star \Lambda^{(\nu_2)})(n)}{(\log n)^{\nu_1+\nu_2-1}} \\
=  \frac{2 \mathfrak{S}_2 X}{(\nu_1+\nu_2-1)!}\left( 1+\mathcal{O}( (\nu_1+\nu_2)^{5(\nu_1+\nu_2)}\delta^{\frac{1}{2}} +
\frac{1}{\log X})\right).
\end{multline*}
\end{thm}


Multiplying both memebrs of the quality by $(\nu_1+\nu_2 -1)!$, this theorem becomes:
\begin{thm}\label{thm3bis}
For all $\nu_1,\nu_2\ge 1$,
\begin{multline*}
\sum_{n\le X}\Lambda(n+2)\Lambda(n) + \sum_{n \le X}\Lambda(n+2)\sum_{d_1 d_2 =n} \frac{\Lambda(d_1)\Lambda(d_2)}{\log n} f_{\nu_1,\nu_2}\left(\frac{\log d_1}{\log n}\right) \\
= 2 \mathfrak{S}_2 X\left( 1+\mathcal{O}( (\nu_1+\nu_2)^{5(\nu_1+\nu_2)}\delta^{\frac{1}{2}} +
\frac{1}{\log X})\right)
\end{multline*}

where
\begin{equation}
f_{\nu_1,\nu_2}(x) = (\nu_1+\nu_2-1) \binom{\nu_1+\nu_2 -2}{\nu_1-1} x^{\nu_1-1}(1-x)^{\nu_2-1}.
\end{equation}
\end{thm}

Indeed, the convolution product can be written
\begin{multline*}
\sum_{d_1 d_2 =n} \frac{\Lambda(d_1)(\log d_1)^{\nu_1-1}}{(\nu_1-1)!}\frac{\Lambda(d_2)(\log d_2)^{\nu_2-1}}{(\nu_2-1)!}\frac{(\nu_1+\nu_2-1)!}{(\log n)^{\nu_1+\nu_2 -1}}\\
= \sum_{d_1 d_2 =n}(\nu_1+\nu_2-1) \binom{\nu_1+\nu_2 -2}{\nu_1-1} \left(\frac{\log d_1}{\log n}\right)^{\nu_1-1}\left(\frac{\log n - \log d_1}{\log n}\right)^{\nu_2-1} \frac{1}{\log n}.
\end{multline*}

We immediately deduce from Theorem \ref{thm3bis} the well known following result:
\begin{lem}\label{lemme}
\begin{equation*}
\sum_{n\le X}\Lambda(n+2)\Lambda(n) = \mathcal{O}(X)
\end{equation*}
and
\begin{equation*}
\sum_{n \le X}\Lambda(n+2)\sum_{d_1 d_2 =n} \frac{\Lambda(d_1)\Lambda(d_2)}{\log n} 
= \mathcal{O}(X).
\end{equation*}
\end{lem}
\begin{proof}
From the previous theorem, the sum of those two positive terms being a $\mathcal{O}(X)$, we have $\sum_{n\le X}\Lambda(n+2)\Lambda(n) = \mathcal{O}(X)$ and for all $\nu_1$, $\nu_2$, $\sum_{n \le X}\Lambda(n+2)\sum_{d_1 d_2 =n} \frac{\Lambda(d_1)\Lambda(d_2)}{\log n}  f_{\nu_1,\nu_2}\left(\frac{\log d_1}{\log n}\right)
= \mathcal{O}(X)$.
We simply take $f_{1,2}$ and $f_{2,1}$. By definition, we have $f_{1,2}+f_{2,1}=2$ ($=\nu_1+\nu_2-1$).
\end{proof}

\subsection{Bernstein polynomials' appearance}
It is noticeable that
\begin{equation}
f_{\nu_1,\nu_2}(x) = (\nu_1+\nu_2-1) b_{\nu_1-1,\nu_1+\nu_2-2}(x)
\end{equation}
where $b_{k,m}$ is the Bernstein polyniomial: 
\begin{equation}
b_{k,m}= \binom{m}{k}x^k (1-x)^{m-k}.
\end{equation}
The most useful result on Bernstein polynomials for our goal (c.f. \cite{Phil} Theorem 7.1.5 ) is the following:
\begin{thm}\label{bernst}
Given a function $f$ $\in C[0,1]$, and any $\varepsilon>0$ there exists an integers $N$ such that 
\begin{equation*}
\vert (B_m(f;.))_{m \in \mathbb{N}} - f \vert \le \varepsilon
\end{equation*}
for all $n \ge N$, where
\begin{equation*}
 B_m(f; x) = \sum_{k=0}^m f\left(\frac{k}{m}\right) b_{k,m}(x) = \sum_{k=0}^m f\left(\frac{k}{m}\right) \binom{m}{k}x^k (1-x)^{m-k}.
\end{equation*}
\end{thm}

Another remarkable result is that for all $k$,
\begin{equation}
\int_0^1 b_{k,m}(t)dt = \frac{1}{m+1}.
\end{equation}
(c.f. \cite{Phil})
And so
\begin{equation}
\int_0^1 f_{\nu_1,\nu_2}(t)dt = 1.
\end{equation}
We define $\nu' = \nu_1+\nu_2 - 2$ and so $f_{\nu_1,\nu_2}(x) = f_{\nu_1,\nu'- \nu_1 +2}(x)= (\nu'+1) b_{\nu_1-1,\nu'}(x)$.

For all function $F \in C[0;1]$, we define
\begin{equation*}
F_{\nu'}(x) = \sum_{\nu_1=1}^{\nu'+1} \alpha_{\nu_1,\nu'+1}f_{\nu_1,\nu'- \nu_1 +2}(x)
\end{equation*}
where
\begin{equation*}
\alpha_{\nu_1,\nu'+1} = \frac{F\left(\frac{\nu_1-1}{\nu'}\right)}{\nu'+1}.
\end{equation*}
Thus, from Theorem \ref{bernst}, $(F_{\nu'})$ converges uniformly to $F$ on $[0,1]$.

\subsection{From Theorem \ref{thm3bis} to a theorem on twins almost primes}

\subsubsection{Definition of a "plateau" function and "peak" functions}\label{pics}
We want an infinitely differentiable function whose value in $0$ is $0$ and is $1$ in $1$ and which increases very quickly from  $0$ to $1$ (and whose derivative is zero in $0$ and $1$). For example, let $f_m$ be the function defined over $[0,1]$ by $f_m(x) = (4x(1-x))^m$ and $F_m(x) = \frac{1}{a_m}\int_0^x f_m(t) dt$ where 
$a_m = \int_0^1 f_m(t) dt$. (Note $a_m= 4^m\sum_{k=0}^m \frac{(-1)^k}{k+m+1}\binom{m}{k}$). The greater $m$ is, the quicker $F_m$ increases from $0$ to $1$, but the precise value of $m$ is not relevant for our resulst (we suppose $m\ge 10$ for example).

We now create a "plateau" function $F$ which is zero for all $x \le \beta$ and for all $x \ge \gamma$ and a non-zero constant on a interval close and inside of  $[\beta, \gamma]$ the following way:

First, let $h$ be a function such that
\begin{align*}
\left\lbrace
\begin{array}{l}
h(x) = 0 \qquad \qquad \qquad \quad  \text{on } [0, \beta] \\
h(x) =F_m \left(\frac{x-\beta}{\epsilon}\right)\qquad \quad \text{on } [ \beta, \beta + \epsilon]\\
h(x) = 1 \qquad  \qquad \qquad \quad \text{on } [\beta+\epsilon, \gamma-\epsilon]\\
h(x) = F_m \left(\frac{\gamma-x}{\epsilon}\right) \qquad \quad \text{on } [\gamma-\epsilon,\gamma]\\
h(x) = 0 \qquad \qquad \qquad \quad  \text{on } [\gamma,1]\\
\end{array}\right.
\end{align*}

We denote $\mathcal{A}_{\beta ,\gamma} = \int_0^1 h(t) dt$ and we define $F(x) = \frac{h(x)}{\mathcal{A}_{\beta ,\gamma}}$.
From the construction of $h$ we have,
\begin{equation*}
\mathcal{A}_{\beta ,\gamma} = \gamma-\beta +\mathcal{O}(\epsilon)
\end{equation*}
and so
\begin{equation}\label{Agamma}
\frac{1}{\mathcal{A}_{\beta ,\gamma} }= \frac{1}{\gamma-\beta} +\mathcal{O}(\epsilon).
\end{equation}

And thus the integral of $F$ over $[0,1]$ is exactly $1$ (see below an example of a "plateau" function).\\

\begin{center}
\includegraphics[scale=0.6]{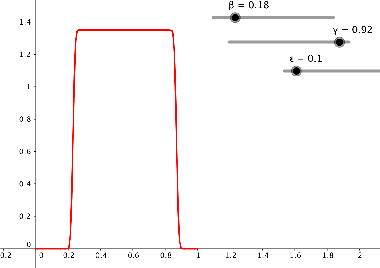}
\end{center}

We then define a "peak in $\beta$" function the following way:
\begin{align*}
\left\lbrace
\begin{array}{l}
F_{\beta}(x) = 0 \qquad \qquad \qquad \qquad \text{on } [0, \beta-\epsilon] \\
F_{\beta}(x) = \frac{1}{\mathcal{A}_{\beta ,\gamma}}F_m \left(\frac{x-\beta+ \epsilon}{\epsilon}\right)\quad \text{on } [ \beta-\epsilon, \beta]\\
F_{\beta}(x) = \frac{1}{\mathcal{A}_{\beta ,\gamma}}F_m \left(\frac{\beta+\epsilon-x}{\epsilon}\right) \quad \text{on } [\beta,\beta +\epsilon]\\
F_{\beta}(x) = 0 \qquad \qquad \qquad \qquad  \text{on } [\beta+\epsilon,1].\\
\end{array}\right.
\end{align*}
The function $F_\gamma$ ("peak in $\gamma$") is defined the same way. Those functions are "peaks" the same height as the "plateau" function defined above. All these functions are differentiables over $[0,1]$. See below an example of "peak in $\beta$" function. \\

\begin{center}
\includegraphics[scale=0.6]{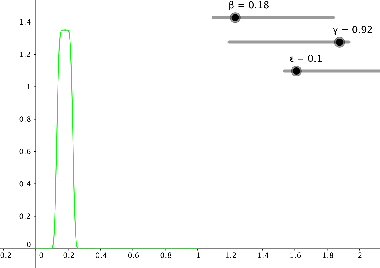}
\end{center}


\subsubsection{Theorem for a continuous function whose integral is $1$}

Let $F$ be a continuous function on $[0;1]$ such that $F(t) = 0$ for $ 0\le t \le \beta$ and for $ \gamma\le t \le 1$, where $0\le\beta<\gamma<1$, and such that
$\int_0^1 F(t)dt=1$.

We take $\varepsilon >0$. From what precedes, there exists an iteger $N$ such that for all $\nu' \ge N$ we have $ \vert F - F_{\nu'}\vert < \varepsilon$.

Furthermore, 
in Theorem \ref{thm2}, for all $\nu_1 \ge 1$, we have   $\vert \rho + \theta \vert < \varepsilon$.

Let us take $\nu'> N$,
so $F = F_{\nu'} + \mathcal{O}(\varepsilon)$. And from Theorem \ref{thm3bis},

\begin{align*}
\sum_{\nu_1=1}^{\nu'+1}\alpha_{\nu_1,\nu'+1} \sum_{n\le X} \Lambda(n)\Lambda(n+2) + \sum_{n \le X}{\Lambda(n+2)}\sum_{d_1 d_2 =n} \frac{\Lambda(d_1)\Lambda(d_2)}{\log n}\sum_{\nu_1=1}^{\nu'+1}\alpha_{\nu_1,\nu'+1} f_{\nu_1,\nu'-\nu_1+2}\left(\frac{\log d_1}{\log n}\right)\\
= \sum_{\nu_1=1}^{\nu'+1}\alpha_{\nu_1,\nu'+1} \sum_{n\le X} \Lambda(n)\Lambda(n+2) + \sum_{n \le X}{\Lambda(n+2)}\sum_{d_1 d_2 =n} \frac{\Lambda(d_1)\Lambda(d_2)}{\log n}F_{\nu'}\left(\frac{\log d_1}{\log n}\right)\\
= \sum_{\nu_1=1}^{\nu'+1}\alpha_{\nu_1,\nu'+1} 2 \mathfrak{S}_2 X\left( 1+\mathcal{O}((\nu_1+\nu_2)^{5(\nu_1+\nu_2)}\delta^{\frac{1}{2}}+
\frac{1}{\log X})\right).
\end{align*}

Now,
\begin{equation}
\sum_{\nu_1=1}^{\nu'+1}\alpha_{\nu_1,\nu'+1} = \sum_{\nu_1=1}^{\nu'+1}\frac{F\left(\frac{\nu_1-1}{\nu'}\right)}{\nu'+1}
\end{equation}
where
\begin{equation*}
\left|  \sum_{\nu_1=1}^{\nu'+1}\frac{F\left(\frac{\nu_1-1}{\nu'}\right)}{\nu'} - \int_0^1 F(t)dt \right| \le \frac{\sup_{t \in [0,1]}\vert F'(t)\vert}{\nu'}= \mathcal{O}\left(\frac{1}{\nu'}\right).
\end{equation*}
If we considered a function of bounded variations.

Note: when we take the "plateau" function defined above, the right hand memebr will be $\frac{1}{a_m\mathcal{A}_{\beta ,\gamma} \epsilon\nu'}$ because $F'(x) = \frac{1}{\mathcal{A}_{\beta ,\gamma} \epsilon}F_m'(x)$ and the maximum of $ F_m'(x)$ is $F_m'(1/2) = 1/a_m$.

So
\begin{equation}
\sum_{\nu_1=1}^{\nu'+1}\alpha_{\nu_1,\nu'+1} = \left( \frac{\nu'}{\nu'+1}\right)\left( \int_0^1 F(t)dt + \mathcal{O}(\frac{1}{\nu'})\right)
\end{equation}
and
\begin{equation}
\sum_{\nu_1=1}^{\nu'+1}\alpha_{\nu_1,\nu'+1} = 1 + \mathcal{O}(\frac{1}{\nu'}).
\end{equation}
 
And on the other hand, from Lemma \ref{lemme},
\begin{equation*}
\sum_{n\le X} {\Lambda(n+2)}\sum_{d_1 d_2 = n}\frac{\Lambda(d_1)\Lambda(d_2)}{\log n}= \mathcal{O}( X).
\end{equation*}
We want to approximate
\begin{align*}
&\sum_{n\le X} \Lambda(n)\Lambda(n+2) + \sum_{n \le X}{\Lambda(n+2)}\sum_{d_1 d_2 =n} \frac{\Lambda(d_1)\Lambda(d_2)}{\log n}F\left(\frac{\log d_1}{\log n}\right)\\
&= \sum_{n\le X} \Lambda(n)\Lambda(n+2) + \sum_{n \le X}{\Lambda(n+2)}\sum_{\substack{1\le d_1 \le n \\ d_1 \mid n}} \frac{\Lambda(d_1)\Lambda\left(\frac{n}{d_1}\right)}{\log n}F\left(\frac{\log d_1}{\log n}\right)\\
\end{align*}
where $F = F_{\nu'} + \mathcal{O}(\varepsilon)$, so the second term is equal to
\begin{align*}
 \sum_{n \le X}{\Lambda(n+2)}\sum_{\substack{1\le d_1 \le n \\ d_1 \mid n}} \frac{\Lambda(d_1)\Lambda\left(\frac{n}{d_1}\right)}{\log n}F_{\nu'}\left(\frac{\log d_1}{\log n}\right)
+ \sum_{n\le X} {\Lambda(n+2)}\sum_{d_1 d_2 = n}\frac{\Lambda(d_1)\Lambda(d_2)}{\log n}\mathcal{O}(\varepsilon).
\end{align*}
We get
\begin{align*}
&\sum_{\nu_1=1}^{\nu'+1}\alpha_{\nu_1,\nu'+1}\sum_{n\le X} \Lambda(n)\Lambda(n+2)+ \sum_{n \le X}{\Lambda(n+2)}\sum_{\substack{1\le d_1 \le n \\ d_1 \mid n}} \frac{\Lambda(d_1)\Lambda\left(\frac{n}{d_1}\right)}{\log n}F_{\nu'}\left(\frac{\log d_1}{\log n}\right)
+ \mathcal{O} (\varepsilon X)\\
&+ (1- \sum_{\nu_1=1}^{\nu'+1}\alpha_{\nu_1,\nu'+1})\sum_{n\le X} \Lambda(n)\Lambda(n+2).
\end{align*}
And from Theorem \ref{thm3bis}, this expression is equal to
\begin{align*}
&\sum_{\nu_1=1}^{\nu'+1}\alpha_{\nu_1,\nu'+1} 2 \mathfrak{S}_2 X\left( 1+\mathcal{O}( (\nu_1+\nu_2)^{5(\nu_1+\nu_2)}\delta^{\frac{1}{2}} 
+\frac{1}{\log X})\right)+\mathcal{O}(\varepsilon X) \\
&+ (1- \sum_{\nu_1=1}^{\nu'+1}\alpha_{\nu_1,\nu'+1})\sum_{n\le X} \Lambda(n)\Lambda(n+2)
\end{align*}
i.e.
\begin{align*}
& 2 \mathfrak{S}_2 X\left( 1+\mathcal{O}( (\nu_1+\nu_2)^{5(\nu_1+\nu_2)}\delta^{\frac{1}{2}} +\frac{1}{\log X})\right)\left(1 +  \mathcal{O}(\frac{1}{\nu'})\right)\\
&+\mathcal{O}(\varepsilon X) + \mathcal{O}( \frac{1}{\nu'})\sum_{n\le X} \Lambda(n)\Lambda(n+2).
\end{align*}
From Lemma \ref{lemme}
\begin{equation}
\sum_{n\le X} \Lambda(n)\Lambda(n+2) = \mathcal{O}(X)
\end{equation}
And $\nu' = \nu_1+\nu_2 -2$ so finally our theorem is
\begin{thm}
Given a continuous function $F$ with bounded variation on $[0;1]$ such that $F(t) = 0$ for $ 0\le t \le \beta$ and for $ \gamma\le t \le 1$, where $0\le \beta<\gamma<1$, satisfying hypothesis (\ref{H10}) and such that 
$\int_0^1 F(t)dt=1$.

For all $\varepsilon >0$ there exists an integer $N'$such that for all $\nu_1+\nu_2 \ge N'+2$,
\begin{align*}
&\sum_{n\le X} \Lambda(n)\Lambda(n+2) + \sum_{n \le X}{\Lambda(n+2)}\sum_{\substack{ d_1 d_2 =n \\ n^{\beta} \le d_1 \le n^{\gamma}}} \frac{\Lambda(d_1)\Lambda(d_2)}{\log n}F\left(\frac{\log d_1}{\log n}\right)\\
&= 2 \mathfrak{S}_2 X\left( 1+\mathcal{O}( \frac{1}{\nu_1+\nu_2-2}+ \varepsilon+(\nu_1+\nu_2)^{5(\nu_1+\nu_2)}\delta^{\frac{1}{2}}
+\frac{1}{\log X})\right).
\end{align*}
\end{thm}
And we deduce the following corollary
\begin{cor}\label{cor2}
Given a continuous function $F$ with bounded variation on $[0;1]$ such that $F(t) = 0$ for $ 0\le t \le \beta$ and for $ \gamma\le t \le 1$, where $0\le \beta<\gamma<1$, satisfying hypothesis (\ref{H10}) and such that 
$\int_0^1 F(t)dt=\mathcal{A}$.

For all $\varepsilon >0$ there exists an integer $N'$ such that for all $\nu_1+\nu_2 \ge N'+2$,
\begin{align*}
&\mathcal{A}\sum_{n\le X} \Lambda(n)\Lambda(n+2) + \sum_{n \le X}{\Lambda(n+2)}\sum_{\substack{ d_1 d_2 =n \\ n^{\beta} \le d_1 \le n^{\gamma}}} \frac{\Lambda(d_1)\Lambda(d_2)}{\log n}F\left(\frac{\log d_1}{\log n}\right)\\
&= 2 \mathcal{A} \mathfrak{S}_2 X\left( 1+\mathcal{O}( \frac{1}{\nu_1+\nu_2-2}+ \varepsilon+(\nu_1+\nu_2)^{5(\nu_1+\nu_2)}\delta^{\frac{1}{2}}
+\frac{1}{\log X})\right).
\end{align*}
\end{cor}
In the following calculations, when using this corollary  
we will still use the notation $\sum_{\nu_1=1}^{\nu'+1}\alpha_{\nu_1,\nu'+1} = \mathcal{A}(1 + \mathcal{O}(\frac{1}{\nu'}))$.

\subsubsection{Towards the final asymptotic}

We now consider the "almost indicator function" of the interval $[\beta, \gamma]$, that we denote by $1^*_{[\beta,\gamma]}$, which is zero on $[0,\beta]$ and on $[\gamma,1]$ and constantly equal to$\frac{1}{\mathcal{A}_{\beta ,\gamma}}$ on $[\beta, \gamma]$. We then have $\vert 1_{[\beta,\gamma]} - F\vert \le F_{\beta} + F_{\gamma}$ where $F$,  $ F_{\beta}$ and $ F_{\gamma}$ are the "pleateau" and "peak" functions defined previously.

We use Theorem \ref{thm3bis} and Corollary \ref{cor2} with the functions $F$ and  $G= F_{\beta} + F_{\gamma}$ the following way:

We denote 
\begin{equation*}
\alpha'_{\nu_1,\nu'+1} = \frac{G\left(\frac{\nu_1-1}{\nu'}\right)}{\nu'+1}
\end{equation*}
and the corresponding $G_{\nu'}$ similar to $F_{\nu'}$. We notice that
\begin{align*}
\sum_{\nu_1=1}^{\nu'+1}\alpha'_{\nu_1,\nu'+1}\le \frac{\nu'}{\nu'+1} \left(\int_0^1 (F_{\beta}+ F_{\gamma})(t)dt  + \frac{\sup_{t \in [0,1]}\vert (F_{\beta}+F_{\gamma})'(t)\vert}{\nu'}\right)\\
= \frac{\nu'}{\nu'+1} \left( \frac{4\epsilon}{\mathcal{A}_{\beta ,\gamma}} +  \mathcal{O}(\frac{1}{\nu'})\right) =\mathcal{O}(\epsilon+\frac{1}{\nu'}).
\end{align*}
Theorem \ref{thm3bis} will be used with the function $G_{\nu'}$ so we have:
\begin{align}
\sum_{\nu_1=1}^{\nu'+1}\alpha'_{\nu_1,\nu'+1} \sum_{n\le X} \Lambda(n)\Lambda(n+2) + \sum_{n \le X}{\Lambda(n+2)}\sum_{d_1 d_2 =n} \frac{\Lambda(d_1)\Lambda(d_2)}{\log n}G_{\nu'}\left(\frac{\log d_1}{\log n}\right)\\
\notag = \sum_{\nu_1=1}^{\nu'+1}\alpha'_{\nu_1,\nu'+1} 2 \mathfrak{S}_2 X\left( 1+\mathcal{O}( (\nu_1+\nu_2)^{5(\nu_1+\nu_2)}\delta^{\frac{1}{2}} +
\frac{1}{\log X})\right).
\end{align}
And, recalling our work with $F_{\nu}$, 
we will have in a similar way $G = G_{\nu'}+\varepsilon$ from some positive integer $N'$,and
\begin{equation}
1_{[\beta,\gamma]} = F + G + \mathcal{O}(\epsilon)
\end{equation}
so:
\begin{align*}
&\sum_{n\le X} \Lambda(n)\Lambda(n+2) + \sum_{n \le X}{\Lambda(n+2)}\sum_{d_1 d_2 =n} \frac{\Lambda(d_1)\Lambda(d_2)}{\log n}1_{[\beta,\gamma]}\left(\frac{\log d_1}{\log n}\right) \\
& = \sum_{n\le X} \Lambda(n)\Lambda(n+2) + \sum_{n \le X}{\Lambda(n+2)}\sum_{d_1 d_2 =n} \frac{\Lambda(d_1)\Lambda(d_2)}{\log n}(F + G+ \mathcal{O}(\epsilon))\left(\frac{\log d_1}{\log n}\right)\\
\end{align*} 
which is equal to
 \begin{align*}
 &= \sum_{\nu_1=1}^{\nu'+1}\alpha_{\nu_1,\nu'+1}\sum_{n\le X} \Lambda(n)\Lambda(n+2)+ \sum_{n \le X}{\Lambda(n+2)}\sum_{\substack{1\le d_1 \le n \\ d_1 \mid n}} \frac{\Lambda(d_1)\Lambda\left(\frac{n}{d_1}\right)}{\log n}F_{\nu'}\left(\frac{\log d_1}{\log n}\right)\\
&\qquad \qquad+ \mathcal{O} (\varepsilon X)
+ (1- \sum_{\nu_1=1}^{\nu'+1}\alpha_{\nu_1,\nu'+1})\sum_{n\le X} \Lambda(n)\Lambda(n+2)\\
&+ \sum_{n \le X}{\Lambda(n+2)}\sum_{d_1 d_2 =n} \frac{\Lambda(d_1)\Lambda(d_2)}{\log n}G_{\nu'}\left(\frac{\log d_1}{\log n}\right) 
+\mathcal{O} (\varepsilon X)+ \mathcal{O}(\epsilon X).\\
\end{align*}
We introduce the same way the coefficients $\alpha'_{\nu_1,\nu'+1} $, the two last terms become:
\begin{align*}
 & \sum_{\nu_1=1}^{\nu'+1}\alpha'_{\nu_1,\nu'+1}\sum_{n\le X} \Lambda(n)\Lambda(n+2)+ \sum_{n \le X}{\Lambda(n+2)}\sum_{\substack{1\le d_1 \le n \\ d_1 \mid n}} \frac{\Lambda(d_1)\Lambda\left(\frac{n}{d_1}\right)}{\log n}G_{\nu'}\left(\frac{\log d_1}{\log n}\right)\\
 &+ (1- \sum_{\nu_1=1}^{\nu'+1}\alpha_{\nu_1,\nu'+1} -  \sum_{\nu_1=1}^{\nu'+1}\alpha'_{\nu_1,\nu'+1} )\sum_{n\le X} \Lambda(n)\Lambda(n+2)+\mathcal{O} \left((\varepsilon + \epsilon)X\right). 
\end{align*}
But 
\begin{equation}
(1- \sum_{\nu_1=1}^{\nu'+1}\alpha_{\nu_1,\nu'+1} -  \sum_{\nu_1=1}^{\nu'+1}\alpha'_{\nu_1,\nu'+1} )\sum_{n\le X} \Lambda(n)\Lambda(n+2)
= \mathcal{O}\left((\epsilon+\frac{1}{\nu'})X\right).
\end{equation}

And applying Theorem \ref{thm3bis} and Corollary \ref{cor2} with the functions $F_{\nu'}$ and $G_{\nu'}$ we get:
\begin{multline*}
\sum_{n\le X} \Lambda(n)\Lambda(n+2) + \sum_{n \le X}{\Lambda(n+2)}\sum_{d_1 d_2 =n} \frac{\Lambda(d_1)\Lambda(d_2)}{\log n}1_{[\beta,\gamma]}\left(\frac{\log d_1}{\log n}\right)\\
= 2 \mathfrak{S}_2 X\left( 1+\mathcal{O}( \frac{1}{\nu_1+\nu_2-2}+\epsilon+ \varepsilon+(\nu_1+\nu_2)^{5(\nu_1+\nu_2)}\delta^{\frac{1}{2}}
+\frac{1}{\log X})\right).
\end{multline*}
And so, by definition of the function $1_{[\beta,\gamma]}$, we get
\begin{align}\label{final}
\notag \sum_{n\le X} \Lambda(n)\Lambda(n+2) +\frac{1}{\mathcal{A}_{\beta ,\gamma}} \sum_{n \le X}\Lambda(n+2)\sum_{\substack{ d_1 d_2 =n \\ n^{\beta} \le d_1 \le n^{\gamma}}} \frac{\Lambda(d_1)\Lambda(d_2)}{\log n} \\
= 2 \mathfrak{S}_2 X\left( 1+\mathcal{O}( \frac{1}{\nu_1+\nu_2-2}+\epsilon+ \varepsilon+(\nu_1+\nu_2)^{5(\nu_1+\nu_2)}\delta^{\frac{1}{2}}
+\frac{1}{\log X})\right).
\end{align}.

\subsubsection{Final Theorem}\label{thmfinal}
We set up the parameters for our result.

Let $\varepsilon>0$ be any real number.
We take $\epsilon=\varepsilon$ in the definitions of \ref{pics}. And we take $\nu= \max(N, \frac{1}{\varepsilon})$, where $N$ is such that
 for all $\nu'\ge N$, $\vert F- F_{\nu'}\vert<\varepsilon$ in Theorem (\ref{bernst}).

We take $\delta$ such that $(\nu_1+\nu_2)^{5(\nu_1+\nu_2)}\delta^{1/2} < \varepsilon$ where $\nu_1+\nu_2 -2 = \nu'>\nu$.
i.e.
\begin{equation}
\delta < \frac{\varepsilon^2}{\left(\nu+2\right)^{10\left(\nu+2\right)}}.
\end{equation}


If we choose $X_0= e^{\frac{10000}{\delta}}$, we have, for all $X\ge X_0$,
\begin{equation*}
\frac{1}{\log X}<\varepsilon \qquad\qquad \text{et} \qquad \qquad \delta >\frac{100}{\sqrt{\log X}}.
\end{equation*}.

Note: The conjecture \ref{conje} that became our hypothesis $(H_{10})$ used with $\delta$ above (and, as we recall, $\theta= 5(\nu_1+ \nu_2)^2+5$) gives us the existence of a constant $A_0$ depending only on $\varepsilon$ and from which we deduce $A''$ which is a $ \mathcal{O}(1/\log X)$.

Thus, in the identity (\ref{final}), the right and term is $2 \mathfrak{S}_2 X \left( 1 + \mathcal{O}(\varepsilon)\right)$.

Furthermore, from (\ref{Agamma}), we replace $\frac{1}{\mathcal{A}_{\beta ,\gamma}}$ with $\frac{1}{\gamma-\beta}+ \mathcal{O}(\varepsilon)$ and use Lemma \ref{lemme}: 
\begin{equation*}
\mathcal{O}(\varepsilon)\sum_{n \le X}\Lambda(n+2)\sum_{\substack{ d_1 d_2 =n \\ n^{\beta} \le d_1 \le n^{\gamma}}} \frac{\Lambda(d_1)\Lambda(d_2)}{\log n} 
= \mathcal{O}(\varepsilon X).
\end{equation*}.
 
So we get the following theorem

\begin{thm}\label{resultat}
Under the conjecture \ref{conje}, 
given any $\varepsilon>0$, and any $\beta\ge 0$ and $\gamma>\beta$, there exists a real number $X_0$ such that, for all $X\ge X_0$,
\begin{align*}
\sum_{n\le X} \Lambda(n)\Lambda(n+2) +\frac{1}{\gamma-\beta} \sum_{n \le X}\Lambda(n+2)\sum_{\substack{ d_1 d_2 =n \\ n^{\beta} \le d_1 \le n^{\gamma}}} \frac{\Lambda(d_1)\Lambda(d_2)}{\log n} \\
= 2 \mathfrak{S}_2 X\left( 1+\mathcal{O}(\varepsilon)\right).
\end{align*}
\end{thm}
In particular, if we take $\beta=0$ we get the following theorem
\begin{thm}
Under the conjecture \ref{conje}, given any $\varepsilon>0$, and any $\gamma$ where $0<\gamma<1$,  there exists a real number $X_0$ such that for all $X\ge X_0$,
\begin{align*}
\sum_{n\le X} \Lambda(n)\Lambda(n+2) +\frac{1}{\gamma} \sum_{n \le X}\Lambda(n+2)\sum_{\substack{ d_1 d_2 =n \\ d_1 \le n^{\gamma}}} \frac{\Lambda(d_1)\Lambda(d_2)}{\log n} \\
= 2 \mathfrak{S}_2 X\left( 1+\mathcal{O}( \varepsilon))\right).
\end{align*}
\end{thm}

We deduce the following corollary
\begin{cor}
Under the conjecture \ref{conje}, there exists an infinity of integers $n$ such that $n+2$ is prime and $n$ is the product of at most two prime powers, one of which being relatively small, i.e. lower than a small power of $n$ and we can get a precise asymptotic.
\end{cor}

\subsubsection*{A particular case: a theorem on twins amlost prime with non localized variables}

\begin{thm}
Under the conjecture \ref{conje}, given any $\varepsilon>0$, there exists a real number $X_0$such that for all  $X\ge X_0$,
\begin{multline*}
\sum_{n\le X}\Lambda(n+2)\Lambda(n) + \sum_{n \le X}\Lambda(n+2)\sum_{d_1 d_2 =n} \frac{\Lambda(d_1)\Lambda(d_2)}{\log n} \\
= 2 \mathfrak{S}_2 X\left( 1+\mathcal{O}( \varepsilon)\right).
\end{multline*}
\end{thm}
\begin{proof}

We use Theorem \ref{thm3bis} with $f_{1,2}$ and $f_{2,1}$, we add up the two asymptotics and simplify by $2$.
Using the identity
\begin{multline*}
\sum_{n\le X}\Lambda(n+2)\Lambda(n) + \sum_{n \le X}\Lambda(n+2)\sum_{d_1 d_2 =n} \frac{\Lambda(d_1)\Lambda(d_2)}{\log n} \\
= 2 \mathfrak{S}_2 X\left( 1+\mathcal{O}( 3^{15}\delta^{\frac{1}{2}} +
\frac{1}{\log X})\right)
\end{multline*}
with $\delta$ and $X_0$ chosen the same as before.
\end{proof}

We deduce the following corollary
\begin{cor}
Under the conjecture \ref{conje}, there exists an infinity of integers $n$ such that $n+2$ is a prime and $n$ is the product of at most two prime powers, and we can get a precise asymptotic.
\end{cor}


\subsection{Theorem without the prime powers}\label{thmsanspuis}
Our aim is to get a theorem as close as possible to the twim prime conjecure, so we want to "get rid of" the prime powers in our sum by estimating their contribution (suspecting that it will be negligible).

Let us take $\Lambda_1(n) = \log n$ if $n$ is a prime and $0$ otherwise. 
We want to prove that
\begin{align*}
\sum_{n\le X} \Lambda_1(n)\Lambda_1(n+2) +\frac{1}{\gamma-\beta} \sum_{n \le X}\Lambda_1(n+2)\sum_{\substack{ d_1 d_2 =n \\ n^{\beta} \le d_1 \le n^{\gamma}}} \frac{\Lambda_1(d_1)\Lambda_1(d_2)}{\log n} \\
= 2 \mathfrak{S}_2 X\left( 1+\mathcal{O}(\varepsilon)\right)
\end{align*}
We will prove below that indeed, the contribution of the prime powers will be a $\mathcal{O}(X/\log X)$.

Indeed, for the first term we have
\begin{multline*}
\sum_{n\le X} \Lambda_1(n)\Lambda_1(n+2)=
\sum_{n\le X} \Lambda(n)\Lambda(n+2) - \sum_{n\le X} (\Lambda-\Lambda_1)(n)\Lambda(n+2)\\
- \sum_{n\le X} \Lambda_1(n)(\Lambda-\Lambda_1)(n+2)
\end{multline*}
but
\begin{equation*}
\sum_{n\le X} (\Lambda-\Lambda_1)(n)\Lambda_1(n+2) \le \log (X+2)\sum_{n\le X} (\Lambda-\Lambda_1)(n)
\end{equation*}
and (c.f. \cite{Ten} for example)
\begin{equation*}
\sum_{n\le X} (\Lambda-\Lambda_1)(n) = \psi(X)-\theta(X)= \mathcal{O}(\sqrt{X})
\end{equation*}
($\psi$ and $\theta$ areTchebychev functions).
And also 
\begin{equation*}
\sum_{n\le X} \Lambda_1(n)(\Lambda-\Lambda_1)(n+2) \le \log (X)\sum_{n\le X} (\Lambda-\Lambda_1)(n+2)
\end{equation*}
We find that
\begin{equation}
\sum_{n\le X} \Lambda_1(n)\Lambda_1(n+2)=
\sum_{n\le X} \Lambda(n)\Lambda(n+2) + \mathcal{O}(\sqrt{X}\log X).
\end{equation}

For the second term, we use the following theorem
\begin{lem}\label{Halbrich}
Given $a$ a prime number or a prime power, 
the number of primes $p$ less than or equal to $X$ such that $a p +2$ is a prime is no more than $\mathfrak{C}\frac{X}{\log^2 X}$ where $\mathfrak{C}$ is a constant, independent of $X$ and $a$. (Consequently, this number is a $\mathcal{O}\left(\frac{X}{\log^2 X}\right)$).
\end{lem}
\begin{proof}
This is Theorem $3.12$ from \cite{Halb} in the case $k=l=1$ with $b=2$ and $a$ a prime power.

We notice that $\prod_{\substack{p>2\\p\mid 2 a}}\frac{p-1}{p-2}$ contains only one factor for all prime power $a$, so this product is $\le 2$. Thus the constant $\mathfrak{C}$ is indeed independent of $X$ and $a$.
\end{proof}
So, we have
\begin{align*}
 \sum_{n \le X}\Lambda_1(n+2)\sum_{\substack{ d_1 d_2 =n \\ n^{\beta} \le d_1 \le n^{\gamma}}} \frac{\Lambda_1(d_1)\Lambda_1(d_2)}{\log n} 
& = \sum_{n \le X}\Lambda(n+2)\sum_{\substack{ d_1 d_2 =n \\ n^{\beta} \le d_1 \le n^{\gamma}}} \frac{\Lambda(d_1)\Lambda(d_2)}{\log n}\\
&- \sum_{n \le X}\Lambda(n+2)\sum_{\substack{ d_1 d_2 =n \\ n^{\beta} \le d_1 \le n^{\gamma}}} \frac{\Lambda(d_1)\Lambda(d_2) -\Lambda_1(d_1)\Lambda_1(d_2)}{\log n}\\
&- \sum_{n \le X}(\Lambda-\Lambda_1)(n+2)\sum_{\substack{ d_1 d_2 =n \\ n^{\beta} \le d_1 \le n^{\gamma}}} \frac{\Lambda_1(d_1)\Lambda_1(d_2)}{\log n}\\
& = \sum_{n \le X}\Lambda(n+2)\sum_{\substack{ d_1 d_2 =n \\ n^{\beta} \le d_1 \le n^{\gamma}}} \frac{\Lambda(d_1)\Lambda(d_2)}{\log n}
+ T_1 + T_2.
\end{align*}
In order to bound $\vert T_2 \vert$, we use the fact that $ \Lambda_1(d_1)\le \log X^{\gamma}=\gamma \log X$ and $ \Lambda_1(d_2)\le \log X$, and, for any $n$, the number of summands over $d_1$ and $d_2$ is no more than twice the number of prime factors in $n$ (as $d_1$ and $d_2$ are necessarily prime).
But, from \cite{Erd}, the number of prime factors of $n$ is no more than
\begin{equation*}
\frac{\log n}{\log \log n}(1+ o(1))\le 2\frac{\log n}{\log \log n}.
\end{equation*}
So we find that
\begin{align*}
\vert T_2\vert \le \sum_{n\le X} \frac{4\log n}{\log \log n} \times \frac{\gamma \log^2 X }{\log n}(\Lambda-\Lambda_1)(n+2)
\end{align*}
We saw earlier than $\sum_{n\le X} (\Lambda-\Lambda_1)(n) = \mathcal{O}(\sqrt{X})$.
So
\begin{equation}
\vert T_2\vert \le \mathcal{O}(\sqrt{X}\log^2 X)
\end{equation}
When bounding $\vert T_1\vert $, we notice that
\begin{multline*}
\Lambda(d_1)\Lambda(d_2) -\Lambda_1(d_1)\Lambda_1(d_2) 
 = \Lambda(d_2)(\Lambda-\Lambda_1)(d_1) + \Lambda_1(d_1)(\Lambda-\Lambda_1)(d_2)
\end{multline*}
and calculate
\begin{equation}
\sum_{X<n\le 2X}\Lambda_1(n+2)\sum_{\substack{d_1 d_2 =n \\ n^{\beta} \le d_1 \le n^{\gamma}}} \frac{\Lambda(d_2)(\Lambda-\Lambda_1)(d_1)}{\log n}.
\end{equation}
This sum is less than or equal to twice the sum:
\begin{align*}
\sum_{X<n\le 2X}\Lambda_1(n+2)\sum_{\substack{d_1 d_2 =n \\  d_1\le d_2}} \frac{\Lambda(d_2)(\Lambda-\Lambda_1)(d_1)}{\log n}\\
\le \frac{1}{\log X}\sum_{d_1\le \sqrt{X}}(\Lambda-\Lambda_1)(d_1)\sum_{\substack{X<n\le 2X\\n=d_1 d_2}}\Lambda_1(n+2)\Lambda(d_2).
\end{align*}
We have
\begin{equation}
\Lambda(d_2) = \Lambda_1(d_2)+(\Lambda-\Lambda_1)(d_2).
\end{equation}
The contribution of the second term ($\Lambda-\Lambda_1$) being no more than
\begin{align*}
 \frac{1}{\log X}\sum_{d_1\le \sqrt{X}}(\Lambda-\Lambda_1)(d_1)\sum_{\frac{X}{d_1}<d_2\le \frac{2X}{d_1}}(\Lambda-\Lambda_1)(d_2)\log (2X+2)\\
\ll \frac{1}{\log X} \sum_{d_1\le \sqrt{X}}(\Lambda-\Lambda_1)(d_1)\sqrt{\frac{2X}{d_1}}\log (2X +2)\\
\ll  \frac{1}{\log X}\sqrt{2X}\log (2X+2)\sqrt{\sqrt{X}}= \mathcal{O}(X^{3/4}).
\end{align*}
So we estimate
\begin{align*}
\frac{1}{\log X}\sum_{d_1\le \sqrt{X}}(\Lambda-\Lambda_1)(d_1)\sum_{\substack{X<n\le 2X\\n=d_1 d_2}}\Lambda_1(n+2)\Lambda_1(d_2)
\end{align*}
which is no more than
\begin{align*}
\frac{1}{\log X}\sum_{d_1\le \sqrt{X}}(\Lambda-\Lambda_1)(d_1)\sum_{\substack{p\in \left[\frac{X}{d_1};\frac{2X}{d_1}\right]\\ p \text{ prime}\\ d_1 p +2 \text{ prime} }}\log(X+2)\log X\\
\le \frac{1}{\log X}\sum_{d_1\le \sqrt{X}}(\Lambda-\Lambda_1)(d_1)\log(X+2)\log X\times\mathcal{O}\left(\frac{X/d_1}{\log^2(X/d_1)}\right)
\end{align*}
from Lemma (\ref{Halbrich}) with $a=d_1$.
So our term is
\begin{align*}
\le \sum_{d_1\le \sqrt{X}}\frac{(\Lambda-\Lambda_1)(d_1)}{d_1}\mathfrak{C'}\frac{X}{\log X}
\end{align*}
where $\mathfrak{C'}$ is a constant independent of $d_1$.

Furthermore, the series $\sum \frac{(\Lambda-\Lambda_1)(d_1)}{d_1}$ is convergent so the partial sums are bounded. We set
\begin{equation*}
\mathfrak{C''}= \sum_{n\ge 1}\frac{(\Lambda-\Lambda_1)(n)}{n}\mathfrak{C'}
\end{equation*}
and so we have
\begin{equation}
\sum_{X<n\le 2X}\Lambda_1(n+2)\sum_{\substack{d_1 d_2 =n \\ n^{\beta} \le d_1 \le n^{\gamma}}} \frac{\Lambda(d_2)(\Lambda-\Lambda_1)(d_1)}{\log n}\le \mathfrak{C''}\frac{X}{\log X}
\end{equation}
where $\mathfrak{C''}$ is independent of $X$.

At last, we have
\begin{align*}
\sum_{n\le X}\Lambda_1(n+2)\sum_{\substack{d_1 d_2 =n \\ n^{\beta} \le d_1 \le n^{\gamma}}} \frac{\Lambda(d_2)(\Lambda-\Lambda_1)(d_1)}{\log n} = \sum_{\frac{X}{2}<n\le X}... +  \sum_{\frac{X}{4}<n\le \frac{X}{2}}... + ...\\
\le \mathfrak{C''}\frac{X}{2\log (X/2)}+\mathfrak{C''}\frac{X}{4\log (X/4)}+....=  \mathfrak{C''}X\sum_{k< \frac{\log X}{\log 2}}\frac{1}{2^k\log \frac{X}{2^k}}.
\end{align*}
Now, for all $k\le \frac{\log X}{2\log 2}$, $2^k\log(X/2^k)\ge 2^{k-1}\log X$ so we divide this last sum in two parts:

\begin{align*}
\sum_{k\le \frac{\log X}{\log 2}}\frac{1}{2^k\log \frac{X}{2^k}}&= \sum_{k< \frac{\log X}{2\log 2}}\frac{1}{2^k\log \frac{X}{2^k}}+
 \sum_{\frac{\log X}{2\log 2}\le k< \frac{\log X}{\log 2}}\frac{1}{2^k\log \frac{X}{2^k}}\\
&\le \frac{1}{\log X}\sum_{1\le k\le \frac{\log X}{2\log 2}}\frac{1}{2^{k-1}}+\frac{1}{\log \sqrt{X}}\sum_{ \frac{\log X}{2\log 2}\le k< \frac{\log X}{\log 2}}\frac{1}{2^{k}}
&\le  \frac{5}{\log X}.
\end{align*}
And so
\begin{equation}
\sum_{X<n\le 2X}\Lambda_1(n+2)\sum_{\substack{d_1 d_2 =n \\ n^{\beta} \le d_1 \le n^{\gamma}}} \frac{\Lambda(d_2)(\Lambda-\Lambda_1)(d_1)}{\log n}\le \frac{10 \mathfrak{C}'' X}{\log X}.
\end{equation}
We will get the same way
\begin{equation}
\sum_{X<n\le 2X}\Lambda_1(n+2)\sum_{\substack{d_1 d_2 =n \\ n^{\beta} \le d_1 \le n^{\gamma}}} \frac{\Lambda_1(d_1)(\Lambda-\Lambda_1)(d_2)}{\log n}\le \frac{10 \mathfrak{C}'' X}{\log X}
\end{equation}
And finally
\begin{equation}
\vert T_1\vert \le \frac{20 \mathfrak{C}'' X}{\log X} = \mathcal{O}\left(\frac{X}{\log X}\right)
\end{equation}
Thus, we can state a theorem without prime powers:
\begin{thm}\label{sanspuis}
Let us define the function $\Lambda_1(n) = \log n$ if $n$ is a prime and $0$ otherwise. Under the conjecture \ref{conje}, 
given any $\varepsilon>0$, any $\beta\ge 0$ and $\gamma>\beta$, there exists a real number $X_0$ such that for all $X\ge X_0$,
\begin{align*}
\sum_{n\le X} \Lambda_1(n)\Lambda_1(n+2) +\frac{1}{\gamma-\beta} \sum_{n \le X}\Lambda_1(n+2)\sum_{\substack{ d_1 d_2 =n \\ n^{\beta} \le d_1 \le n^{\gamma}}} \frac{\Lambda_1(d_1)\Lambda_1(d_2)}{\log n} \\
= 2 \mathfrak{S}_2 X\left( 1+\mathcal{O}(\varepsilon)\right).
\end{align*}
\end{thm}
And taking $\beta=0$ we get the following theorem:
\begin{thm}
Under the conjecture \ref{conje}, given any $\varepsilon>0$, and any $\gamma$ where $0<\gamma<1$, there exists a real number $X_0$ such that for all $X\ge X_0$,
\begin{align*}
\sum_{n\le X} \Lambda_1(n)\Lambda_1(n+2) +\frac{1}{\gamma} \sum_{n \le X}\Lambda_1(n+2)\sum_{\substack{ d_1 d_2 =n \\  d_1 \le n^{\gamma}}} \frac{\Lambda_1(d_1)\Lambda_1(d_2)}{\log n} \\
= 2 \mathfrak{S}_2 X\left( 1+\mathcal{O}(\varepsilon)\right).
\end{align*}
\end{thm}
We recall that for both those theorems, $X_0$ is chosen such that $\frac{1}{\log X}<\varepsilon$.





\bibliographystyle{plain}
\bibliography{biblio}

\end{document}